\documentclass[11pt, reqno]{amsart}

\usepackage{amsfonts,amssymb,amsmath,amsthm, ulem}
\usepackage[headinclude,DIV=13]{typearea}
\usepackage[utf8]{inputenc}
\usepackage{hyperref}
\usepackage{bigints}
\usepackage{mathrsfs,bbm,stmaryrd}
\usepackage{tikz,pgfplots}
\usetikzlibrary{positioning,patterns}
\usetikzlibrary{decorations.pathreplacing,calligraphy}
\usepackage{geometry}
\usepackage{algorithm}
\usepackage{algpseudocode}
\algrenewcommand\textproc{}
\usepackage{graphicx, psfrag,enumerate,enumitem}
\usepackage[justification=centering]{caption}
\usepackage{soul}
\usepackage{cancel}
\usepackage{color}
 \usepackage{setspace}
\usetikzlibrary{trees}
\usepackage{subfigure}
\usepackage[justification=centering]{caption}
\usepackage{array}

\newtheorem{theorem}{Theorem}[section]
\newtheorem{proposition}[theorem]{Proposition}
\newtheorem{dyn}[theorem]{Dynamics}

\newtheorem{corollary}[theorem]{Corollary}
\newtheorem{lemma}[theorem]{Lemma}

\newtheorem{assumption}{Assumption}
\newtheorem{remark}[theorem]{Remark}
\numberwithin{equation}{section}

\theoremstyle{definition} 
\theoremstyle{definition} 
\newcommand{\Zset}{\mathbb{V}}
\newcommand{\Zbarset}{\overline{\mathbb{V}}}
\newcommand{\Z}{\nu}
\newcommand{\Zbar}{\bar{\Z}}
\newcommand{\Zhat}{\chi}
\newcommand{\Zbarhat}{\bar{\Zhat}}
\newcommand{\Zhatset}{\mathbb{W}}
\newcommand{\Zbarhatset}{\overline{\Zhatset}}

\newcommand{\eg}{\textit{e.g. }}
\newcommand{\minus}{\scalebox{0.35}[1.0]{$-$}}
\newcommand{\correct}[1]{\textcolor{black}{#1}}
\algrenewcommand\algorithmicrequire{\textbf{Input:}}
\algrenewcommand\algorithmicensure{\textbf{Output:}}
\numberwithin{equation}{section}

\begin{document}
\renewcommand{\labelenumii}{(\roman{enumi}).\alph{enumii}}

\title[A spinal construction for general type-space populations with interactions]
{Spinal constructions for continuous type-space branching processes with interactions}

\author{Charles Medous}
\address{\parbox{\linewidth}{Charles Medous, Univ. Grenoble Alpes, CNRS, IF, 38000 Grenoble, France;}}
\address{\parbox{\linewidth}{Univ. Grenoble Alpes, INRIA, 38000 Grenoble, France}}
\email{charles.medous@univ-grenoble-alpes.fr}

\begin{abstract}
We consider branching processes describing structured, interacting  populations in continuous time. Dynamics of each individual's characteristics and branching properties can be influenced by the entire population. We propose a \correct{Girsanov-type result based on a spinal construction, and establish a many-to-one formula}. By combining this result with the spinal decomposition, we derive a \correct{generalized} continuous-time version of the Kesten-Stigum theorem that incorporates interactions. Additionally, we propose an alternative approach of the spine construction for exact simulations of stochastic size-dependent populations.
 \end{abstract}
\maketitle
 





\addtocontents{toc}{\setcounter{tocdepth}{-10}}

\section*{Introduction}
\label{section: intro}

Spine techniques and spinal trees are classical tools in the general context of branching processes since the work of  Kallenberg \cite{Kallenberg77}, Chauvin and Rouault \cite{Chauvin88,Chauvin91}, and later Kurtz, Lyons, Pemantle and Peres \cite{LPP95,Lyons97,KLPP97}. 
Spinal trees are constructed based on an original branching process by distinguishing a lineage, called the spine. Its only living representative- the spinal individual- follows a biased reproduction law compared to the other individuals in the process, ensuring that the spine does not die out. In the specific yet widely studied case of size-biased trees, the reproductive law $(\widehat{p}_n,n\geq 0)$ of the spinal individual is defined by
$$
\widehat{p}_n = \frac{np_n}{m}, \quad \textrm{for} \ n\geq 0,
$$
where $m$ is the mean value of the law of reproduction $(p_n,n\geq 0)$ in the branching process. This size-biased reproductive law of the spinal individual is closely related to the biased ancestral reproduction \cite{Chauvin91,Georgii03}. In fact the spine was found to characterize  the process of the trait of a uniformly sampled individual, in a large population approximation see \eg \cite{Marguet19}. The sampling of $k\geq 2$ distinct individuals from those living in a population at a time $t$ is associated with a $k$-spines construction. For further literature on multiple-particles sampling we refer the reader to \cite{harris2017many, harris2020coalescent,johnston2019coalescent,cheek2023ancestral} and the references herein.

Many-to-One formulas \cite{harris1996large, Georgii03} are prominent among classical spine results. Such formulas \correct{are derived} from a Girsanov-type result on the change of probability measure associated with the spine, which can be regarded as a Doob's transform, as described in \cite{chetrite2015nonequilibrium, Athreya00} for instance. These formulas give expectations of sums over particles in the branching process in terms of a Feynman-Kac path integral expectation related to the spinal individual. Consequently, the spinal individual is often referred to as a "typical individual" within the population. The connection to Feynman-Kac path integrals implies a shared foundation between these concepts. For a comprehensive overview on this subject, we refer to \cite{Delmoral04}. 

Another interesting property of spinal constructions is the "spinal decomposition" \cite{Chauvin88, LPP95}. It establishes that the spine process is equivalent to the initial branching process, with the removal of one individual and the addition of an immigration source. The introduction of new individuals into the population follows the biased reproductive law specific to the spine. In recent decades, the spinal decomposition has emerged as a highly valuable tool for investigating branching processes. One of its notable contributions is providing a new proof of the $L\log L$ criteria, which were originally proved by Kesten and Stigum for Galton-Watson (GW) processes \cite{KS66} and by Biggins for continuous-time branching processes \cite{Biggins77} using analytical methods. These results give specific conditions on the reproductive law, ensuring the non-degeneracy of the martingale involved in the spinal change of measure at infinity.
By combining the spinal decomposition with a previously known result from measure theory, Lyons, Pemantle, and Peres \cite{LPP95} provided a probabilistic "conceptual" proof of Kesten and Stigum's theorem for single-type GW processes. This method proved to be easily generalizable to continuous-time structured branching processes \cite{Athreya00,Georgii03,biggins2004measure} as well as branching Brownian motions \cite{kyprianou2004travelling,englander2007branching,git2007exponential}.
More recently Bertoin and Mallein extended this proof for general branching Levy processes \cite{bertoin2018biggins}. 
For equivalent results on superprocesses we refer to \cite{Liu09, Liu11, Liu13, eckhoff2015spines,ren2016spine}. 
Finally, we mention Hardy and Harris \cite{Hardy09} who adapted the spinal decomposition to prove the $\mathcal{L}^p$-convergence of some key martingales, which was later used to establish strong laws of large numbers \cite{englander2010strong}.

The fundamental assumption in the aforementioned works is the branching property, which assumes that the behavior of all particles in the process during their life is independent of one another. However, in various systems of population dynamics such as genetics, epidemiology, chemistry, and even queueing systems, interactions between individuals do occur and this fundamental hypothesis falls apart. 
Recently, Bansaye \cite{B21} established a spine construction and Many-to-One formulas for interacting branching populations, where the branching rates and reproductive laws depend on the traits of all individuals in the population. These traits belong to a finite set and are fixed during the life of the individuals. Using the spinal decomposition, Bansaye has found an $L\log L$ criterion for a single-type, density-dependent population. We also mention a recent work on spine processes for density-dependent individual-based models in a large population approximation \cite{henry2023time}.
 
In this article we consider a wide class of continuous-time structured branching processes with general interactions. These processes are used to model structured populations, where the behavior of each individual is influenced by the overall population state. Every individual in the population is characterized by its trait, taking values in a compact subset of $\mathbb{R}^d$. The lifespan of each individual is exponentially distributed with a time-inhomogeneous rate that depends on the traits of all individuals. Upon an individual's death, a random number of children are generated, each inheriting random traits at birth, that are influenced by the traits of all individuals in the population. Between these branching events, the evolution of the traits of all individuals in the population is deterministic and also influenced by the entire population's state, \correct{extending the work of Bansaye in \cite{B21} to include many physical models}. Notice that the branching parameters are determined by the traits of all individuals, thereby  the branching property no longer holds in this framework.
We introduce a comprehensive $\psi$-spinal construction for those processes, using a change of measure associated with a positive weight function $\psi$ that depends on the trait of the spinal individual and of those of every individual. For a fixed function $\psi$ both the spinal individual and those outside the spine are subject to a bias. 
We derive a Girsanov-type formula associated with this change of measure, taking the form of a path-integral formulation that involves a non-linear operator. A classical approach to establishing limiting results, such as the central limit theorem or large deviations, involves determining the eigenfunctions of such operators \cite{biggins2004measure,englander2007branching,Cloez17}. However, due to the presence of interactions, this operator is contingent on the entire population, necessitating the eigenfunctions to be dependent on the traits of all individuals. Thus, the weight function $\psi$ must rely on both the trait of the spinal individual and the traits of all individuals in order to account for this dependency. 
Under certain non-explosion assumptions regarding the branching parameters and the set of weight functions, we obtain a modified Many-to-One formula. Unlike the classical Many-to-One formula- that describes the behavior of the branching process using only the behavior of the spine- our formula relies on the whole spinal population. 

Subsequently, we use this result in conjunction with the spinal decomposition to establish a generalized $L\log L$ criteria for processes with interactions, \correct{generalizing the work of \cite{B21} for non constant functions $\psi$}. More precisely, we exhibit both a sufficient condition and a necessary condition for the non-degeneracy at infinity of the additive martingale associated with the $\psi$-spinal change of measure, for a large set of $\psi$ functions.

\correct{Finally we propose an alternative use of the spinal construction, where the function $\psi$ is not chosen as a an eigenfunction but to simplify the dynamics of the associated spine process.} As an example we study a particular case of structured Yule process with  mass loss events happening at size-dependent rates. Yule processes are pure birth processes that are widely used in population genetics to model and reconstruct phylogenetic trees, see \textit{e.g.} Aldous' review \cite{Aldous01}. We use the spinal construction with a multiplicative weight function to retrieve a conditional branching property in the associated spine process, and propose an efficient algorithmic construction based on this property.
\\
\textbf{Notation.}
In the sequel $\mathbb{N^*} = \{1,2,\cdots\}$ will denote the set of positive integers, $\mathbb{R}_+:=[0,+\infty)$ the real line, $\overline{\mathbb{R}}_+:= \mathbb{R}_+\cup\{+\infty\}$ and $\mathbb{R}_+^*:=(0,+\infty)$. We will denote respectively by $\mathfrak{B}\left(A,B\right)$ (resp. $\mathcal{C}^i\left(A,B\right)$) the set of measurable (resp. $i$ times continuously differentiable) $B$-valued functions on a set $A$. For every couple $(f,g)$ of real-valued measurable functions on a set $A$, we denote for all $x$, $fg(x)$ the product $f(x)g(x)$. 

The set of trait $\mathcal{X}$ is a \correct{subset} of $\mathbb{R}^d$ \correct{and for all $(x,y) \in \mathcal{X}^2$, we will denote $x\cdot y$ and $\vert x\vert$ respectively the canonical scalar product and the $\ell_1$-norm on $\mathbb{R}^d$}. 

\correct{We introduce the Ulam-Harris-Neveu notations \cite{Neveu86} to label the individuals of the population:}
$$\mathcal{U} := \left\{\emptyset\right\} \cup \bigcup_{k \geq 0}\left(\mathbb{N}^*\right)^{k+1}.$$
We consider branching processes starting from multiple initial individuals, thus the root $\emptyset$ will be treated as a phantom individual and its direct descendants will be the ancestor generation. 
For two elements $u,v$ of $\mathcal{U} \backslash \left\{\emptyset\right\}$, there exist two positive integers $n,p$ such that $u = (u_1, \ldots, u_n)$ and $v = (\correct{v_1},\ldots, v_p)$ and we write $uv:=(u_1,\ldots,u_n,v_1,\ldots,v_p)$ the concatenation of $u$ and $v$. We identify both $\emptyset u$ and $u \emptyset$ with $u$.
An individual $v \in \mathcal{U}$ is a descendant of $u$ if there exists $w \in \mathcal{U}$ such that $v=uw$. In this case we denote $u \preceq v$.

Let us introduce $\Zbarset$ the set composed of all finite point measures on $\mathcal{U}\times\mathcal{X}$
\begin{equation*}
\Zbarset := \left\{\sum_{i=1}^{N} \delta_{(u^i,x^i)},\  N \in \mathbb{N}, \ \left(u^i, 1 \leq i \leq N\right) \in  \mathcal{U}^{N}, \  \left(x^i, 1 \leq i \leq N\right) \in  \mathcal{X}^{N}   \right\}.
\end{equation*}
We also define the set of marginal population measures, that is
\begin{equation*}
\Zset := \left\{\sum_{i=1}^N \delta_{x^i},\ N \in \mathbb{N}, \ \left(x^i, 1 \leq i \leq N\right) \in  \mathcal{X}^{N} \right\}.
\end{equation*}
For any measure $\bar{\nu} = \sum_{i=1}^{N} \delta_{(u^i,x^i)}$ in $\Zbarset$, we will write $\nu := \sum_{i=1}^{N} \delta_{x^i}$ its projection on $\Zset$. By convention, if the number of points in the measure is $N=0$, $\bar{\nu}$ and $\nu$ are the trivial zero measures on $\mathcal{U}\times\mathcal{X}$ and $\mathcal{X}$.
We introduce for every $\bar{\nu} \in \Zbarset$, every $g \in \mathfrak{B}\left(\mathcal{U}\times \mathcal{X},\mathbb{R}\right)$ and every $f\in \mathfrak{B}\left(\mathcal{X},\mathbb{R}\right)$  
\begin{equation*}
\langle \bar{\nu}, g \rangle := \int_{\mathcal{U} \times \mathcal{X}}g(u,x)\bar{\nu}(\textrm{d}u,\textrm{d}x), \ \textrm{ and } \ \langle \nu, f \rangle := \int_{\mathcal{X}}f(x)\nu(\textrm{d}x).
\end{equation*}
Finally, we denote by $\mathbb{D}\left(A,B\right)$ the Skorohod space of $B$-valued càdlàg functions on a subset $A$ of $\mathbb{R}_+$. 
For every process $\left(\Z_t, t \in A \right) \in \mathbb{D}\left(A,B\right)$ and $x \in B$, we will denote 
\begin{equation*}
\mathbb{E}_x\left[f\left(\Z_t\right) \right] := \mathbb{E}\left[ f\left(\Z_t\right) \vert X_0=x\right] \correct{.}
\end{equation*}

\section{Definition of the population}
\label{section: model}
In this section we describe informally the population process. In Section \ref{section: proof model} we give a rigorous definition as a strong solution of a stochastic differential equation (SDE).

The population is described at any time $t \in \mathbb{R}_+$ by the finite point measure $\Zbar_t \in \Zbarset$ giving the label and the trait of every individual living in the population at this time. We introduce 
\begin{equation*}
\mathbb{G}\left( t\right) := \left\{u \in \mathcal{U}: \int_{\mathcal{U}\times\mathcal{X}}\mathbbm{1}_{\{v=u\}}\Zbar_t\left(\textrm{d}v,\textrm{d}x\right) \correct{\geq 1} \right\}
\end{equation*} 
the set of labels of living individuals at time $t$. For every individual labeled by $u \in \mathbb{G}\left( t\right)$ we \correct{denote by} $X^u_t$ its trait, and with a slight abuse of notation, $X^u_s$ will denote the trait of its unique ancestor living at time $s \in [0,t]$.
Using these notations, we can write 
\begin{equation*}
\Zbar_t = \sum_{u \in \mathbb{G}\left( t\right)}\delta_{\left(u,X^u_t\right)} \quad \textrm{and} \quad \Z_t := \sum_{u \in \mathbb{G}\left( t\right)}\delta_{X^u_t}.
\end{equation*}

The initial population is given by a measure $\bar{z} = \sum_{i=1}^{N}\delta_{\left(u^i,x^i\right)}$. 
During their lives, the traits of the individuals in the population evolve according to population-dependent dynamics. We introduce $\mu\in\mathfrak{B}(\mathcal{X} \times \Zset \times \mathbb{R}_+,\mathcal{X})$, such that, in a population $\Zbar_t$ at time $t$, for all $u \in \mathbb{G}\left( t\right)$
\begin{equation*}
\frac{\textrm{d}X^u_t}{\textrm{d}t} = \mu\left(X^u_t,\Z_t,t\right).
\end{equation*}
The branching rate is given by a function $B\in\mathfrak{B}(\mathcal{X} \times \Zset \times \mathbb{R}_+,\mathbb{R}_+)$. Thus, an individual with trait $x$ in a population $\Z_t$ at time $t$ dies at an instantaneous rate $B\left(x,\Z_t,t\right)$. 

It produces an offspring of $n$ individuals, where $n$ is randomly chosen with distribution $\left(p_{k}\left(x,\Z_t,t\right), k \in \mathbb{N}\right)$ of first moment $m\left(x,\Z_t,t\right)$. Thus branching events that lead to $n$ children happen at rate $B_n(\cdot) := B(\cdot)p_n(\cdot)$. If $n\geq 0$, we denote by  $\boldsymbol{y} = \left(y^1, \cdots,y^n\right)\in\mathcal{X}^n$ the offspring traits at birth, randomly chosen according to the law $K_n(x,\Z_t,t,\cdot)\mathcal{M}_n(\text{d}\cdot)$, where $K_n(x,\Z_t,t,\cdot)\in\mathfrak{B}(\mathcal{X}^n,\mathbb{R}_+)$ and $\mathcal{M}_n(\cdot)$ is a finite measure on $\mathcal{X}^{\mathbb{N}}$ such that, for all $A\in \mathcal{X}^{\mathbb{N}}$, 
$\mathcal{M}_n\left(\pi_{\{i\geq n+1\}}A\right) = 0$, where $\pi_{\{i\geq n+1\}}$ denotes the projection map on the set ${\{i \in \mathbb{N}^*, i\geq n+1\}}$.
For example, $\mathcal{M}_n(\cdot)$ could be the product measure of the Lebesgue measure on $\mathcal{X}^{n}$ and the null measure $z(\cdot)$ for the other components of the vector $\mathbf{y} \in \mathcal{X}^{\mathbb{N}}$, such that
$$
\mathcal{M}_n(\text{d}\mathbf{y}) := \bigotimes_{i=1}^n\text{d}y^i\bigotimes_{i\geq k+1}z(\text{d}y^i).
$$
Another classical example is the perfect cellular division. An individual of trait $x$ gives birth to $2$ children of same size $x/2$. To handle this case we have to slightly modify the formalism and consider that the random vector of trait is given by $\mathbf{Y} = (x\Theta,x\Theta)$ where the law of $\Theta$ is given by $K_2(x,\Z,t,\theta)\mathcal{M}_2(\text{d}\theta) = \delta_{1/2}(\text{d}\theta)$. More generally we can introduce a measurable function $F_{x,\Z,t}$ such that $F_{x,\Z,t}\left(\mathbf{y}\right) = \boldsymbol{\theta}$ and $\boldsymbol{\theta}$ is chosen according to the law $K_n(x,\Z_t,t,\cdot)\mathcal{M}_n(\text{d}\cdot)$. That being said, we will stick to the previously introduced formalism.

The labeling choice for these children is arbitrary yet necessary to uniquely define a stochastic point process in $\mathbb{D}\left(\mathbb{R}_+,\Zbarset\right)$. Here, for a parent individual of label $u$, for all $1 \leq i \leq n$, the $i$-th child is labeled $ui$ and its trait is $y^i$, the $i$-th coordinate of the vector $\boldsymbol{y}$. 

We \correct{denote by} $T_{\textrm{Exp}} \in\overline{\mathbb{R}}_+$ the explosion time of the process $\left(\Zbar_t, t\geq 0 \right)$, defined as the limit of its jumps times $(T_k, k\geq 0)$.  In order to ensure the non-explosion of this process in finite time we introduce the following set of \correct{hypotheses}. 
\begin{assumption}\label{assumption A}
We consider the following assumptions:
\begin{enumerate}
\item There exists $\mu_0\in\mathcal{C}(\mathbb{R}_+,\mathbb{R}_+)$, such that for all $\left(x,\Z,t\right) \in \mathcal{X}\times\Zset\times\mathbb{R}_+$
\begin{equation*}
\vert\mu\left(x,\Z,t\right)\vert \leq \mu_0(t)\left(1+ \vert x \vert + \left\vert\frac{\int_\mathcal{X}x\Z(\text{d}x)}{\langle \Z, 1 \rangle}\right\vert \right),
\end{equation*} 
\item For all $\left(x,\Z,t\right) \in \mathcal{X}\times\Zset\times\mathbb{R}_+$, $
B_1\left(x,\Z,t\right) <+ \infty.$
\item There exists $b_0\in\mathcal{C}(\mathbb{R}_+,\mathbb{R}_+)$, such that for all $\left(x,\Z,t\right) \in \mathcal{X}\times\Zset\times\mathbb{R}_+$
\begin{equation*}
\sum_{n \neq 1}nB_n\left(x,\Z,t\right) \leq b_0(t)\left(1+\vert x\vert\right).
\end{equation*}
\item For all $\left(x,\Z,t,n\right) \in \mathcal{X}\times\Zset\times\mathbb{R}_+\times \mathbb{N}^*$,
\begin{equation*}
K_n\left(x,\Z,t,\mathcal{A}_{n}(x)\right) = 0,\quad \textrm{where } \ \mathcal{A}_{n}(x) := \bigg\{ (y^i, 1\leq i \leq n)\in \mathcal{X}^n: \ \ \sum_{i=1}^n \vert y^i\vert > \vert x \vert \bigg\}.
\end{equation*}
\end{enumerate}
\end{assumption}
The growth rate of the traits of individuals is bounded in the first hypothesis by an exponential growth rate controlled by the trait of each individual and the mean trait in the population. The second point ensures that events that do not change the number of individuals do not accumulate in finite time. \correct{The third hypothesis uniformly controls the minimum lifetime of an individual and is used in the proof of Proposition \ref{prop: unique original} to control the first moment of the total mass}. This hypothesis, together with the first one, ensures that the lifespan of each individual decreases exponentially at most with its trait. Note that this assumption does not constrain the function $B_1(\cdot)$. The last hypothesis restricts the framework under consideration to fragmentation processes that do not create matter or energy. 
This set of hypotheses is sufficiently large to cover a large portion of models in physics and ecology, exponential growth being a classical assumption in many stochastic models in ecology and evolution, see \textit{e.g.}  \cite{REES2018121, cohen1995population}.

\begin{proposition}\label{prop: unique original}
Under Assumption \ref{assumption A}, the sequence $(T_k, k\geq 0)$ of jumps times of the process $\left(\Zbar_t, t\geq 0\right)$ tends to infinity almost surely.
\end{proposition} 
We can thus conclude, following the proof of Theorem 2.1 in \cite{Marguet19}, that under this set of hypotheses, the process $\left(\Zbar_t, t\geq 0\right)$ is uniquely defined on $\mathbb{R}_+$.
However, the spinal construction introduced in Section \ref{section: Many-to-One formula} can be established with less restrictive hypotheses. In this case, accumulation of jumps times may happen in finite time and the spinal construction holds until this explosion time.

\section{Results}
\label{section: Many-to-One formula}

In this section, we consider the law of a randomly sampled individual in the general branching population described in Section \ref{section: model}. Our main result gives the appropriate change of measure linking this distribution at time $t$ to the trajectory of an auxiliary process until this time. We then explicit this auxiliary process as a spinal construction. 

The spinal construction generates a $\Zbarset$-valued process along with the label of a distinguished individual that can change with time. \correct{When there is no possible confusion, the label of the spinal individual at any time will be denoted $e$, and $x_e$ its trait}. For convenience, we will \correct{denote by} $\Zbarhatset$ and $\Zhatset$ the sets such that
\begin{equation*}\label{Z hat space}
\Zbarhatset := \left\{\left(e,\Zbar\right) \in \mathcal{U} \times \Zbarset: \ \langle \Zbar, \mathbbm{1}_{\{ e\}\times\mathcal{X}} \rangle\geq 1  \right\}, \quad \textrm{and } \quad  
\Zhatset := \left\{\left(x,\Z\right) \in \mathcal{X} \times \Zset: \ \langle \Z, \mathbbm{1}_{\{ x\}} \rangle \geq 1  \right\}.
\end{equation*}
Thus, the spine process is a $\Zbarhatset$-valued branching process and its marginal is a $\Zhatset$-valued branching process.
We propose here a generalized spinal construction, where branching rates are biased with weight functions chosen in a set $\mathcal{D}$, defined by
\begin{equation}\label{def: D space}
\mathcal{D} := \left\{F_f \in \mathfrak{B}\left(\Zhatset\times\mathbb{R}_+,\mathbb{R}_+\right)\ \textit{s.t.} \  \left(f,F\right) \in \mathcal{C}^1\left(\mathcal{X}\times\mathbb{R}_+,\mathbb{R}\right)\times \mathcal{C}^1\left(\mathcal{X}\times\mathbb{R}\times\mathbb{R}_+,\mathbb{R}_+^*\right) \right\},
\end{equation}
where for every $\left(x,\Z,t\right) \in \Zhatset\times\mathbb{R}_+, \ F_f(x,\Z,t) := F(x,\langle\Z,f(\cdot,t)\correct{\rangle},t)$. \correct{In order to alleviate the notations when there is no ambiguity, we will omit the subscript $f$.}

In the following, for every $(u,\Zbar) \in \Zbarhatset$ we denote \correct{by} $x_u$ the trait of the individual of label $u$ in the population $\Zbar$, and for every $n\geq 0$ and every $\boldsymbol{y}=(y^i,1\leq i \leq n) \in \mathcal{X}^n$ we write
\begin{equation}\label{nu+}
\Zbar_+(u,\boldsymbol{y}) := \Zbar -\delta_{(u,x_u)} + \sum_{i=1}^{n} \delta_{(ui,y^i)}, \quad \textrm{and} \quad \Z^+(x,\boldsymbol{y}) := \Z -\delta_{x} + \sum_{i=1}^{n} \delta_{y^i}.
\end{equation}

We introduce the key operator $\mathcal{G}$ involved in the spinal construction. It is defined for all $F \in \mathcal{D}$ and $(x,\Z,t) \in \Zhatset\times \mathbb{R}_+$ by
\begin{multline}\label{def: mathcal G}
\mathcal{G}F(x,\Z,t) := GF\left(x,\Z,t\right) \\ + \sum_{n\geq 0} \Bigg\{ B_n(x,\Z,t)\int_{\mathcal{X}^n} \bigg[ \sum_{i=1}^n F \left(y^i, \Z^+(x,\boldsymbol{y}),t\right) 
- F\left(x, \Z^+(x,\boldsymbol{y}),t\right)\bigg] K_n\left(x,\Z,t,\boldsymbol{y}\right)\mathcal{M}_n(\text{d}\boldsymbol{y})  \\
  + \int_{\mathcal{X}} B_n(\mathfrak{x},\Z,t)\int_{\mathcal{X}^n} \big[ F\left(x, \Z^+(\mathfrak{x},\boldsymbol{y}),t\right)- F(x,\Z,t) \big] K_n\left(\mathfrak{x},\Z,t,\boldsymbol{y}\right)\mathcal{M}_n(\text{d}\boldsymbol{y})\Z(\textrm{d}\mathfrak{x})\Bigg\},
\end{multline}
where the operator $G$ is the generator of the deterministic evolution of the traits between branching events, given for every $F \in \mathcal{D}$, and $(x,\Z,t) \in \Zhatset\times \mathbb{R}_+$ by 
\begin{multline}\label{def: G}
GF\left(x,\Z,t\right) := D_1F\left(x,\langle \Z,f(\cdot,t)\rangle,t\right)\cdot\mu\left(x,\Z,t\right) + D_3F\left(x,\langle \Z,f(\cdot,t)\rangle,t\right)\\ +D_2F\left(x,\langle\Z,f(\cdot,t)\rangle,t \right)\int_{\mathcal{X}}\left[\frac{\partial f}{\partial x}(\mathfrak{x},t)\cdot\mu\left(\mathfrak{x},\Z,t\right) + \frac{\partial f}{\partial t}(\mathfrak{x},t)\right]\Z(\textrm{d}\mathfrak{x}),
\end{multline}
where $D_iF(\cdot,\cdot)$ denotes the derivative of the function $F \in \mathcal{C}^1(\mathcal{X}\times\mathbb{R}\times \mathbb{R}_+,\mathbb{R}_+^*)$ with respect to $i$-th variable.
Rigorously, the previously introduced objects are families of operators $(\mathcal{G}_{x,\Z,t}, x,\Z,t\in \mathcal{X}\times\Zset\times\mathbb{R}_+)$ and $(G_{x,\Z,t}, x,\Z,t \in \mathcal{X}\times\Zset\times\mathbb{R}_+)$. This abuse of notation will be used for the operators subsequently introduced in this work.

Note that the operator $\mathcal{G}$ is generally not the generator of a conservative Markov process on $\Zhatset$. Indeed $\mathcal{G}1\correct{(x,\Z,t)} = B(x,\Z,t)(m(x,\Z,t)-1)$, which is non-zero if there exists $x,\Z,t\in \mathcal{X}\times\Zset\times\mathbb{R}_+$ such that the mean number of children $m\left(x,\Z,t\right) \neq 1$. We point out that these operators are the counterparts- for interacting, structured,    branching populations- of the Schrödinger operator $\mathcal{G}$ introduced in \cite{Cloez17}. To ensure that the functions $\mathcal{G}F$ are finite on $\Zhatset\times \mathbb{R}_+$, we make an assumption on the functions $F \in \mathcal{D}$.
\begin{assumption}\label{assumption B}
Let $(p_n, n\in\mathbb{N})$ and $(K_n, n\in\mathbb{N})$ be the reproduction parameters of the original branching process and $F \in \mathcal{D}$ a weight-function. For all $(x,x_e,\Z,t) \in \mathcal{X}\times\Zhatset\times\mathbb{R}_+$, 
\begin{multline*}
\sum_{n \in\mathbb{N}}p_n(x,\Z,t)\int_{\mathcal{X}^n}F\left(x_e,\Z^+(x,\boldsymbol{y}),t\right)K_n\left(x,\Z,t,\boldsymbol{y}\right)\mathcal{M}_n(\text{d}\boldsymbol{y}) \\
+ \sum_{n \in\mathbb{N}}p_n(x,\Z,t)\int_{\mathcal{X}^n}\sum_{i=1}^n F\left(y^i,\Z^+(x,\boldsymbol{y}),t\right)K_n\left(x,\Z,t,\boldsymbol{y}\right)\mathcal{M}_n(\text{d}\boldsymbol{y}) < +\infty.
\end{multline*}
\end{assumption}
Notice that for a chosen set of parameters of the original branching process $(\Z_s, s\geq 0)$, this assumption restricts the set of suitable weight functions. 

\correct{Now we can introduce the $\Zbarhatset$-valued spine process associated with a function $\psi \in \mathcal{D}$ \correct{satisfying} Assumption \ref{assumption B}, that will be called the $\psi$-spine process in order to keep track of the function used in its construction.
For such functions $\psi$, the generators $(\widehat{\mathcal{L}}^t_{\psi}, t\geq 0)$ of the $\psi$- spine process is defined for all $h \in \mathcal{C}^1(\mathcal{U}\times\mathcal{X},\mathbb{R})$, $H \in  \mathcal{C}^1(\mathcal{U}\times\mathbb{R},\mathbb{R}_+^*)$,
and for all $(e,\Zbar) \in \Zbarhatset\times\mathbb{R}_+$ by} 
\begin{multline}\label{generator spinal}
\widehat{\mathcal{L}}^t_{\psi}H(e,\langle \Zbar, h \rangle) := \widehat{G}H(e,\langle \Zbar, h \rangle) + \sum_{n\geq 0} \int_{\mathcal{U} \times \mathcal{X}}B_n(x,\Z,t)  \\
\times \int_{\mathcal{X}^n}K_n\left(x,\Z,t,\boldsymbol{y}\right) \Bigg\{\mathbbm{1}_{\{u=e\}}\sum_{i=1}^n \big[ H \left(ei, \langle\Zbar_+(e,\boldsymbol{y}), h \rangle \right)- H(e,\langle \Zbar, h \rangle)\big]\frac{\psi(y^i,\Z^+(x_e,\boldsymbol{y}))}{\psi(x_e,\Z,t)} \\
 +\mathbbm{1}_{\{u\neq e\}}  \big[ H\left(e,\langle \Zbar_+(u,\boldsymbol{y}), h \rangle \right) - H(e,\langle \Zbar, h \rangle) \big]\frac{\psi(x_e,\Z^+(x_u,\boldsymbol{y}),t)}{\psi(x_e,\Z,t)} \Bigg\}  \mathcal{M}_n(\text{d}\boldsymbol{y})\Zbar(\textrm{d}u,\textrm{d}x),
\end{multline}
where
\begin{equation*}
\widehat{G}H(e,\langle \Zbar, f \rangle) := D_2H\left(e,\langle\Zbar,h(\cdot)\rangle \right)\int_{\mathcal{U}\times\mathcal{X}}\frac{\partial h}{\partial x}(u,x)\cdot\mu\left(x,\Z,t\right)\Zbar(\textrm{d}u,\textrm{d}x).
\end{equation*}
Notice that \correct{the} branching rates of both spinal and non-spinal individuals are biased by the function $\psi$. Assumption \ref{assumption B} ensures that the total branching rate is finite from every state of the spinal process, however it is not sufficient to avoid explosion of this process in finite time. Dynamics \ref{spinal outside rates} and \ref{spine rates} below will provide a more \correct{detailed} explanation of the spine process associated with this generator.

Finally, we introduce for all $t\geq 0$, the $\mathcal{U}$-valued random variable $U_t$ that picks an individual alive at time $t$. Its law is characterized by the function $p_u\left(\Zbar_t\right)$ which yields the probability to choose the individual of label $u$ in the set $\mathbb{G}\left(t\right)$.
We can now state our main result, that is a Girsanov-type formula for the spinal change of measure. It characterizes the joint probability distribution of $\left(U_t,(\Zbar_s,s\leq t)\right)$- that is the randomly sampled individual in the population $\Zbar_t$ at time $t$ and the whole trajectory of the population until this time- and links it to the law of the spine process through a path-integral formula. 
\begin{theorem}\label{thm:pdmc}
Let $\psi \in\mathcal{D}$ \correct{satisfying} Assumption \ref{assumption B}, $t \geq 0$, and $\bar{z} \in \Zbarset$. 
Let $\left(\left(E_t,\Zbarhat_{t}\right), t\geq 0 \right)$ be the time-inhomogeneous $\Zbarhatset$-valued branching process with interactions defined by the infinitesimal generator $\widehat{\mathcal{L}}_{\psi}$ introduced in \eqref{generator spinal}. Let $\widehat{T}_{\textrm{Exp}}$ denote its explosion time and $\left(\left(Y_t,\Zhat_{t}\right), t\geq 0 \right)$ its projection on $\Zhatset$.

For every \correct{non-negative measurable} function $H$ on $\mathcal{U} \times \mathbb{D}\left([0,t],\Zset\right)$:
\begin{multline*}\label{eq:pdmc}
\mathbb{E}_{\bar{z}}\left[\mathbbm{1}_{\{T_{\textrm{Exp}} > t, \mathbb{G}(t) \neq \emptyset\}}H\left(U_t,(\Zbar_s,s\leq t)\right) \right] = \\ \langle z,\psi(\cdot,z) \rangle \mathbb{E}_{\bar{z}}\left[\mathbbm{1}_{\{\widehat{T}_{\textrm{Exp}}>t\}} \xi\left(E_t,(\Zbarhat_s, s\leq t)\right) H\left(E_t,(\Zbarhat_s, s\leq t)\right)\right],
\end{multline*}
where:
\begin{equation*}
\xi\left(E_t,(\Zbarhat_s, s\leq t)\right) :=  \frac{p_{E_t}\left(\Zbarhat_t\right)}{\psi\left(Y_t,\Zhat_t,t\right)}\exp\left(\int_0^t \frac{\mathcal{G} \psi\left(Y_s,\Zhat_s,s\right)}{\psi\left(Y_s,\Zhat_s,s\right)}\textrm{d}s\right).
\end{equation*}
\end{theorem}
We take inspiration from the work of Bansaye \cite{B21} for the proof. The idea is to decompose both processes on their possible trajectories, then establish by induction on the successive jumps times the equality in law for the trajectories between these times. 
 
The process $\left(\left(E_t,\Zbarhat_{t}\right), t\geq 0 \right)$ gives at any time $t\geq0$ the label of the spinal individual- that encodes the whole spine lineage- and the spinal population. Our result thus links, for every $\psi$, the sampling of an individual and the trajectory of the population to the trajectory of the spine process. The path integral term that links these two terms is difficult to handle in general and finding eigenfunctions of $\mathcal{G}$ may \correct{greatly} simplify the expression \cite{Cloez17,Athreya00,B21}. 
Finding such functions for single type, density-dependent populations is possible in models with simple interactions \cite{B21}[Section 3]. Nevertheless, this becomes a challenging issue in the majority of scenarios. Subsequent sections of this work will explore applications of this formula where the path-integral component is tractable.

We first introduce additional notations concerning the dynamics of the spine process given by the generators introduced in \eqref{generator spinal}. Following notations of Section \ref{section: model}, and disregarding the dependency on the chosen function $\psi \in\mathcal{D}$ \correct{satisfying} Assumption \ref{assumption B} in the subsequent branching parameters, we introduce the dynamics of the traits in the spine process. As previously discussed, the construction distinguishes dynamics of the spine from the rest of the individuals. 
We first introduce the branching parameters of the individuals outside the spine in a spinal population.
\begin{dyn}[Individuals outside the spine]\label{spinal outside rates}
For all $\left(n,x,\left(x_e,\Z\right),t\right)$ in $\mathbb{N}^* \times \mathcal{X} \times \Zhatset \times \mathbb{R}_+$\correct{,}
\begin{enumerate}
\item $\widehat{K}_n\left(x_e,x,\Z,t,\cdot\right) \in \mathfrak{B}\left(\mathcal{X}^n \right)$ is the density of the traits at birth of the $n$ children generated by a non-spinal individual of trait $x$ at time $t$ in a spinal population $\left(x_e,\Z\right)$. 
\begin{equation}\label{def: K hat}
\widehat{K}_n\left(x_e,x,\Z,t,\cdot\right) := \frac{1}{\widehat{\Gamma}_n\left(x_e,x,\Z,t\right)} \psi\left(x_e,\Z^+(x,\cdot),t\right)K_n\left(x,\Z,t,\cdot\right),
\end{equation}  
where $\widehat{\Gamma}_n(\cdot)$ is the normalization function, defined as
$$\widehat{\Gamma}_n\left(x_e,x,\Z,t\right) := \int_{\mathcal{X}^n}\psi\left(x_e,\Z^+(x,\boldsymbol{y}),t\right)K_n\left(x,\Z,t,\boldsymbol{y}\right)\mathcal{M}_n(\text{d}\boldsymbol{y}).$$
We recall the definition of $\Z^+$ in \eqref{nu+}.
\item The law $\left(\widehat{p}_n\left(x_e,x,\Z,t\right), n \in \mathbb{N}\right)$ of the number of children of a \correct{non-spinal} individual of trait $x$ branching at time $t$ in a spinal population $\left(x_e,\Z\right)$, is defined for all $n \in \mathbb{N}$ as
$$
\widehat{p}_n\left(x_e,x,\Z,t\right) := \frac{1}{\sum_{k\in\mathbb{N}}\widehat{\Gamma}_k\left(x_e,x,\Z,t\right)p_k\left(x,\Z,t\right)}\widehat{\Gamma}_n\left(x_e,x,\Z,t\right)p_n\left(x,\Z,t\right).
$$
\correct{Note that Assumption \ref{assumption B} ensures that the sum in the denominator is finite.}
\item Each individual of trait $x$ outside the spine of trait $x_e$ in a population $\Z$ at time $t$, branches to $n$ children at rate $\widehat{B}_n\left(x_e,x,\Z,t\right)$, defined as
$$
\widehat{B}_n\left(x_e,x,\Z,t\right):=\frac{\widehat{\Gamma}_n\left(x_e,x,\Z,t\right)}{\psi\left(x_e,\Z,t\right)}B_n\left(x,\Z,t\right).
$$ 
\end{enumerate}
\end{dyn}

The total branching rate outside the spine is defined, for all $(x_e,\Z,t)\in\Zhatset\times\mathbb{R}_+$ by
\begin{equation*}
\widehat{\tau}\left(x_e,\Z,t\right) := \int_{\mathcal{X}}\sum_{n \geq 0}\widehat{B}_n\left(x_e,x,\Z,t\right)\Z(\textrm{d}x)- \sum_{n \geq 0}\widehat{B}_n\left(x_e,x_e,\Z,t\right).
\end{equation*} 
We now introduce the branching parameters of the spine in a $\psi$-spinal construction.
\begin{dyn}[Spinal individual]\label{spine rates}
 For all $\left(n,\left(x_e,\Z\right),t\right)$ in $\mathbb{N}^* \times \Zhatset \times \mathbb{R}_+$,
\begin{enumerate}
\item $\widehat{K}^*_n\left(x_e,\Z,t,\cdot\right) \in \mathfrak{B}\left(\mathcal{X}^n \right)$ is the density of the traits at birth of the $n$ children generated by the spinal individual of trait $x_e$ at time $t$ in a population $\Z$. 
For all $\boldsymbol{y} \in \mathcal{X}^n$,
\begin{equation}\label{def: K star hat}
\widehat{K}^*_n\left(x_e,\Z,t,\boldsymbol{y}\right) :=  \frac{1}{\widehat{\Gamma}^*_n\left(x_e,\Z,t\right)}\sum_{i=1}^n\psi\left(y^i,\Z^+(x_e,\boldsymbol{y}),t\right)K_n\left(x_e,\Z,t,\boldsymbol{y}\right),
\end{equation}  
where $\widehat{\Gamma}^*_n(\cdot)$ is the normalization function, defined as
$$\widehat{\Gamma}^*_n\left(x_e,\Z,t\right) := \int_{\mathcal{X}^n}\sum_{i=1}^n \psi\left(y^i,\Z^+(x_e,\boldsymbol{y}),t\right)K_n\left(x_e,\Z,t,\boldsymbol{y}\right)\mathcal{M}_n(\text{d}\boldsymbol{y}).$$
\item The law $\left(\widehat{p}^*_n\left(x_e,\Z,t\right), n \in \mathbb{N}\right)$ of the number of children of the spinal individual of trait $x_e$ branching at time $t$ in a population $\Z$, is defined for all $k \in \mathbb{N}$ as
\begin{equation}\label{def: p star hat}
\widehat{p}^*_n\left(x_e,\Z,t\right) := \frac{1}{\sum_{k\in\mathbb{N}}\widehat{\Gamma}^*_k\left(x_e,\Z,t\right)p_k\left(x_e,\Z,t\right)}\widehat{\Gamma}^*_n\left(x_e,\Z,t\right)p_n\left(x_e,\Z,t\right).
\end{equation}  
\item The spinal individual of trait $x_e$ in a population $\Z$ at time $t$, branches to $n$ children at a rate $\widehat{B}_n^*\left(x_e,\Z,t\right)$, defined as
\begin{equation}\label{def: B star hat}
\widehat{B}^*_n\left(x_e,\Z,t\right):=\frac{\widehat{\Gamma}^*_n\left(x_e,\Z,t\right)}{\psi\left(x_e,\Z,t\right)}B_n\left(x_e,\Z,t\right).
\end{equation} 
\item When the spinal individual of trait $x_e$ branches at time $t$ in a population $\Z$ and is replaced by $n$ children with trait $\boldsymbol{y}$, the integer-valued random variable $J(x_e,\Z,t,\boldsymbol{y})$ choosing the new spinal individual after a spinal branching event is given, for all $1 \leq j\leq n$ by
\begin{equation}\label{spine proba}
\mathbb{P}\left(J(x_e,\Z,t,\boldsymbol{y})=j\right) = \frac{\psi\left(y^j,\Z^+(x_e,\boldsymbol{y}),t \right)}{\sum_{i=1}^n \psi\left(y^i,\Z^+(x_e,\boldsymbol{y}),t \right)}.
\end{equation}
The first spine in an initial population $z$ is chosen according to the same law.
\end{enumerate}
\end{dyn}
The total branching rate from every \correct{spinal state} $(x_e,\Z,t)\in\Zhatset\times\mathbb{R}_+$ is \correct{denoted by}
\begin{equation}\label{tau hat tot}  
\widehat{\tau}_{\textrm{tot}}\left(x_e,\Z,t\right) := \sum_{n\geq 0} \widehat{B}_n^*\left(x_e,\Z,t\right) + \widehat{\tau}\left(x_e,\Z,t\right).
\end{equation}
Notice that $\widehat{K}^*_0 = 0$, therefore the spinal individual cannot die without children and the spinal population never goes extinct. Notice that for $\psi \equiv 1$, individuals outside the spine follow the same dynamics as the individuals in the population $\left(\Z_t,t\geq 0\right)$. In this case, the spinal individual of trait $x_e$ branches at time $t$ in a population $\Z$ with rate $m(x_e,\Z,t)B(x_e,\Z,t)$, where $m(\cdot)$ is the mean number of children that is finite under Assumption \ref{assumption B}. The random number of children at a branching event thus  follows the size-biased law $np_n(\cdot)/m(\cdot)$ and the new spinal individual is chosen uniformly among the offspring.

Theorem \ref{thm:pdmc} is verified until the first explosion time of both processes. We established in Proposition \ref{prop: unique original} that under Assumption \ref{assumption A} the branching process does not explode in finite time. To ensure the non explosion of the spine process, we have to consider an additional assumption on the weight function $\psi$ used for the construction.

\begin{assumption}\label{assumption C}
There exists a positive continuous function $\hat{b}_0$ on $\mathbb{R}_+$, such that for all $\left(x,\Z,t\right) \in \mathcal{X}\times\Zset\times\mathbb{R}_+$
\begin{equation*}
\sum_{n \neq 1}n\left(\widehat{B}_n\left(x,\Z,t\right)+ \widehat{B}^*_n\left(x,\Z,t\right) \right) \leq \hat{b}_0(t)\left(1 + \vert x\vert \right).
\end{equation*}
\end{assumption}
This assumption, involving both the branching parameters and the function $\psi$, is stronger than Assumption \ref{assumption B}. The set of weight functions that can be used to construct a spine process that does not explode in finite time may differ from one model to another. However, one may rather use more restrictive conditions that are sufficient for every branching process under Assumption \ref{assumption A}. For example, in mass-conservative models, taking $\psi(x,\Z)=x$ ensures the non-explosion of the spine process regardless the initial branching process that satisfies Assumption \ref{assumption A}.  

\begin{proposition}\label{prop: generator spine}
Under Assumption \ref{assumption A}, for every $\psi \in \mathcal{D}$ \correct{satisfying} Assumptions \ref{assumption B} and \ref{assumption C}, the $\psi$-spine process $((E_t,\Zbarhat_t),t \geq0)$ \correct{defined by the infinitesimal generator $\widehat{\mathcal{L}}_{\psi}$ introduced in \eqref{generator spinal}} does not explode in finite time. Furthermore the generator $\widehat{L}_{\psi}$ of the marginal spine process $((Y_t,\Zhat_t),t \geq0)$ given by Dynamics \ref{spinal outside rates} and \ref{spine rates}, is defined for all function $F \in \mathcal{D}$ and all $\left(x,\Z,t\right) \in \mathcal{X}\times\Zset\times\mathbb{R}_+$ by
\begin{equation*}
\widehat{L}_{\psi}F\left(x,\Z,t\right) := \frac{\mathcal{G}\left[\psi F\right]\left(x,\Z,t\right)}{\psi\left(x,\Z,t\right)} - \frac{\mathcal{G}\psi\left(x,\Z,t\right)}{\psi\left(x,\Z,t\right)}F\left(x,\Z,t\right).
\end{equation*}
\end{proposition}
\correct{The proof of non-explosion is shown following the proof of Proposition \ref{prop: unique original}, and the expression of the generator of the marginal spine process is purely computational. The detailed proof is presented in Appendix \ref{appendix a}.}

It follows that the marginal law of the spine process is characterized by the operator $\mathcal{G}$ and the weight function $\psi$. 
\begin{corollary}\label{cor: pdmc} Let $\psi \in \mathcal{D}$ \correct{satisfying} Assumption \ref{assumption C}, $t\geq 0$ and $\bar{z} \in \Zbarset$. Under Assumption \ref{assumption A}, for any $f$ \correct{non-negative measurable} function on $ \mathcal{U} \times \mathbb{D}\left([0,t],\mathcal{X}\times\Zbarset\right)$,
\begin{multline*}
\mathbb{E}_{\bar{z}}\left[\sum_{u \in \mathbb{G}(t)}\psi\left(X^u_t,\Zbar_t,t \right)f\left(u,\left(\left(X^u_s,\Zbar_s\right),0\leq  s \leq t\right)\right) \right]\\
= \langle z,\psi(\cdot,z,0) \rangle \mathbb{E}_{\bar{z}}\left[ \exp\left(\int_0^t \frac{\mathcal{G} \psi\left(Y_s,\Zhat_s,s\right)}{\psi\left(Y_s,\Zhat_s,s\right)}\textrm{d}s\right)f\left(E_t,\left(\left(Y_s,\Zbarhat_s\right), 0 \leq s \leq t\right)\right)\right].
\end{multline*}
\end{corollary}
\begin{proof}
We use Assumptions \ref{assumption A} and \ref{assumption C} to ensure that $T_{\textrm{Exp}}$ and $\widehat{T}_{\textrm{Exp}}$ are almost surely infinite. Let $f$ be \correct{non-negative measurable} function on $\mathcal{U} \times \mathbb{D}\left([0,t],\mathcal{X}\times\Zbarset\right)$. We introduce $H$ the \correct{non-negative measurable} function defined for all $(u,\bar{z}_s, s\leq t) \in\mathcal{U}\times \mathbb{D}\left([0,t],\Zset\right)$ by
$$
H(u,\bar{z}_s, s\leq t) := \psi(X^u_t,z_t,t)f\left(u,\left(X^u_s,\bar{z}_s\right), s\leq t\right)\langle z_t,1\rangle.
$$
The corollary is thus a direct application of Theorem \ref{thm:pdmc} to the function $H$ with a uniformly sampled individual. 
\end{proof}
This formula gives a change of probability that involves the function $\mathcal{G}\psi/\psi$ with a path-integral formula. This study is related to Feynman-Kac path measures and semigroups. We refer to \cite{Delmoral04} for an overview on this subject.
In the case of a branching process with interactions, the integral term depends on the trajectory of the whole spinal population. In general cases with interactions, the branching property is not verified and the so-called Many-to-One formula- see \textit{e.g.} Proposition 9.3 in \cite{bansaye2015stochastic}- fall apart\correct{.}
However, if $\psi$ is an eigenfunction of the operator $\mathcal{G}$, then we have the following Many-to-One formula for any non-negative measurable function $g$ on $ \mathbb{D}\left([0,t],\mathcal{X}\right)$
\begin{equation*}
\mathbb{E}_{\bar{z}}\left[\sum_{u \in \mathbb{G}(t)}\psi\left(X^u_t,\Zbar_t,t \right)g\left(X^u_s,s \leq t\right) \right] = C_t \mathbb{E}_{\bar{z}}\left[g\left(Y_s, s \leq t\right)\right],
\end{equation*}
where $C_t$ is a time-dependent positive constant. This formula, established in \cite{Cloez17} \correct{for branching processes without interactions}, reduces the empirical measure of the trajectories of all the individuals until time $t$ to the law of the trajectory of a unique individual in the spinal construction, the spinal individual. The spinal individual in this case can be considered as a typical individual, reflecting the average behavior of the whole population. 

\begin{remark}\label{rem mto}
If we assume for all $\left(x,\Z,t\right) \in \mathcal{X}\times\Zset\times\mathbb{R}_+$, that $B(x,\Z,t) = B(t)$ and for all $k\geq 0$, $p_k(x,\Z,t) = p_k(t)$ in the considered branching process, taking $\psi \equiv 1$ gives the classical Many-to-One formula \cite{BDMT2011}: 
\begin{equation*}
\mathbb{E}_{\bar{z}}\left[\sum_{u \in \mathbb{G}(t)}g\left(X^u_s,s \leq t\right) \right] = \mathbb{E}_{\bar{z}}\left[\langle \Z_t,1 \rangle\right] \mathbb{E}_{\bar{z}}\left[g\left(Y_s, s \leq t\right)\right]\correct{,}
\end{equation*}
where the average number of individuals in the population is given at time $t$ by 
\begin{equation*}
\mathbb{E}_{\bar{z}}\left[\langle \Z_t,1 \rangle\right] = \langle z,1\rangle \exp\left(\int_0^tB(s)(m(s) -1)\textrm{d}s \right).
\end{equation*}
\end{remark}

\section{Kesten-Stigum criterion}
\label{section: llogl}
In this section, we present a generalized Kesten-Stigum theorem for structured processes with interactions. \correct{For every functions $\psi \in \mathcal{D}$, we exhibit a non-negative martingale of the original branching process with interactions $(\Z_t,t\geq0)$ associated with the $\psi$-spinal construction and, under some assumptions on the set of functions $\psi$, we present a result on the degeneracy of its limiting martingale.}
For this section, we will suppose that Assumption \ref{assumption A} holds to prevent explosion of the process in finite time.
\begin{proposition}\label{prop:martingale}
Under Assumption \ref{assumption A}, for every $\psi \in \mathcal{D}$ \correct{satisfying} Assumption \ref{assumption B},
\begin{equation}\label{def:martingale}
W_t(\psi) := \sum_{u \in \mathbb{G}(t)}\exp\left(-\int_0^t\frac{\mathcal{G} \psi\left(X^u_s,\Z_s,s\right)}{\psi\left(X^u_s,\Z_s,s\right)}\textrm{d}s\right)\psi\left(X^u_t,\Z_t,t\right)
\end{equation}
is a non-negative martingale with respect to the filtration $\left(\mathcal{F}_t, t\geq 0\right)$ generated by the original process. It almost surely converges to a random variable $W(\psi) \in [0, \infty)$. 
\end{proposition}
Notice that if the process $(\Z_t, t\geq 0)$ goes extinct almost surely, then $W(\psi)=0$ almost surely. However, even on the survival event, the martingale $(W_t(\psi), t\geq 0)$ may also almost surely degenerate to 0. Note that Proposition \ref{prop:martingale} holds true for every function $\psi\in\mathcal{D}$. However, in order to prove the Kesten-Stigum result on the limiting martingale, we have to restrict the set of functions $\psi$. We recall the notation $\Z^+(x,\boldsymbol{y}) := \Z + \sum_{i=1}^n\delta_{y^i} - \delta_{x}$, introduced in \eqref{nu+}.
\begin{assumption}\label{assumption H_technical} For all $(x_e,(x,\nu),t) \in \mathcal{X}\times\Zhatset\times\mathbb{R}_+$, for all $k\geq 0$ and for all $\boldsymbol{y}$ in a subset $\mathcal{A} \subset \mathcal{X}^n$ such that $K_n\left(x,\nu,t,\boldsymbol{y}\right)\mathcal{M}_n\left(\mathcal{A}\right) >0$: 
$$ \psi\left(x_e,\Z^+(x,\boldsymbol{y}),t\right) = \psi(x_e,\nu,t).$$
\end{assumption}
Note that this assumption is weaker than assuming that the function $\psi$ is independent from the population. In fact the function $\psi$ is allowed to depend on a quantity that is conserved at non-spinal jump events, for example the total population mass if the jumps are mass conservative. Another interesting case is when $\psi$ is a function of a coupling variable that derives from the population, like the quantity of available resources for the individuals.
\begin{assumption}\label{assumption H_inf_technical} There exists $M \in \mathbb{R}_+^*$ such that, for all $((x,\nu),t,n) \in \Zhatset\times\mathbb{R}_+\times\mathbb{N}^*$ and for all $1\leq j\leq n$:
\begin{multline*}
\int_{\mathcal{X}^n}\psi\left(y^j,\Z^+(x,\boldsymbol{y}),t\right)\log\left(\psi\left(y^j,\Z^+(x,\boldsymbol{y}),t\right)\right)K_n(x,\nu,t,\boldsymbol{y})\mathcal{M}_n(\text{d}\boldsymbol{y})\\ \leq M \int_{\mathcal{X}^n}\sum_{i=1}^n\psi(y^i,\Z^+(x,\boldsymbol{y}),t)K_n(x,\nu,t,\boldsymbol{y})\mathcal{M}_n(\text{d}\boldsymbol{y}).
\end{multline*}
\end{assumption}
This technical assumption ensures that the function $\psi$ does not assign too much load to any specific child at a branching event. Notice that bounded functions verify this hypothesis.  
\begin{assumption}\label{assumption H} Assumptions \ref{assumption A} and \ref{assumption C} hold true and there exist $c,C,\overline{\tau},\underline{\tau} \in \mathbb{R}_+^*$ such that, for all $(x,\nu,t) \in \Zhatset\times\mathbb{R}_+$  
 $$c \leq \frac{\mathcal{G} \psi(x,\Z,t)}{\psi(x,\Z,t)}\leq C \quad \textrm{and } \quad \underline{\tau} \leq \sum_{n\geq 0}\widehat{B}^*_n(x,\nu,t) \leq \overline{\tau},$$
 where $\mathcal{G}$ and $\widehat{B}^*_n$ are respectively defined in \eqref{def: mathcal G} and \eqref{def: B star hat}.
\end{assumption}
The first hypothesis controls the range of variations of the exponential term in the martingale $W_t(\psi)$. Note that all eigenfunctions of the operator $\mathcal{G}$ verify this hypothesis. The second point ensures that the branching events of the spinal individual in the $\psi$-spinal process do not stop nor accumulate too fast.
We can now express the main result of this section.
\begin{theorem}\label{thm LlogL}
Let $\psi \in \mathcal{D}$ satisfying Assumption \ref{assumption H} and introduce for all $n \in \mathbb{N}$ 
$$
\overline{\widehat{p}}^*_n:=\sup_{(x,\Z,t) \in \Zhatset\times\mathbb{R}_+}\widehat{p}^*_n(x,\Z,t), \quad \ \underline{\widehat{p}}^*_n:=\inf_{(x,\Z,t) \in \Zhatset\times\mathbb{R}_+}\widehat{p}^*_n(x,\Z,t),
$$
$$
\overline{\widehat{K}}^*_n(\cdot):=\sup_{(x,\Z,t) \in \Zhatset\times\mathbb{R}_+}\widehat{K}^*_n(x,\Z,t,\cdot) \quad \text{and} \quad \quad \ \underline{\widehat{K}}^*_n(\cdot):=\inf_{(x,\Z,t) \in \Zhatset\times\mathbb{R}_+}\widehat{K}^*_n(x,\Z,t,\cdot),
$$
where $(\widehat{p}^*_n, k\geq 0)$ and $\widehat{K}^*_n$ are respectively the law of the number of children and the measure giving the offspring traits for the spine, defined respectively in \eqref{def: p star hat} and \eqref{def: K star hat}. 
\begin{enumerate}
\item Under Assumption \ref{assumption H_technical}, if
\begin{equation}
\sum_{n\geq 1}\overline{\widehat{p}}^*_n<+\infty, \quad \sup_{n\geq 1}\int_{\mathcal{X}^n}\overline{\widehat{K}}^*_n(\boldsymbol{y})\mathcal{M}_n(\text{d}\boldsymbol{y})<+\infty
\end{equation}
and
\begin{equation}\label{crit sup finite}
\sum_{n\geq 1} \overline{\widehat{p}}^*_n \int_{\mathcal{X}^n}\sup_{(x,\Z,t) \in \Zhatset\times\mathbb{R}_+}\left[\log\left(\sum_{i=1}^n\psi(y^i,\Z^+(x,\boldsymbol{y}),t) \right)\right] \overline{\widehat{K}}^*_n(\boldsymbol{y})\mathcal{M}_n(\text{d}\boldsymbol{y})<+\infty
\end{equation}
then, for all initial measure $z \in\Zset$ 
$$\mathbb{E}_{z}\left[W(\psi)\right] = \langle z,\psi(\cdot,z,0)\rangle.$$ 
\item Under Assumption \ref{assumption H_inf_technical},
\begin{equation}\label{crit inf infinite}
\sum_{n\geq 1} \underline{\widehat{p}}^*_n \int_{\mathcal{X}^n}\inf_{(x,\Z,t) \in \Zhatset\times\mathbb{R}_+}\left[\log\left(\sum_{i=1}^n\psi(y^i,\Z^+(x,\boldsymbol{y}),t) \right)\right]\underline{\widehat{K}}^*_n(\boldsymbol{y})\mathcal{M}_n(\text{d}\boldsymbol{y}) = +\infty
\end{equation}
implies that $W(\psi) = 0$ almost surely.
\end{enumerate}
\end{theorem}
The idea of the proof, based on the conceptual proofs established in \cite{LPP95,Georgii03,B21}, is to use Theorem \ref{thm:pdmc} to consider the dual problem associated with the $\psi$-process. Then consider the spinal process as a process with immigration where the spinal individual provides new individuals at a $\psi$-biased rate, and the new spine is the new source of immigration. However as the function $\psi$ is not constant, the spinal construction also changes the dynamics of individuals outside the spine, that do not behaved as those in the original process $(\Z_t, t\geq 0)$. Assumption \ref{assumption H_technical} and \ref{assumption H_inf_technical} are used to control the behavior of the individuals outside the spine.

In the particular case $\psi \equiv 1$ or any positive constant. We have for all $x,\Z,t \in \Zhatset\times \mathbb{R}_+$
$$
\frac{\mathcal{G} \psi(x,\Z,t)}{\psi(x,\Z,t)} = B(x,\Z,t)\left(m(x,\Z,t) -1 \right),\quad \widehat{K}^*_n(x,\Z,t) = K_n(x,\Z,t)$$
and $$\widehat{p}^*_n(x,\Z,t) = \frac{np_n(x,\Z,t)}{m(x,\Z,t)}.
$$
Note that Assumptions \ref{assumption H_inf_technical} and \ref{assumption H_technical} are verified, and Assumption \ref{assumption H} becomes 
\begin{assumption}\label{assumption H psi 1} Assumptions \ref{assumption A} and \ref{assumption C} hold true and there exist $c,C,\underline{B},\overline{B} \in \mathbb{R}_+^*$ such that, for all $x,\Z,t \in \Zhatset\times\mathbb{R}_+$  
 $$c \leq B(x,\Z,t)(m(x,\Z,t)-1) \leq C \quad \textrm{and } \quad \underline{B} \leq B(x,\Z,t) \leq \overline{B} .$$
\end{assumption}
This assumption ensures that the process is strongly supercritical- in \correct{the} sense that the uniform lower bound of the reproductive law is strictly greater than 1- and that the only absorbing state for the branching process is the null measure $\Z \equiv 0$. This uniform hypothesis could be partially \correct{relaxed} under strong positivity assumptions on the generator of the branching process with interactions, see \cite{KLPP97} in the discrete case and \cite{harris1963theory} Chapter 3 for continuous time. It also restricts the setting to branching processes with bounded branching rates. 
\begin{corollary}\label{thm LlogL psi 1}
We assume that $\psi \equiv 1$ and that Assumption \ref{assumption H psi 1} holds true. Let 
$$
\overline{p}_n:= \underset{(x,\Z,t) \in \Zhatset\times\mathbb{R}_+}{\sup}p_n(x,\Z,t) \quad \text{and} \quad \underline{p}_n:=\underset{(x,\Z,t) \in \Zhatset\times\mathbb{R}_+}{\inf}p_n(x,\Z,t).
$$
If $\sum_{n\geq 1}\overline{p}_n < +\infty$ and $\sum_{n\geq 1}\underline{p}_n > 0$, we can introduce $\bar{L}$ and $\underline{L}$ the $\mathbb{N}^*$-valued random variables of law given respectively by $(\overline{p}_n/\sum_{k\geq 1}\overline{p}_k, n\geq 0)$ and $(\underline{p}_n/\sum_{k\geq 1}\underline{p}_k, n\geq 0)$.  
\newline
In this case, we have the following results on the limiting martingale:
\begin{itemize}
\item If $\mathbb{E}\left[\bar{L} \log\left(\bar{L}\right) \right] < +\infty$, then for all initial measure $z \in\Zset$, $\mathbb{E}_{z}\left[W(1)\right] = \langle z,1\rangle.$
\item If $\mathbb{E}\left[\underline{L} \log\left(\underline{L}\right) \right] = +\infty$,
then $W(1) = 0$ almost surely.
\end{itemize}
\end{corollary}
Notice that these conditions are similar to the conditions (16b) and (18b) in \cite{Athreya00}. In the case of constant reproductive law $(p_n, n\geq 0)$, it is well known that these two conditions become a dichotomy \cite{KS66}. Athreya \cite{Athreya00} showed that this dichotomy remains valid for multitype Galton-Watson processes with a finite set of traits and no interactions. 
\begin{remark}
We recall that for non-structured branching processes the mean size $\mathbb{E}\left[N_t\right]$ of the population, where $N_t:= \langle \Z_t,1 \rangle$, is given by
$\mathbb{E}\left[N_t\right] = N_0 \exp\left(\int_0^tB(s)(m(s) -1)\textrm{d}s \right)$.
One may ask under what conditions this exponential growth accurately reflects the rate of increase in the population size.
In this particular case, $W(1)_t/ N_0 = N_t / \mathbb{E}\left[N_t\right]$, thus finding conditions for the non-degeneracy of this martingale gives a direct answer to the question. In the general case with interactions, Corollary \ref{cor: pdmc} gives
$$
W(1)_t\frac{\mathbb{E}\left[N_t\right]}{N_0} = \sum_{u \in \mathbb{G}(t)}e^{-\int_0^tB\left(X^u_s,\Z_s,s\right)\left(m\left(X^u_s,\Z_s,s\right) -1 \right)\textrm{d}s}\mathbb{E}_{\bar{z}}\left[ e^{\int_0^t B\left(Y_s,\Zhat_s,s\right)\left(m\left(Y_s,\Zhat_s,s\right)-1\right)\textrm{d}s}\right],
$$
that is close to $N_t$ if the mean behavior of the spinal individual is not far from the averaging behavior of all the individuals in the branching process with interactions.
\end{remark}

\section{A Yule process with competition}
\label{section: Yule}
In this section, we introduce an alternative application of the spinal construction method that enables us to obtain a spine process with straightforward dynamics from a branching process with intricate interactions. As a toy model, we consider a time-inhomogeneous Yule process with competitive interactions between the individuals, affecting their traits. The individuals in the population are characterized by their trait $x \in \mathbb{R}_+^*$ that can be for example a mass or a size. An individual with trait $x$ divides at an instantaneous rate $r(t)x$ where $r$ is a measurable function on $\mathbb{R}_+$, into two children of size $\Lambda x$ and $(1-\Lambda)x$, where $\Lambda$ is a $[0,1]$-valued random variable with probability density function (p.d.f) $q$. We assume that:
\begin{equation}\label{def Yule param Lambda}
m_{\textrm{div}} := \mathbb{E}\left[\Lambda\right] \in (0,1),\quad \textrm{and } \quad
K_{\textrm{div}} := \mathbb{E}\left[\frac{1}{\Lambda (1-\Lambda)}\right] < +\infty.
\end{equation}
This mass-conservative mechanism of division is classical in cell modeling \cite{fritsch15}.
Moreover, each individual experiences the influence of the whole population, leading to a reduction of their trait. Consequently, at an instantaneous rate $d(t)N_t$ - where $N_t$ is the population size at time $t$, and $d$ is a positive measurable function on $\mathbb{R}_+$ - each individual loses a fraction $(1-\Theta)$ of its size. $\Theta$ is a $[0,1]$-valued random variable with $p$ its p.d.f and we assume that
\begin{equation}\label{def Yule param Theta}
m_{\textrm{loss}} := \mathbb{E}\left[ \Theta \right] \in (0,1) \quad \textrm{and } \quad K_{\textrm{loss}} := \mathbb{E}\left[ \frac{1}{\Theta} \right] < +\infty .
\end{equation}
These events can be interpreted as an inhibition of reproductive material in cells due to competitive interactions within the population.
Finally, we consider that the trait of each individual grows exponentially at an instantaneous rate $\mu(t)$. The dependency in time of these parameters can represent an external control or the effect of a deterministic environment.
This defines a branching process with interactions $(\Z_t, t\geq0)$, whose law is characterized by the infinitesimal generators $(\mathcal{J}^t, t\geq 0)$. For all $t\geq0$, $\mathcal{J}^t$ is defined on the set $\mathcal{D}_J$, where
$$
\mathcal{D}_J := \left\{H_h \in \mathfrak{B}\left(\mathbb{R}_+\right), \exists \left(h,H\right) \in \mathcal{C}^1\left(\mathbb{R}_+,\mathbb{R}\right)\times \mathcal{C}^1\left(\mathbb{R}_+,\mathbb{R}_+^*\right):\ \forall \Z \in \Zhatset \ H_h\left(\Z\right) = H\left(\langle\Z,h\rangle\right)   \right\}
$$
by 
\begin{align*}
\mathcal{J}^tH_h(&\Z) =  H'\left(\langle\Z,h\rangle\right)\int_{\mathbb{R}_+}h'(y)\mu(t)y\Z(\textrm{d}y) \\
&+ \int_{\mathbb{R}_+} r(t) y\int_0^1 \left[ H_h\left(\Z-\delta_y+\delta_{\lambda y} + \delta_{(1-\lambda)y}\right) - H_h\left(\Z\right) \right] q(\lambda)\textrm{d}\lambda \Z(\textrm{d}y)\\
&+ d(t)\langle \Z, 1 \rangle \int_{\mathbb{R}_+}\int_0^1 \left[ H_h\left(\Z-\delta_y+\delta_{\theta y}\right) - H_h\left(\Z\right) \right] p(\theta)\textrm{d}\theta \Z(\textrm{d}y).
\end{align*}
We notice that Assumption \ref{assumption A} is verified by the parameters of this branching process, and thus it does not explode in finite time almost surely. Note that in this population, the dynamics are correlated such that an increase in population size accelerates the rate of loss, while loss events slow the rate of division. 
Even for such a simple model, finding analytic expression of eigenfunctions for the non-local operator $\mathcal{G}$ for all $t$ is a complex task, and the existence of such eigenfunctions is not guaranteed in general, see \cite{Brown80,Coville10}. Furthermore, the $\psi$-spinal process associated with an eigenfunction $\psi$ might have highly intricate dynamics. 

\correct{Here we propose a new approach to spinal constructions: we use the change of measure associated with the spinal construction in order to simplify the dynamics within the spine process. For this model, we exhibit an appropriate function $\psi$ for which the $\psi$-spinal process is a Markov process indexed by an independent binary tree where every branch lives for an exponential time of mean $1$.}
We believe that this method can be generalized to different models. 

We choose $\psi \in \mathcal{C}^1(\mathbb{R}_+^* \times \mathbb{R}\times \mathbb{R}_+,\mathbb{R}_+^*)$ and $\phi \in \mathcal{C}^1(\mathbb{R}_+^*\times \mathbb{R}_+ , \mathbb{R})$ such that for all $ (x,y,t) \in \mathbb{R}_+^* \times \mathbb{R}\times \mathbb{R}_+$
\begin{equation}\label{function psi}
\psi(x,y,t) := xe^{-y} \quad \textrm{and} \qquad \phi(x,t) := \ln\left(xr(t)K_{\textrm{div}}\right).
\end{equation}
Applied to a spinal state $(x_e,\Zhat,t) \in \Zhatset$ where $\Zhat := \sum_{u\in\mathbb{G}(t)}\delta_{x^u}$, this weight function verifies
$$
\psi(x_e,\Zhat,t) = \frac{x_e}{\prod_{u\in\mathbb{G}(t)}(r(t)K_{\textrm{div}}x^u)}
$$
and ensures Assumption \ref{assumption B}.
We can now determine the parameters of the spinal process using this function $\psi$. The behavior of the traits between branching events remains unchanged compared to the Yule process under consideration.
The division events occur at rate $1$ for both the spinal and the non-spinal individuals. The random variable $\widehat{\Lambda}$, which determines the distribution of mass during division in the spinal construction, has a density function $\widehat{q}$ given by:
\begin{equation}\label{q hat}
\widehat{q}(\lambda) := \frac{1}{K_{\textrm{div}}}\frac{q(\lambda)}{\lambda(1-\lambda)}.
\end{equation}
As a result, in this $\psi$-spinal process, division events are no longer dependent on the size of the individuals.

At a rate $\widehat{B}_1(x,\Z,t):=K_{\textrm{loss}}d(t)\langle \Z, 1 \rangle$, individuals outside the spine lose a random fraction $1-\widehat{\Theta}$ of their mass  where $\widehat{\Theta}$ follows a probability density function $\widehat{p}$ given by:
\begin{equation}\label{p hat}
\widehat{p}(\theta) = \frac{1}{K_{\textrm{loss}}}\frac{p(\theta)}{\theta}.
\end{equation}
The loss events for the spinal individual follow the same dynamics as those in the Yule process being considered. It is worth noting that Assumption \ref{assumption C} holds in this case, leading to $\widehat{T}_{\textrm{Exp}} = \infty$ almost surely.
Therefore, by using appropriate function $\psi$ in the spinal construction, we can make division events independent of loss events. \correct{In this case, the spinal process verifies the branching property, moreover it is a piece-wise deterministic Markov process with a distinguished individual indexed on a binary tree with unit branching rate. This $\psi$-spinal process falls within the scope of \cite{BDMT2011}, and the limits theorems within it apply to this process. However, concluding on the long time behavior of the original branching process with interactions is not direct as Theorem \ref{thm LlogL} does not apply for this chosen function $\psi$} 

In the following we show that we can use this property to enhance the simulation complexity of the branching process with interactions. 
A classical exact method to simulate non-homogeneous Poisson processes is the thinning algorithm, introduced by Lewis and Shedler \cite{Lewis79}. It is used to simulate Poisson processes of intensity $c(t)$ on a window $[0,T]$ for a fixed $T>0$. The idea is to generate possible jump times $(t_i, 1 \leq i \leq n)$ at a rate $\bar{c} := \sup_{[0,T]} c(t)$ and accept them with probability $c(t_i)/\bar{c}$. When the intensity $c(\cdot)$
depends not only on time $t$ but also on the entire past of the point process, one can use Ogata’s modified thinning algorithm \cite{Ogata81}. Given the information of the first $k$ points, $(t_i, 1 \leq i \leq k)$ the intensity $c(\cdot)$ is
deterministic on $[t_k, T_{k+1}]$ with $T_{k+1}$ the next time of jump. As a result,
generating the next point in such processes can be considered as generating the first point in an inhomogeneous Poisson process. This idea has been more recently adapted for branching process, see \textit{e.g.} \cite{FM04,CM07,fritsch15}.
The main limitation of this method is that $\bar{c}$ can become excessively large even for small simulation windows $T$. This results in the rejection of most of the generated points. Another exact method, based on inverse transform sampling, consists in generating the arrival times of the process by sampling a uniform random variable $U$ in $[0,1]$, see \textit{e.g.} \cite{Devroye86}. The arrival times $t_k$  are thus generated by inversion of the cumulative distribution function of the jumps times, such that $1-\exp(\int_0^{t_k}c(s)\textrm{d}s) = U$. However, an exact inversion is inaccessible in general cases, and in this model in particular.

Here, we propose a new simulation method, based on the spinal construction, that can be much faster than the Ogata's algorithm. The idea is to use the fact that the division events are independent from the mass in the spine process. Thus, we can first generate a binary tree with unit rate, then draw a realization of the distribution of masses at division and choose the spinal individual, and finally spawn a Poisson point process indexed on this tree that distributes the loss events. Theses three steps are illustrated in Figure \ref{fig:Algo}. At last, the trait of each individual at every time is computed using the deterministic growth and the encountered loss events. A general detailed algorithm can be found in Appendix \ref{appendix algorithm}.

\begin{figure}[htbp]
\centering  
\includegraphics[width=1.0\textwidth]{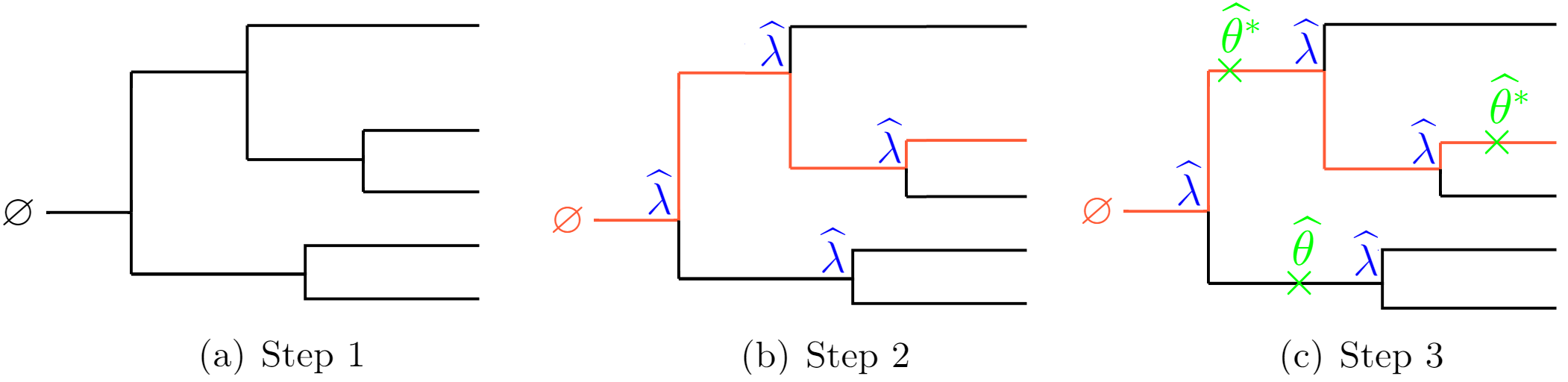}
\caption{Algorithm description}
\label{fig:Algo}
\end{figure}

We then use Theorem \ref{thm:pdmc} to compute various statistics of the branching process based on the statistics obtained from the spine process.
For all $ (x,\Z,t) \in \mathbb{R}_+^* \times \Zset\times \mathbb{R}_+$, the expression of the operator $\mathcal{G}$ applied to the chosen $\psi$, defined in \eqref{function psi}, is given by
\begin{equation*}
\frac{\mathcal{G}\psi\left(x,\Z,t\right)}{\psi\left(x,\Z,t\right)} = \mu(t) + \left(1-\frac{\dot{r}(t)}{r(t)} -\mu(t) + d(t)(N-1)\left(K_{\textrm{loss}}-1\right)\right)N - r(t) B,
\end{equation*}
where for all $\Z \in\Zset$, $N := \langle \Z,1\rangle$ and $B:= \int_{\mathbb{R}_+^*}x\Z(\textrm{d}x)$ denote respectively the size of the population $\Z$ and its total biomass.

\correct{\textbf{Algorithm efficiency}. In order to compare Ogata's method and our spinal method we fix the parameters $r,d$ and $\mu$. The random variables $\Theta$ and $\Lambda$ follow beta laws, respectively $\beta(a,b)$ and $\beta(\alpha,\alpha)$. We will take the values $a=10, b=2$ that correspond to small losses and $\alpha = 20$ to get a narrow symmetrical distribution at division.  With these distributions, the parameters introduced in \eqref{def Yule param Lambda} and \eqref{def Yule param Theta} become
\begin{equation*}
m_{\textrm{div}} = \frac{1}{2},\quad 
K_{\textrm{div}} =4 + \frac{2}{\alpha-1}, \quad m_{\textrm{loss}} = \frac{a}{a+b} \quad \textrm{and } \quad K_{\textrm{loss}} = \frac{a+b-1}{a-1} .
\end{equation*}
We construct two spinal processes using $\psi_1$ and $\psi_2$  such that, for every state $(x_e,\Zhat,t) \in \Zhatset$ where $\Zhat := \sum_{u\in\widehat{\mathbb{G}}(t)}\delta_{x^u}$,
$$
\psi_1(x_e,\Zhat,t) := \frac{x_e}{\prod_{u\in\widehat{\mathbb{G}}(t)}\left(rK_{\textrm{div}}x^u\right)}, \quad \text{and} \quad \psi_2(x_e,\Zhat,t) := \frac{x_e}{\prod_{u\in\widehat{\mathbb{G}}(t)}\left(K_{\textrm{div}}x^u\right)}. 
$$
The dynamics of both spinal processes constructed with these functions $\psi$ are summarized in Figure \ref{fig 1} where we use $r_{\psi} := \mathbbm{1}_{\left\{\psi = \psi_1\right\}} + r\mathbbm{1}_{\left\{\psi = \psi_2\right\}}$.}
\begin{figure}[h!]
{\newcolumntype{P}[1]{>{\centering\arraybackslash}p{#1}}
\renewcommand{\arraystretch}{1.5}
\begin{tabular}{ |P{2cm}|P{2cm}|P{1.5cm}|P{3cm}|P{2.5cm}|  }
 \hline
 & division rate &loss rate& distribution at division &  distribution at loss \\
 \hline
 Spine& $r_{\psi}$  & $d\widehat{N}_t$& $\widehat{\Lambda}^* \sim \beta(\alpha-1,\alpha-1)$ & $ \widehat{\Theta}^* \sim \beta(a,b)$\\
 Non-spinal & $r_{\psi}$  & $K_{\textrm{loss}}d\widehat{N}_t$& $\widehat{\Lambda} \sim \beta(\alpha-1,\alpha-1)$ & $ \widehat{\Theta} \sim \beta(a-1,b)$\\
 \hline
\end{tabular}}
\caption{Dynamics of the $\psi-$spinal processes}
\label{fig 1}
\end{figure}

\correct{In order to benchmark these two spinal methods and the Ogata's method, we estimate the mean size $\bar{x}$ of an individual picked at random in the population at time $t$, starting from a population $\bar{z}$ of size $N_0:=\langle z,1\rangle$ and total mass $B_0:=\int xz(\text{d}x)$. 
For $i\in\{1,2\}$, the $\psi_i$-spinal processes at any time $t$ are denoted by $\Zhat^i_t$ and we denote by $\widehat{N}^i_t := \langle \Zhat^i_t,1\rangle$ and $\widehat{B}^i_t := \int x\Zhat^i_t(\text{d}x)$ the size and total mass of the population $\widehat{\mathbb{G}}^i_t$ at any time $t$. We will also use the following notation for any $t$
$$
\Pi^i_t := \prod_{u\in\widehat{\mathbb{G}}^i_t}X^u_t.
$$
Theorem 2.2 applied with the $\psi$-spinal constructions gives
$$
\bar{x}:=\frac{B_0e^{\mu t}}{\left(r_{\psi_i}K_{\text{div}}\right)^{N_0}} \mathbb{E}_{\bar{z}}\left[\frac{(r_{\psi_i}K_{\text{div}})^{\widehat{N}^i_t}}{\widehat{N}^i_t} \exp\left(C^i_1\int_0^t\widehat{N}^i_s\text{d}s + C_2\int_0^t\left(\widehat{N}^i_s\right)^2\text{d}s-r\int_0^t\widehat{B}^i_s\text{d}s\right)\frac{\Pi^i_t}{\Pi_0}\right] 
$$
where
$$
C^i_1 := r_{\psi_i} - \mu-d\left(K_{\text{loss}}-1\right), \quad \text{and} \quad C_2 := d\left(K_{\text{loss}}-1\right).
$$
We developed python functions, based on the algorithm detailed in Appendix \ref{appendix algorithm}, to generate these trajectories. In this case, the difficulty in the implementation lies in optimizing the computation of the integral terms. We present in Appendix \ref{appendix algorithm} the formula used to compute these integral terms.
Figure \ref{fig:Algo compare} compares the running times $T_{\text{S}}$ of the two spinal methods with the running time $T_{\text{O}}$ of the Ogata's method for the estimation of $\bar{x}$ with $100.000$ sample paths for different sets of parameters. At each line of Figure \ref{fig:Algo compare}, we change one parameter at a time and the remaining parameters are fixed at a base value that changes at each row. The code used to generate this figure is available in \cite{github}.}
\begin{figure}[htbp]
\centering  
\includegraphics[width=1.0\textwidth]{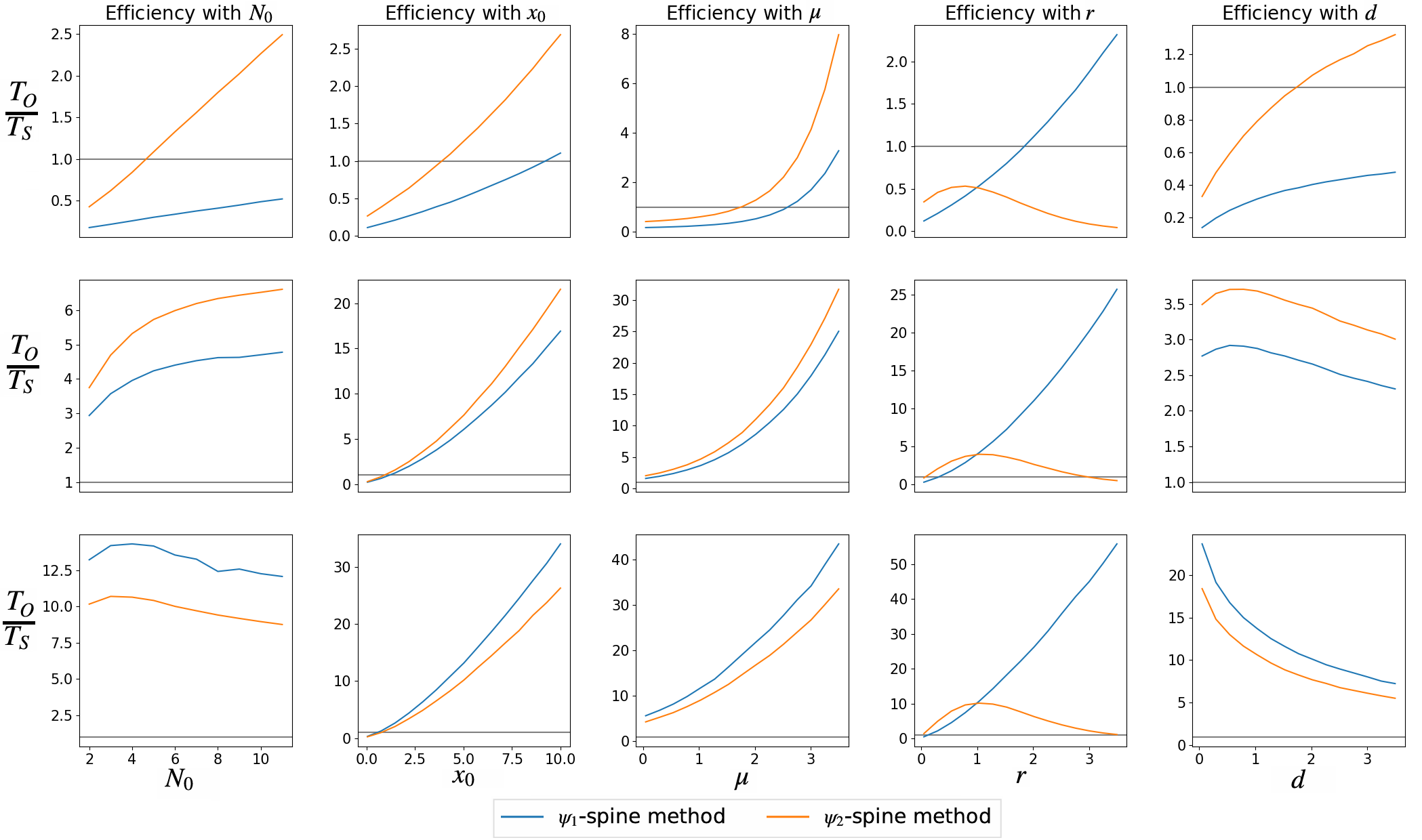}
\caption{Running time comparison}
\medskip
\small
Running time for generating 50000 exact trajectories of the population with the spinal method, normalized by the running time of Ogata's method. At each column, one parameter is modified and at each line, we change the parameters base values: from the first to the third line, $(N_0, x_0 , \mu, r, d)$ are respectively $(2, 0.5, 0.4, 0.4, 0.4)$, $(2, 3, 0.8, 0.8, 0.8)$ and $(2, 6, 1.2, 1.2, 1.2)$. The grey lines indicate a comparable efficiency between spinal and Ogata's methods.
\label{fig:Algo compare}
\end{figure}

\correct{Notice that the spine methods complexity does not depend on the individual traits, unlike Ogata's complexity that grows exponentially with the biomass. This point explains the exponential difference in complexity with the parameters $x_0$ and $\mu$. The evolution of the efficiency with the initial size is due to the fact that the division rate in the spine method does not depend on the size of the population. Also the division rate in the first spinal method is independent of $r$, explaining the evolution of the efficiency with the division parameter $r$. However this is not the case for the second spinal method that has a more intricate dependency on $r$. Finally, notice that the complexities of the spine methods with the loss parameter $d$ are linear, unlike the complexity of the Ogata's method. In fact, the number of generated Poisson points increases with $d$ but the probability of rejection of those points decreases with $d$, resulting in a complex dependency with the parameter $d$. 
Depending on the parameter values, using a spine method to sampling paths instead of the classical Ogata's one, can be exponentially faster.}

\section{Proofs} \label{section: proofs}

\subsection{Proof of Section \ref{section: model}} \label{section: proof model}
In this section, we derive the result on the existence and uniqueness of the considered branching process. 
\begin{proof}[Proof of Proposition \ref{prop: unique original}]
The description of the process $(\Zbar_t, t \geq0)$ introduced in Section \ref{section: model} leads to a canonical SDE driven by the dynamics of the trait and a multivariate point process. This SDE is then used to show that the mean number of individuals in the population and the trait are bounded at any time $t \geq 0$. Using Assumption \ref{assumption A}.2 we conclude that there is no explosion in finite time of the population, \correct{and thus, following the proof in \cite{FM04}, it is sufficient to establish uniqueness and existence of a solution of this SDE by construction.}

We recall that the offspring traits at birth $\boldsymbol{y}=(y^1,\cdots,y^n)$ of an individual of trait $x$ in a population $\Z$ at time $t$ is given by the law $K_n(x,\Z,t,\boldsymbol{y})\mathcal{M}_n(\text{d}\boldsymbol{y})$. A classical way to define the associated SDE is to assume that the random traits depend only on a uniform random variable on $[0,1]$, see \cite{BDMT2011, Marguet19}. However this method correlates the trait at birth between every child. To solve this issue, we will introduce the measure $\mathcal{M}$ on $\mathbb{N}\times\mathcal{X}^{\mathbb{N}}$ such that 
$$
\mathcal{M}(\text{d}n,\text{d}\boldsymbol{y}) := \sum_{i\geq1}\delta_{i}(\text{d}n)\mathcal{M}_i(\text{d}\boldsymbol{y}).
$$

Let $E= \mathbb{R}_+ \times\mathcal{U} \times \mathbb{R}_+ \times \mathbb{N} \times \mathcal{X}^{\mathbb{N}}$ and $Q\left(\textrm{d}s,\textrm{d}u,\textrm{d}r,\textrm{d}n,\textrm{d}\boldsymbol{y} \right)$
be a Poisson point measure  on $E$ with intensity $q$ such that
$$q\left(\textrm{d}s,\textrm{d}u,\textrm{d}r,\textrm{d}n,\textrm{d}\boldsymbol{y} \right) = \textrm{d}s\left(\sum_{v\in\mathcal{U}}\delta_v(\textrm{d}u)\right)\textrm{d}r\mathcal{M}(\text{d}n,\text{d}\boldsymbol{y}).$$
We denote by $\left(\mathcal{F}_t, t\geq 0\right)$ the canonical filtration associated with this Poisson point measure and $\left(T_k, k\geq 0\right)$ the sequence of jumps times, that is the random sequence of times of arrivals given by the Poisson point measure $Q$.

Let $\bar{z} \in \Zbarset$. For every function $g \in \mathcal{C}^1\left(\mathcal{U}\times\mathcal{X},\mathbb{R} \right)$ and $t\geq0$, the process $\left(\Zbar_t, t\geq 0\right)$ starting from $\bar{z}$ verifies 
\begin{multline}\label{PPM original}
\langle\Zbar_t,g\rangle := \ \langle\bar{z},g\rangle + \int_0^t \int_{\mathcal{U}\times\mathcal{X}}\frac{\partial g}{\partial x}(u,x)\cdot \mu\left(x,\Z_s,s\right)\Zbar_s(\textrm{d}u,\textrm{d}x)\textrm{d}s \\
+\int_{E}
\mathbbm{1}_{\left\{u\in\mathbb{G}\left(\Zbar_{s\minus} \right)\right\}}\mathbbm{1}_{\left\{r \leq B_n(X^u_{s\minus},\Z_{s\minus},s)K_n(X^u_{s\minus},\Z_{s\minus},s,\boldsymbol{y})\right\}}\left(\sum_{i=1}^ng\left(ui,y^i\right)
 - g(u,X^u_{s\minus})\right)\\
\times  Q\left(\textrm{d}s,\textrm{d}u,\textrm{d}r,\textrm{d}n,\textrm{d}\boldsymbol{y} \right).
\end{multline}

Note that the dynamical construction of the marginal measure-valued process only does not ensure uniqueness: the individual for the next branching event is chosen according to its trait, and thus two individuals with the same trait can be indistinctly chosen to be the branching one. The labeling of individuals allows us to overcome this problem. However, any other labeling method could work, see \textit{e.g.} \cite{FM04}.
Here we consider non-bounded rates and we follow with small adaptations the proof of Lemma 2.5 in \cite{Marguet19}. Let $T>0$. We prove the non-accumulation of branching events on $[0,T]$. First we use equation \eqref{PPM original} applied to the constant function equal to 1, that gives the number of individuals in the population, denoted $(N_t, t\geq 0)$. Using Assumption \ref{assumption A}.3,  we have for all $t<T_k \wedge T$ 
$$
\mathbb{E}_{\bar{z}}\left[N_t \right] \leq N_0 + \int_0^t b_0(s)\left(\mathbb{E}_{\bar{z}}\Big[\sum_{u\in\mathbb{G}(t)}\left\vert X^u_s\right\vert \Big]+\mathbb{E}_{\bar{z}}\left[N_s \right]\right)\textrm{d}s.
$$
Next, we take a sequence of functions $(g_n,n\in\mathbb{N})$ where for all $n$, $g_n \in \mathcal{C}^1(\mathcal{U}\times\mathcal{X},\mathbb{R}_+)$ and $\lim_{n\rightarrow \infty}g_n(u,x) = \vert x\vert$ and $\lim_{n\rightarrow \infty}\frac{\partial g_n}{\partial x}(u,x)\cdot\mu(x,\Z,s) \leq \vert\mu(x,\Z,s)\vert$ for all $\Z,s$. Applying equation \eqref{PPM original} to these functions and using Assumptions \ref{assumption A}.1 and \ref{assumption A}.4, we have when $n \rightarrow +\infty$,
\begin{align*}
\mathbb{E}_{\bar{z}}\Big[\sum_{u\in\mathbb{G}(t)}\left\vert X^u_t\right\vert \Big] \leq& \langle \bar{z},\vert \cdot \vert\rangle + \int_0^t\mu_0(s)\left(\mathbb{E}_{\bar{z}}\Big[\sum_{u\in\mathbb{G}(s)}\left\vert X^u_s\right\vert \Big] + \mathbb{E}_{\bar{z}}\Big[\Big\vert\sum_{u\in\mathbb{G}(s)} X^u_s\Big\vert \Big] +\mathbb{E}_{\bar{z}}\left[N_s \right]\right)\textrm{d}s. 
\end{align*}
According to Grönwall's lemma, for all $t<T_k \wedge T$, we have
$$
\mathbb{E}_{\bar{z}}\left[N_t \right] + \mathbb{E}_{\bar{z}}\Big[\sum_{u\in\mathbb{G}(t)}\left\vert X^u_t\right\vert \Big] \leq \left(N_0 + \langle \bar{z},\vert \cdot \vert\rangle \right)e^{A(T)t} < \infty
$$
where $A(T) = \sup_{s\leq T} b_0(s)+\mu_0(s)$. The number of individuals is thus almost surely finite at finite time as well as the trait of every individual. Assumption \ref{assumption A}.2 ensures that in finite time, there is no accumulation of branching events that does not change the size of the population. 
\end{proof}
\subsection{Proof of Section \ref{section: Many-to-One formula}}\label{section: proof Many-to-One formula} 
Theorem \ref{thm:pdmc} is proved following the steps of the proof of Theorem 1 in \cite{B21}. Let $\psi \in \mathcal{D}$ \correct{satisfying} Assumption \ref{assumption B}, and $\left(\left(E_t,\Zbarhat_{t}\right), t\geq 0 \right)$ be the time-inhomogeneous $\Zbarhatset$-valued branching process of generators $\left(\widehat{\mathcal{L}}_{\psi}^t, t\geq 0\right)$, defined in \eqref{generator spinal}. We will first need to introduce some notations.
We will denote $(U_k, k\geq 0)$ the sequence of $\mathcal{U}$-valued random variables giving the label of the branching individuals at the jumps times $(T_k, k \geq 0)$. Let $(N_k, k\geq 0)$ be the sequence of $\mathbb{N}$-valued random variables giving the number of children at each branching event and we denote  for brevity $A_k := (U_k,N_k)$ for all $k\geq 0$. At the $k$-th branching time $T_k$, we denote $\mathcal{Y}_k$ the $\mathcal{X}^{N_k}$-valued random variable giving the vector of offspring traits. Finally we introduce for all $k \geq 0$, $\mathcal{V}_k:=\left(T_k,A_k,\mathcal{Y}_k\right)$. We similarly define  $((\widehat{T}_k,\widehat{U}_k,\widehat{N}_k,\widehat{\mathcal{Y}}_k), k \geq 0)$, the sequence of jumps times, labels of the branching individual, number of children and trait of these children at birth in the spinal construction. Notice that the distribution of number of children and traits depend on whether the branching individual is the spinal one or not. We will also use for all $k \geq 0$, $\widehat{\mathcal{V}}_k:=(\widehat{T}_k,\widehat{A}_k,\widehat{\mathcal{Y}}_k)$ where $\widehat{A}_k := (\widehat{U}_k,\widehat{N}_k)$. At time $s \in [\widehat{T}_{k-1},\widehat{T}_k)$, the label of the spinal individual is denoted $E_k$ and its trait $Y_k$. For a given initial population $\bar{z}= \sum_{i=1}^{n}\delta_{(i,x^i)} \in \Zbarset$, we use by convention $U_0=\emptyset, \ N_0=n, \ \mathcal{Y}_0=(x^i, 1 \leq i \leq n)$ almost surely. The same convention holds for the spine process. For all $0 \leq k $, we introduce the associated filtrations 
$$
\mathcal{F}_k = \sigma\left(\mathcal{V}_i, 0\leq i\leq k  \right),\ \textrm{and } \ 
\widehat{\mathcal{F}}_k = \sigma\left(E_k, \left(\widehat{\mathcal{V}}_i, 0\leq i\leq k\right)  \right).
$$
Notice that these notations, summarized in Figure \ref{schema spine}, are well-defined until the explosion time of the branching and spine processes and that the vectors $\left(\mathcal{V}_k, k\geq0  \right)$ where for all $k \geq 0$, $\mathcal{V}_k:=\left(T_k,A_k,\mathcal{Y}_k\right)$ characterize the trajectories of $(\Zbar_t, t\geq0)$ until the explosion time.

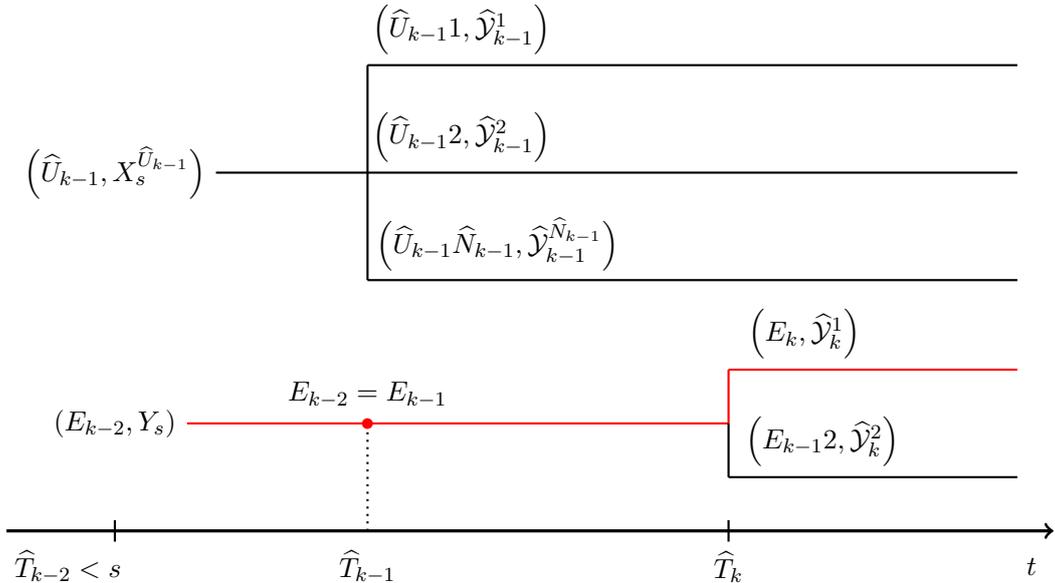
\begin{figure}[htbp]
\centering
\captionsetup{justification=centering}
\begin{tikzpicture}[thick,scale=0.95, every node/.style={scale=0.95}]
\draw[->][very thick] (0,0) -- (14.5,0);
\draw node at (1.5,5) {$\Big(\widehat{U}_{k-1},X^{\widehat{U}_{k-1}}_s\Big)$};
\draw node at (1.5,1.5) {$\left(E_{k-2},Y_s\right)$};
\draw (1.5,-0.15) -- (1.5, 0.15);
\draw node at (0.85,-0.5) {$\widehat{T}_{k-2} < s$};
\draw[red,thick] (2.5,1.5) -- (10,1.5);
\draw node at (5,-0.5) {$\widehat{T}_{k-1}$};
\draw[dotted] (5,0) -- (5, 1.55);
\draw[red] node at (5,1.49) {$\bullet$};
\draw node at (5,1.9) {$E_{k-2}=E_{k-1}$};
\draw[thick] (2.9,5) -- (5,5);

\draw node at (6.3,7) {$\left(\widehat{U}_{k-1}1, \widehat{\mathcal{Y}}_{k-1}^1\right)$};
\draw node at (6.3,5.5) {$\left(\widehat{U}_{k-1}2, \widehat{\mathcal{Y}}_{k-1}^2\right)$};
\draw node at (6.8,4) {$\Big(\widehat{U}_{k-1}\widehat{N}_{k-1}, \widehat{\mathcal{Y}}_{k-1}^{\widehat{N}_{k-1}}\Big)$};
\draw[thick] (5,3.5) -- (5,6.5);
\draw[thick] (5,6.5) -- (14,6.5);
\draw[thick] (5,5) -- (14,5);
\draw[thick] (5,3.5) -- (14,3.5);

\draw[red,thick] (10,2.25) -- (10,1.5);
\draw[thick] (10,0.75) -- (10,1.5);
\draw[red,thick] (10,2.25) -- (14,2.25);
\draw[thick] (10,0.75) -- (14,0.75);
\draw node at (11.05,2.75) {$\left(E_k, \widehat{\mathcal{Y}}_{k}^1\right)$};
\draw node at (11.3,1.25) {$\left(E_{k-1}2, \widehat{\mathcal{Y}}_{k}^2\right)$};

\draw (10,-0.15) -- (10,0.15);
\draw node at (10,-0.5) {$\widehat{T}_k$};
\draw node at (14.2,-0.5) {$t$};
\end{tikzpicture}
\caption{Sequential notations for the spine process.}\label{schema spine}
\end{figure}
For every initial population $\bar{z} \in \Zbarset$ and every $k\geq0$, we introduce the set of sequences of $k$ branching events that lead to non-extinguished trajectories $\mathfrak{U}_k(\bar{z}) \subset (\mathcal{U} \times \mathbb{N})^k$, starting from $\bar{z}$.
We also introduce for all $a \in \mathfrak{U}_k(\bar{z})$ and all $0 \leq i \leq k, \  \mathbb{G}_i(a)$ the set of labels of individuals living between the $i$-th and $(i+1)$-th branching events in the population where all the branching events were given by $a$. 
By decomposing the branching process $(\Zbar_t,t\geq0)$ based on the sequences $a \in \mathfrak{U}_k(\overline{z})$ and sampling at time $t$ an individual $e \in \mathbb{G}_k(a)$, we get that for every \correct{non-negative measurable} function $H$  on $\mathcal{U} \times \mathbb{D}\left([0,t],\Zset\right)$
\begin{multline}\label{eq: decomposition }
\mathbb{E}_{\bar{z}}\left[\mathbbm{1}_{\{T_{\textrm{Exp}} > t, \mathbb{G}(t) \neq \emptyset\}}H\left(U_t,(\Zbar_s,s\leq t)\right) \right] =\\ \sum_{k \geq 0}\sum_{a \in \mathfrak{U}_k(\bar{z})}\sum_{ e \in \mathbb{G}_k(a)} \mathbb{E}_{\bar{z}}\left[H\left(e,(\Zbar_s,s\leq t)\right)p_e(\Zbar_t)\mathbbm{1}_{\{ T_k \leq t < T_{k+1}\}}\prod_{i=0}^{k}\mathbbm{1}_{\{A_i = a_i\}} \right].
\end{multline}

The expectation on the right-hand side of \eqref{eq: decomposition } is linked by a Girsanov-type result to the  spinal construction, as shown in Lemma \ref{lemma induction}. The difference with our proof and the proof of Theorem 1 in \cite{B21} lies essentially in the demonstration of this lemma.  
\begin{lemma}\label{lemma induction}
$\left. \right.$\\
For any $k >0$, $\bar{z} \in \Zbarset$ and  $a = ((u_i,n_i), 0 \leq i \leq k) \in \mathfrak{U}_k(\bar{z})$, let $F$ be a \correct{non-negative measurable} function on $\Pi_{i=1}^k\left(\mathbb{R}_+ \times \mathcal{U}\times \mathbb{N} \times \mathcal{X}^{n_i}\right)$. For any $e \in \mathbb{G}_{k}\left(a\right)$,
\begin{multline*}\label{induction equation}
\mathbb{E}_{\bar{z}}  \left[F\left(\mathcal{V}_i, 0 \leq i \leq k\right)\prod_{i=0}^{k}\mathbbm{1}_{\{A_i = a_i\}}\right]= \left\langle z, \psi\left(\cdot,z,0\right) \right\rangle\\
\times \mathbb{E}_{\bar{z}} \Bigg[
\xi_k\left(E_k,\left(\widehat{\mathcal{V}}_i, 0 \leq i \leq k\right) \right) 
\left.F\left(\widehat{\mathcal{V}}_i, 0 \leq i\leq k\right) \mathbbm{1}_{\{E_{k} = e\}} \prod_{i=0}^{k}\mathbbm{1}_{\{\widehat{A}_i = a_i\}} \right],
\end{multline*}
where 
\begin{equation*}
\xi_k\left(E_k,\left(\widehat{\mathcal{V}}_i, 0 \leq i \leq k\right) \right) := \frac{1}{\psi\left(Y_{\widehat{T}_k},\Zhat_{\widehat{T}_k},\widehat{T}_k\right)}\exp\left(\int_{0}^{\widehat{T}_{k}}\frac{\mathcal{G}\psi\left(Y_s,\Zhat_s,s\right)}{\psi\left(Y_s,\Zhat_s,s\right)}\textrm{d}s\right).
\end{equation*}
\end{lemma}
We prove this lemma by induction on the number of branching events $k$. We will first state a technical lemma to avoid too long computations. 
\begin{lemma}\label{chain rule}
For all $k >0$, $\bar{z} \in \Zbarset$ and  $a = ((u_i,n_i), 0 \leq i \leq k) \in \mathfrak{U}_k(\bar{z})$,
\begin{equation*}
\frac{\xi_k\left(E_k,\left(\widehat{\mathcal{V}}_i, 0 \leq i \leq k\right) \right)}{\xi_{k-1}\left(E_{k-1},\left(\widehat{\mathcal{V}}_i, 0 \leq i \leq k-1\right) \right)} = \frac{\psi\left(Y_{\widehat{T}_k^{\minus}},\Zhat_{\widehat{T}_k^{\minus}},\widehat{T}_k\right)}{\psi\left(Y_{\widehat{T}_k},\Zhat_{\widehat{T}_k},\widehat{T}_k\right)}\exp\left(\int_{\widehat{T}_{k-1}}^{\widehat{T}_{k}}\lambda\left(Y_s,\Zhat_s,s\right)\textrm{d}s\right) \quad \textrm{a.s.}
\end{equation*} 
where, for all $s \in [\widehat{T}_{k-1},\widehat{T}_k)$,
\begin{equation}\label{eq lambda}
\lambda\left(Y_s,\Zhat_s,s\right) := \widehat{\tau}_{tot}\left(Y_s,\Zhat_s,s\right) - \int_{\mathcal{X}}B\left(x,\Zhat_s,s\right)\Zhat_s(\textrm{d}x).
\end{equation}
\end{lemma}
We recall that $\widehat{\tau}_{tot}$ is the total branching rate of the spine process defined in \eqref{tau hat tot}. The proof of this lemma is in Appendix \ref{appendix chain rule}.
\begin{proof}[Proof of Lemma \ref{lemma induction}]
Following the proof of Theorem 1 in \cite{B21}, the result is established by induction on the number $k$ of branching events. The original branching process and its associated spine process may stop branching in finite time, in this case the total numbers of branching events, respectively $N_{\textrm{tot}}$ and $\widehat{N}_{\textrm{tot}}$, are finite. For all $k$ such that $k>N_{\textrm{tot}}$, the $k$-th branching event of the original construction arrives at $T_k = +\infty$. In this case we set $A_k = (\emptyset,0)$ by convention. The same convention is used for the spinal construction.
Let $\bar{z} := \sum_{i=1}^{n}\delta_{(i,x^i)} \in\Zbarset$ be the initial population and let $F$ be a \correct{non-negative measurable} function on $\mathbb{R}_+ \times \mathcal{U}\times \mathbb{N} \times \mathcal{X}^{n}$. 
Then, by definition, $\mathcal{V}_0 = \widehat{\mathcal{V}}_0 = (0,(\emptyset,n),(x^i, 1 \leq i \leq n))$ almost surely.

Therefore, for all $e \in \mathbb{G}(0)$
\begin{align*}
\mathbb{E}_{\bar{z}}\left[\xi_0\left(E_{0},\widehat{\mathcal{V}}_0\right)F\left(\widehat{\mathcal{V}}_0\right) \mathbbm{1}_{\{E_{0} = e\}}\mathbbm{1}_{\{\widehat{A}_0 = (\emptyset,n)\}}\right]&=  F\left(\mathcal{V}_0\right)\mathbbm{1}_{\{A_0 = (\emptyset,n)\}}\mathbb{E}_{\bar{z}}\left[\frac{\mathbbm{1}_{\{E_{0} = e\}}}{\psi(Y_0,z,0)} \right].
\end{align*} 
We recall that the individual $1 \leq i \leq n$ of trait $x^i$ in the population $\bar{z}$ is chosen to be the spinal individual with probability $\psi(x^i,z,0)(\langle z, \psi(\cdot,z,0) \rangle)^{-1}$. Then we have 
$$
\mathbb{E}_{\bar{z}}\left[\xi_0\left(E_{0},\widehat{\mathcal{V}}_0\right)F\left(\widehat{\mathcal{V}}_0\right) \mathbbm{1}_{\{E_{0} = e\}}\mathbbm{1}_{\{\widehat{A}_0 = (\emptyset,n)\}}\right]=  \frac{1}{\langle z, \psi\left(\cdot,z,0\right) \rangle}\mathbb{E}_{\bar{z}}\left[\mathbbm{1}_{\{A_0 = (\emptyset,n)\}}F\left(0,A_0,\mathcal{Y}_0\right)\right].
$$
Thus the result holds for $k =0$. Now let $k \geq 1$ and assume that the following induction hypothesis holds at rank $k-1$.\newline
\textbf{Induction Hypothesis.}\textit{
For every $a=\left(a_i, 0 \leq i\leq k-1\right)\in \left(\mathcal{U}\times \mathbb{N}\right)^{k-1}$ with $a_i = (u_i,n_i)$, every \correct{non-negative measurable} function $F$ on $\underset{i=1}{\overset{k-1}{\bigotimes}}\left(\mathbb{R}_+ \times \mathcal{U}\times \mathbb{N} \times \mathcal{X}^{n_i}\right)$ and every label $e \in \mathbb{G}_{k-1}\left(a\right)$:
\begin{multline}\label{eq: induction n-1}
\mathbb{E}_{\bar{z}}\left[F\left(\mathcal{V}_i, 0 \leq i \leq k-1\right)\prod_{i=0}^{k-1}\mathbbm{1}_{\{A_i = a_i\}}\right]= \left\langle z, \psi\left(\cdot,z\right) \right\rangle \\ 
\times\mathbb{E}_{\bar{z}} \left[
\xi_{k-1}\left(E_{k-1},\left(\widehat{\mathcal{V}}_i, 0 \leq i \leq k-1 \right)\right)   
F\left(\widehat{\mathcal{V}}_i, 0 \leq i\leq k-1\right) 
 \mathbbm{1}_{\{E_{k-1} = e\}}\prod_{i=0}^{k-1}\mathbbm{1}_{\{A_i = a_i\}} \right].
\end{multline}
}
Let $a = \left(a_i, 0 \leq i\leq k\right) \in \left(\mathcal{U}\times \mathbb{N}\right)^n$ with $a_i = (u_i,n_i)$ and $e \in \mathbb{G}_{k}\left(a\right)$. We denote $a'=(a_i, 0 \leq i \leq k-1)$ and take $F_k^a$ a \correct{non-negative measurable} function on $\underset{i=1}{\overset{k}{\bigotimes}}\left(\mathbb{R}_+ \times \mathcal{X}^{n_i}\right)$ such that, for all $\left(\left(t_i,y_i\right), 0 \leq i \leq k\right) \in \underset{i=1}{\overset{n}{\bigotimes}}\left(\mathbb{R}_+ \times \mathcal{X}^{n_i}\right)$:
\begin{equation}\label{eq: Gka}
 F_k^a\left(\left(t_i,y_i\right), 0 \leq i \leq n\right) := F^{a'}_{k-1}\left(\left(t_i,y_i\right), 0 \leq i \leq k-1\right)I(t_k-t_{k-1})F(y_k),
\end{equation}  
where $F^{a'}_{k-1}, I$ and $F$ are \correct{non-negative measurable} and bounded functions, respectively on $\underset{i=1}{\overset{k-1}{\bigotimes}}\left(\mathbb{R}_+ \times \mathcal{X}^{n_i}\right)$, $\mathbb{R}_+$ and $\mathcal{X}^{n_k}$. 
\correct{We denote by $e'\in \mathbb{G}_{k-1}\left(a'\right)$ the deterministic label of the ancestor of $e$ just before the $k$-th jump. Thus there exists a unique $j\in\mathbb{N}^*\cup\{\emptyset\}$ \correct{satisfying} $e'j=e$. If $j\neq\emptyset$, it means that $e'$ is the branching individual in the $k$-th jump, otherwise it is any other individual in the population.}  We introduce the $\mathbb{N}^*\cup\{\emptyset\}$-valued \correct{random variable} $J_k$ choosing the label of the spinal individual at the $k$-th branching event, so that $E_k = E_{k-1}J_{k}$ almost surely.

Following the proof of Lemma 1 in \cite{B21}, we express both sides of the equality \eqref{eq: induction n-1} for $k \geq 1$ conditionally on the filtrations at the previous step $k-1$. We recall that, for all $1 \leq j \leq k$
$$
\mathcal{F}_j = \sigma\left(\mathcal{V}_i, 0\leq i\leq j \right),\ \textrm{and } \ 
\widehat{\mathcal{F}}_j = \sigma\left(E_j, \left(\widehat{\mathcal{V}}_i, 0\leq i\leq j\right)  \right).
$$ 
Using \eqref{eq: Gka} and conditioning on the filtration after the $(k-1)$-th branching event, we have
\begin{multline}\label{eq rec original}
\mathbb{E}_{\bar{z}} \left[F^a_k\left(\left(T_i,\mathcal{Y}_i\right), 0 \leq i \leq k\right)\prod_{i=0}^{k}\mathbbm{1}_{\{A_i = a_i\}}\right] = \\ \mathbb{E}_{\bar{z}}\left[F^{a'}_{k-1}\left(\left(T_i,\mathcal{Y}_i\right), 0 \leq i \leq k-1\right)C\left(\mathcal{V}_i, 0 \leq i \leq k-1\right)\prod_{i=0}^{k-1}\mathbbm{1}_{\{A_i = a_i\}} \right], 
\end{multline}
where:
\begin{equation}\label{def C}
C\left(\mathcal{V}_i, 0 \leq i \leq k-1\right) := \mathbb{E}\left[ I\left(T_k - T_{k-1} \right)F(\mathcal{Y}_k)\mathbbm{1}_{\{A_k = a_k\}} \big\vert \mathcal{F}_{k-1} \right].
\end{equation}
We apply the induction hypothesis \eqref{eq: induction n-1} in the previous equation \eqref{eq rec original} with $H$ given by
\begin{equation*}
H\left(\mathcal{V}_i, 0 \leq i \leq k-1\right) := F^{a'}_{k-1}\left(\left(T_i,\mathcal{Y}_i\right), 0 \leq i \leq k-1\right)C\left(\mathcal{V}_i, 0 \leq i \leq k-1\right).
\end{equation*}
Thus 
\begin{multline}\label{eq branching}
\mathbb{E}_{\bar{z}}\left[F^a_k\left(\left(T_i,\mathcal{Y}_i\right), 0 \leq i \leq k\right)\prod_{i=0}^{k}\mathbbm{1}_{\{A_i = a_i\}}\right]= \left\langle z, \psi\left(\cdot,z,0\right) \right\rangle \\
\times \mathbb{E}_{\bar{z}} \Bigg[
\xi_{k-1}\left(E_{k-1},\left(\widehat{\mathcal{V}}_i, 0 \leq i \leq k-1 \right)\right)  F^{a'}_{k-1}\left(\left(\widehat{T}_i,\widehat{\mathcal{Y}}_i\right), 0 \leq i \leq k-1\right) \\
\times C\left(\widehat{\mathcal{V}}_i, 0 \leq i\leq k-1\right) 
 \mathbbm{1}_{\{E_{k-1} = e'\}}\prod_{i=0}^{k-1}\mathbbm{1}_{\{A_i = a_i\}} \Bigg].
\end{multline}

Similarly, we use \eqref{eq: Gka} for the spine process. Conditioning on the filtration after the $(k-1)$-th branching event, we have
\begin{multline}\label{eq induction}
\mathbb{E}_{\bar{z}} \Bigg[ 
\xi_k\left(E_{k},\left(\widehat{\mathcal{V}}_i, 0 \leq i \leq k \right) \right)
F^a_k\left(\left(\widehat{T}_i,\widehat{\mathcal{Y}}_i\right), 0 \leq i\leq k\right) \mathbbm{1}_{\{E_{k} = e\}} \prod_{i=0}^k\mathbbm{1}_{\{\widehat{A}_i = a_i\}} \Bigg]=\\
 \mathbb{E}\Big[ \xi_{k-1}\left(E_{k-1},\left(\widehat{\mathcal{V}}_i, 0 \leq i \leq k-1 \right) \right)
F^{a'}_{k-1}\left(\left(\widehat{T}_i,\widehat{\mathcal{Y}}_i\right), 0 \leq i\leq k-1\right) \\
\times \widehat{C}\left(e',\left(\widehat{\mathcal{V}}_i, 0 \leq i \leq k-1 \right) \right)\prod_{i=0}^{k-1}\mathbbm{1}_{\{\widehat{A}_i = a_i\}}\Big], 
\end{multline}
where, using Lemma \ref{chain rule},
\begin{multline}\label{def C hat}
\widehat{C}\left(e',\left(\widehat{\mathcal{V}}_i, 0 \leq i \leq k-1 \right) \right) :=  \mathbbm{1}_{\{E_{k-1} = e'\}}\\
\times\mathbb{E}\left[\frac{\psi\left(Y_{\widehat{T}_k^-},\Zhat_{\widehat{T}_k^-},\widehat{T}_k\right)}{\psi\left(Y_{\widehat{T}_k},\Zhat_{\widehat{T}_k},\widehat{T}_k\right)}\exp\left(\int_{\widehat{T}_{k-1}}^{\widehat{T}_{k}}\lambda\left(Y_s,\Zhat_s,s\right)\textrm{d}s\right) \mathbbm{1}_{\{J_{k} = j\}}I\left(\widehat{T}_k -\widehat{T}_{k-1} \right) \right. \\
\times F\left(\widehat{\mathcal{Y}}_k\right)\mathbbm{1}_{\{\widehat{A}_k = a_k\}} \left| \widehat{\mathcal{F}}_{k-1}\right. \Bigg].
\end{multline}
If we can establish that
\begin{equation}\label{eq: C C hat}
\widehat{C}\left(e',\left(\widehat{\mathcal{V}}_i,0\leq i\leq k-1\right)\right)= \mathbbm{1}_{\{E_{k-1}=e'\}} C \left(\left(\widehat{\mathcal{V}}_i,0\leq i\leq k-1\right)\right),
\end{equation}
then the expectations in the right-hand sides of equations \eqref{eq branching} and \eqref{eq induction} are equal and we get 
\begin{multline*}
\mathbb{E}_{\bar{z}}\left[F^a_k\left(\left(T_i,\mathcal{Y}_i\right), 0 \leq i \leq k\right)\prod_{i=0}^{k}\mathbbm{1}_{\{A_i = a_i\}}\right]= \left\langle z, \psi\left(\cdot,z,0\right) \right\rangle \\
\times\mathbb{E}_{\bar{z}} \Bigg[ 
\xi_k\left(E_{k},\left(\widehat{\mathcal{V}}_i, 0 \leq i \leq k \right) \right)
F^a_k\left(\left(\widehat{T}_i,\widehat{\mathcal{Y}}_i\right), 0 \leq i\leq k\right) \mathbbm{1}_{\{E_{k} = e\}} \prod_{i=0}^k\mathbbm{1}_{\{\widehat{A}_i = a_i\}} \Bigg].
\end{multline*}
From this equality and using a monotone class argument for the functions $F_k^a$ defined in \eqref{eq: Gka}, we obtain \eqref{eq: induction n-1} at rank $k$, that concludes the proof.

In the following we show \eqref{eq: C C hat}, by describing the dynamics of both processes.\newline
\textbf{Computations for the branching process.} Conditioning in chain the expression of $C$ defined in \eqref{def C}, we get
\begin{multline*}
C\left(\mathcal{V}_i, 0 \leq i \leq k-1\right) = \\
\mathbb{E}\Big[ I(T_k - T_{k-1})\mathbb{E}\Big[\mathbbm{1}_{\{A_k = a_k\}}\mathbb{E}\big[F(\mathcal{Y}_k)\big\vert \mathcal{F}_{k-1}, T_k,A_k \big]\Big\vert \mathcal{F}_{k-1}, T_k \Big]\Big\vert \mathcal{F}_{k-1}\Big].
\end{multline*}
Using the conditional distribution of $\mathcal{Y}_n$, we have
$$
\mathbb{E}\left[F(\mathcal{Y}_k)\big\vert \mathcal{F}_{k-1}, T_k,A_k \right]=
 \int_{\mathcal{X}^{N_k}} F(\boldsymbol{y}) K_{N_k}\left(X^{U_k}_{T_k^-},\Z_{T_k^-},T_k,\boldsymbol{y}\right)\mathcal{M}_{N_k}(\text{d}\boldsymbol{y}).
$$
Then we notice that for $a_k = (u_k,n_k)$,
$$
\mathbb{E}\left[\mathbbm{1}_{\{A_k = a_k\}} \big\vert \mathcal{F}_{k-1}, T_k \right] = \frac{B_{n_k}\left(X^{u_k}_{T_k^-},\Z_{T_k^-},T_k\right)}{\tau\left(\Z_{T_k^-},T_k\right)},
$$ 
where $\tau\left(\Z,s\right) := \int_{\mathcal{X}}B(x,\Z,s)\Z(\textrm{d}x)$ is the total branching rate.
Finally, using that the time between two jumps follows an inhomogeneous exponential law of instantaneous rate $\tau(\cdot)$, we have
\begin{multline*}
C\left(\mathcal{V}_i, 0 \leq i \leq k-1\right) = \int_{T_{k-1}}^{+\infty} I(t-T_{k-1})\exp\left(-\int_{T_{k-1}}^{t}\tau\left(\Z_s,s\right)\textrm{d}s\right)\\ \times  B_{n_k}\left(X^{u_k}_{t},\Z_{t},t\right)\int_{\mathcal{X}^{n_k}} F(\boldsymbol{y})K_{n_k}\left(X^{u_k}_{t},\Z_{t},t,\boldsymbol{y}\right)\mathcal{M}_{n_k}(\text{d}\boldsymbol{y}) \textrm{d}t.
\end{multline*}
\newline
\textbf{Computations for the spinal construction.}
We follow the computations for the branching process, and \correct{distinguish} between whether the branching individual is the spinal one or not.
Conditioning on the next jump time, $\widehat{C}$ defined in \eqref{def C hat}, is such that
\begin{multline}\label{eq C hat}
\widehat{C}\left(e',\left(\widehat{\mathcal{V}}_i,0\leq i\leq k-1\right)\right) =  \mathbbm{1}_{\{E_{k-1} = e'\}}\mathbb{E}\left[ \left.\exp\left(\int_{\widehat{T}_{k-1}}^{\widehat{T}_{k}}\lambda\left(Y_s,\Zhat_s,s\right)\textrm{d}s\right)I\left(\widehat{T}_k - \widehat{T}_{k-1} \right)\right. \right. \\
\times  \left.\left.\psi\left(Y_{\widehat{T}_k^-},\Zhat_{\widehat{T}_k^-},\widehat{T}_k\right)\mathbb{E}\left[ \left.\frac{F\left(\widehat{\mathcal{Y}}_k\right)\mathbbm{1}_{\{J_{k} = j\}}\mathbbm{1}_{\{\widehat{A}_k = a_k\}}}{\psi\left(Y_{\widehat{T}_k},\Zhat_{\widehat{T}_k},\widehat{T}_k\right)} \right| \widehat{\mathcal{F}}_{k-1},\widehat{T}_k \right]\right| \widehat{\mathcal{F}}_{k-1} \right].
\end{multline}
We handle this last equation by complete induction, on the event $\{j=\emptyset \} \cup \{j \neq \emptyset\}$, to show that 
\begin{multline}\label{eq exhasution}
\psi\left(Y_{\widehat{T}_k^-},\Zhat_{\widehat{T}_k^-},\widehat{T}_k\right)\mathbb{E}\left[ \left.\frac{F\left(\widehat{\mathcal{Y}}_k\right)\mathbbm{1}_{\{J_{k} = j\}}\mathbbm{1}_{\{\widehat{A}_k = a_k\}}}{\psi\left(Y_{\widehat{T}_k},\Zhat_{\widehat{T}_k},\widehat{T}_k\right)} \right| \widehat{\mathcal{F}}_{k-1},\widehat{T}_k \right] = \\
\frac{B_{n_k}\left(X^{u_k}_{T_k^-},\Zhat_{\widehat{T}_k^-},\widehat{T}_k\right)}{\widehat{\tau}_{\textrm{tot}}\left(\Zhat_{\widehat{T}_k^-},\widehat{T}_k\right)} \int_{\mathcal{X}^{\widehat{N}_k}} 
F\left(\boldsymbol{y}\right)K_{\widehat{N}_k}\left(X^{\widehat{U}_k}_{\widehat{T}_k^-},\Zhat_{\widehat{T}_k^-},\widehat{T}_k,\boldsymbol{y}\right)\mathcal{M}_n(\text{d}\boldsymbol{y}).
\end{multline}

\textit{Branching outside the spine.}
If $j=\emptyset$, then the branching individual is not the spinal one, and $e = e'j = e'$ . We follow the same conditioning than for the branching process, and use the fact that on the event $\{j=\emptyset\}$ the trait of the spinal individual is $\widehat{\mathcal{F}}_{k-1}$-measurable. Using the expression of the $\psi$-biased distribution $\widehat{K}$ of $\widehat{\mathcal{Y}}_k$, defined in \eqref{def: K hat}, we get
\begin{multline*}
\mathbb{E}\left[ \left.\left.\frac{F\left(\widehat{\mathcal{Y}}_k\right)\mathbbm{1}_{\{J_{k} = \emptyset\}}\mathbbm{1}_{\{\widehat{A}_k = a_k\}}}{\psi\left(Y_{\widehat{T}_k},\Zhat_{\widehat{T}_k},\widehat{T}_k\right)} \right| \widehat{\mathcal{F}}_{k-1},\widehat{T}_k \right]= \mathbb{E}\right[\mathbbm{1}_{\{\widehat{A}_k = a_k\}}\mathbbm{1}_{\{J_{k} = \emptyset\}} \\
\left.\left.\times \int_{\mathcal{X}^{\widehat{N}_k}} 
\frac{F\left(\boldsymbol{y}\right)}{\widehat{\Gamma}_{n_k}\left(Y_{\widehat{T}_k^-},X^{u_k}_{\widehat{T}_k^-},\Zhat_{\widehat{T}_k^-},\widehat{T}_k\right)}K_{n_k}\left(X^{\widehat{U}_k}_{\widehat{T}_k^-},\Zhat_{\widehat{T}_k^-},\widehat{T}_k,\boldsymbol{y}\right)\mathcal{M}_n(\text{d}\boldsymbol{y})\right| \widehat{\mathcal{F}}_{k-1},\widehat{T}_k \right].
\end{multline*}
We then recall the distribution of $\widehat{A}_k$ outside the spine, established in Dynamics \ref{spinal outside rates} :
$$
\mathbb{E}\left[\mathbbm{1}_{\{\widehat{A}_k = a_k\}}\mathbbm{1}_{\{J_{k} = \emptyset\}} \big\vert \widehat{\mathcal{F}}_{k-1}, \widehat{T}_k \right] = \frac{\widehat{\Gamma}_{n_k}\left(Y_{\widehat{T}_k^-},X^{u_k}_{\widehat{T}_k^-},\Zhat_{\widehat{T}_k^-},\widehat{T}_k\right)}{\psi\left(Y_{\widehat{T}_k^-},\Zhat_{\widehat{T}_k^-},\widehat{T}_k \right)} \frac{B_{n_k}\left(X^{u_k}_{T_k^-},\Zhat_{\widehat{T}_k^-},\widehat{T}_k\right)}{\widehat{\tau}_{\textrm{tot}}\left(\Zhat_{\widehat{T}_k^-},\widehat{T}_k\right)}.
$$ 

This gives \eqref{eq exhasution} on the event $\{j=\emptyset\}$.

\textit{Spine branching.} We follow the same computations when $j \neq \emptyset$, that corresponds to the case when the branching individual is the spinal one, i.e $u_k = e'$. In this case, the distribution of $Y_{\widehat{T}_k}$ now depends on the traits $\widehat{\mathcal{Y}}_k$. Thus conditioning on $\widehat{\mathcal{Y}}_k$ and using the distribution of the next spinal individual, defined in \eqref{spine proba}, we have 
\begin{multline*}
\mathbbm{1}_{\{j \neq \emptyset\}}\mathbb{E}\left[ \left.\left.\frac{F\left(\widehat{\mathcal{Y}}_k\right)\mathbbm{1}_{\{J_{k} = j\}}\mathbbm{1}_{\{\widehat{A}_k = a_k\}}}{\psi\left(Y_{\widehat{T}_k},\Zhat_{\widehat{T}_k},\widehat{T}_k\right)} \right| \widehat{\mathcal{F}}_{k-1},\widehat{T}_k \right]=\mathbb{E}\right[ \mathbbm{1}_{\{\widehat{A}_k = (e',n_k)\}}\\
\times\left.\left.\mathbb{E}\left[ \left.   \frac{F\left(\widehat{\mathcal{Y}}_k\right)}{ \sum_{i=1}^{\widehat{N}_k}\psi\Big(\widehat{\mathcal{Y}}_k^i,\Zhat_{\widehat{T}_k^-}-\delta_{Y_{\widehat{T}_k^-}}+\sum_{l=1}^{\widehat{N}_k}\delta_{\widehat{\mathcal{Y}}_k^l},\widehat{T}_k\Big) } \right| \widehat{\mathcal{F}}_{k-1},\widehat{T}_k,\widehat{A}_k\right]\right| \widehat{\mathcal{F}}_{k-1},\widehat{T}_k \right].
\end{multline*}
We then use the distribution $\widehat{K}^*$ of $\widehat{\mathcal{Y}}_k$, defined in \eqref{def: K star hat} when the branching individual is the spinal one, conditioning on $\widehat{A}_k$, defined in Dynamics \ref{spine rates}.
\begin{multline*}
\mathbbm{1}_{\{\widehat{A}_k = (e',n_k)\}}\mathbb{E}\left[ \left. \frac{F\left(\widehat{\mathcal{Y}}_k\right)}{ \sum_{i=1}^{\widehat{N}_k}\psi\Big(\widehat{\mathcal{Y}}_k^i,\Zhat_{\widehat{T}_k^-}-\delta_{Y_{\widehat{T}_k^-}}+\sum_{l=1}^{\widehat{N}_k}\delta_{\widehat{\mathcal{Y}}_k^l},\widehat{T}_k\Big) } \right| \widehat{\mathcal{F}}_{k-1},\widehat{T}_k,\widehat{A}_k\right]=\\
\mathbbm{1}_{\{\widehat{A}_k=(e',n_k)\}}\int_{\mathcal{X}^{n_k}} \frac{F\left(\boldsymbol{y}\right)}{\widehat{\Gamma}^*_{n_k}\left(Y_{\widehat{T}_k^-},\Zhat_{\widehat{T}_k^-},\widehat{T}_k\right) }K_{n_k}\left(Y_{\widehat{T}_k^-},\Zhat_{\widehat{T}_k^-},\widehat{T}_k,\boldsymbol{y}\right)\mathcal{M}_{n_k}\left(\textrm{d}\boldsymbol{y} \right).
\end{multline*}
Finally we use the distribution of $\widehat{A}_k$ when the branching individual is the spinal one, defined in Dynamics \ref{spine rates}.
$$
\mathbb{E}\left[\mathbbm{1}_{\{\widehat{A}_k = (e',n_k)\}} \big\vert \widehat{\mathcal{F}}_{k-1}, \widehat{T}_k \right] = \frac{\widehat{\Gamma}^*_{n_k}\left(Y_{\widehat{T}_k^-},\Zhat_{\widehat{T}_k^-},\widehat{T}_k\right)}{\psi\left(Y_{\widehat{T}_k^-},\Zhat_{\widehat{T}_k^-},\widehat{T}_k \right)} \frac{B_{n_k}\left(Y_{T_k^-},\Zhat_{\widehat{T}_k^-},\widehat{T}_k\right)}{\widehat{\tau}_{tot}\left(\Zhat_{\widehat{T}_k^-},\widehat{T}_k\right)}.
$$ 
This gives \eqref{eq exhasution} on the event $\{j \neq \emptyset\}$.

Now, combining \eqref{eq C hat} and \eqref{eq exhasution}, we get
\begin{multline*}
\widehat{C}\left(e',\left(\widehat{\mathcal{V}}_i,0\leq i\leq k-1\right)\right) =  \mathbbm{1}_{\{E_{k-1} = e'\}}\mathbb{E}\Bigg[ \exp\left(\int_{\widehat{T}_{k-1}}^{\widehat{T}_{k}}\lambda\left(Y_s,\Zhat_s,s\right)\textrm{d}s\right)I\left(\widehat{T}_k - \widehat{T}_{k-1} \right)  \\
\times \frac{B_{n_k}\Big(X^{u_k}_{T_k^-},\Zhat_{\widehat{T}_k^-},\widehat{T}_k\Big)}{\widehat{\tau}_{\textrm{tot}}\left(\Zhat_{\widehat{T}_k^-},\widehat{T}_k\right)} \int_{\mathcal{X}^{n_k}} 
F\left(\boldsymbol{y}\right)K_{n_k}\left(X^{u_k}_{\widehat{T}_k^-},\Zhat_{\widehat{T}_k^-},\widehat{T}_k,\boldsymbol{y}\right)\mathcal{M}_{n_k}(\text{d}\boldsymbol{y})\Bigg| \widehat{\mathcal{F}}_{k-1} \Bigg].
\end{multline*}
Using that the time between two jump follows an inhomogeneous exponential law of instantaneous rate $\widehat{\tau}_{\textrm{tot}}$, we have
\begin{multline}\label{eq: end}
\widehat{C}\left(e',\left(\widehat{\mathcal{V}}_i,0\leq i\leq k-1\right)\right)=\mathbbm{1}_{E_{k-1} = e'}
\int_{\widehat{T}_{k-1}}^{+\infty} \exp\left(\int_{\widehat{T}_{k-1}}^{\widehat{T}_{k}}\lambda\left(Y_s,\Zhat_s,s\right)-\widehat{\tau}_{\textrm{tot}}\left(Y_s,\Zhat_s,s\right)\textrm{d}s\right) \\
\times I(t-\widehat{T}_{k-1})B_{n_k}\left(X^{u_k}_{t},\Z_{t},t\right) \int_{\mathcal{X}^{n_k}} F(\boldsymbol{y})K_{n_k}\left(X^{u_k}_{t},\Z_{t},t,\boldsymbol{y}\right)\mathcal{M}_{n_k}\left(\textrm{d}\boldsymbol{y} \right) \textrm{d}t.
\end{multline}
Finally, using in \eqref{eq: end} the fact that $\lambda$, defined in \eqref{eq lambda}, is the difference of branching rates between the branching process and the spine process, we get \eqref{eq: C C hat} and it concludes the proof.
\end{proof}

\begin{proof}[Proof of Theorem \ref{thm:pdmc}]
Let $\psi \in\mathcal{D}$, $t \geq 0$, and $\bar{z} \in \Zbarset$. 
Let $\left(\left(E_t,\Zbarhat_{t}\right), t\geq 0 \right)$ be the time-inhomogeneous $\Zbarhatset$-valued branching process of generators $(\widehat{\mathcal{L}}_{\psi}^t, t\geq 0)$ defined in \eqref{generator spinal}. Let $\widehat{T}_{\textrm{Exp}}$ denote its explosion time and $\left(\left(Y_t,\Zhat_{t}\right), t\geq 0 \right)$ its projection on $\Zhatset$. For every $t<T_{\textrm{Exp}}$, there exists $k \in \mathbb{N}$ such that $\left\{T_k \leq t < T_{k+1}\right\}$ where $T_{k+1}=+\infty$ if there is no more jumps after $T_k$. Thus we can write 
\begin{equation*}
\mathbb{E}_{\bar{z}}\left[\mathbbm{1}_{\{T_{\textrm{Exp}} > t, \mathbb{G}(t) \neq \emptyset\}}H\left(U_t,(\Zbar_s,s\leq t)\right)\right] = \sum_{k \geq 0}\mathbb{E}_{\bar{z}}\left[\mathbbm{1}_{\{T_k \leq t < T_{k+1}, \mathbb{G}(t) \neq \emptyset\}}H\left(U_t,(\Zbar_s,s\leq t)\right)\right].
\end{equation*}
For every $t\geq 0$, $k \in \mathbb{N}$, every non-extinguished sequence $a= (a_i, 1 \leq i \leq k) \in \mathfrak{U}_k(\bar{z})$  with $a_i = (u_i,n_i)$ and every $e \in \mathbb{G}_{k}\left(a\right)$, there exists a \correct{non-negative measurable} function $F_{k,a}^{t,e}$ on $\underset{i=1}{\overset{k}{\bigotimes}}(\mathbb{R}_+ \times \mathcal{U}\times \mathbb{N} \times \mathcal{X}^{n_i})$, such that  
\begin{multline*}
\mathbb{E}_{\bar{z}}\left[\mathbbm{1}_{\{T_{\textrm{Exp}} > t, \mathbb{G}(t) \neq \emptyset\}}H\left(U_t,(\Zbar_s,s\leq t)\right) \right] =\\ \sum_{k \geq 0}\sum_{a \in \mathfrak{U}_k(\bar{z})}\sum_{ e \in \mathbb{G}_k(a)} \mathbb{E}_{\bar{z}}\left[\mathbb{E}\left[\mathbbm{1}_{\{T_{k+1} >t\}} \vert \mathcal{F}_k\right]\mathbbm{1}_{\{T_k \leq t\}}F_{k,a}^{t,e}\left(\mathcal{V}_i, 0 \leq i \leq k\right)\prod_{i=0}^{k}\mathbbm{1}_{\{A_i = a_i\}} \right].
\end{multline*}
We apply Lemma \ref{lemma induction} with the non-negative measurable function $F$ defined for all sequence $\left(v_i, 0\leq i\leq k\right)\in \underset{i=1}{\overset{k}{\bigotimes}}(\mathbb{R}_+ \times \mathcal{U}\times \mathbb{N} \times \mathcal{X}^{n_i})$, where for all $0\leq i\leq k$, $v_i := (t_i,a_i,y_i)$, by
\begin{equation*}
F\left(v_i,0 \leq i\leq k\right) := \mathbb{P}\bigg( S_{k+1} >t-t_k \bigg\vert \bigcap_{0 \leq i\leq k}\left\{\mathcal{V}_i=v_i\right\} \bigg)\mathbbm{1}_{t_k \leq t}F_{k,a}^{t,e}\left(v_i,0 \leq i\leq k\right),
\end{equation*}
where $S_{k+1}$ is the random variable giving the $(k+1)$-th inter-arrival time of jumps in the original process. Note that $S_{k+1}$ follows the same probability law as $T_{k+1}-T_{k}$. We thus get, for all $e \in\mathbb{G}_k(a)$
\begin{multline}\label{eq break}
\mathbb{E}_{\bar{z}}\left[\mathbbm{1}_{\{\mathbb{G}(t) \neq \emptyset\}\cap\{T_{\textrm{Exp}} > t\}}H\left(U_t,(\Zbar_s,s\leq t)\right) \right] = \left\langle z, \psi\left(\cdot,z,0\right) \right\rangle \sum_{k \geq 0}\sum_{a \in \mathfrak{U}_k(\bar{z})}\sum_{ e \in \mathbb{G}_k(a)} \mathbb{E}_{\bar{z}}\Big[\mathbbm{1}_{\{E_{k} = e\}} \\
\times \xi_k\left(E_k,\left(\widehat{\mathcal{V}}_i, 0 \leq i \leq k\right) \right)F\left(\widehat{\mathcal{V}}_i, 0 \leq i \leq k\right)\prod_{i=0}^{k}\mathbbm{1}_{\{\widehat{A}_i = a_i\}} \Big].
\end{multline}
Then, using that $\lambda$ introduced in \eqref{eq lambda} verifies $\lambda := \widehat{\tau}_{\textrm{tot}} - \tau$,
\begin{align*}
\mathbb{P}\bigg( S_{k+1} >t-\widehat{T}_k \bigg\vert \bigcap_{0 \leq i\leq k}\left\{\widehat{\mathcal{V}}_i=v_i\right\} \bigg)\mathbbm{1}_{\{\widehat{T}_k \leq t\}} &= e^{\int_{\widehat{T}_k}^t\lambda\left(Y_s,\Zhat_s,s \right)\textrm{d}s}e^{-\int_{\widehat{T}_k}^t\widehat{\tau}_{\textrm{tot}}\left(Y_s,\Zhat_s,s \right)\textrm{d}s}\mathbbm{1}_{\{\widehat{T}_k \leq t\}}\\
&=e^{\int_{\widehat{T}_k}^t\lambda\left(Y_s,\Zhat_s,s \right)\textrm{d}s}\mathbb{E}\left[\mathbbm{1}_{\{\widehat{T}_{k+1} > t\}} \vert \widehat{\mathcal{F}}_k\right]\mathbbm{1}_{\{\widehat{T}_k \leq t\}}.
\end{align*}
We recall that on the event $\{\widehat{T}_{k+1} > t\}$, there is no jump in the interval $(\widehat{T}_k,t]$, thus, applying Lemma \ref{chain rule} on this interval we get
\begin{multline}\label{eq up to t}
\xi_k\left(E_k,\left(\widehat{\mathcal{V}}_i, 0 \leq i \leq k\right) \right) \mathbb{P}\bigg( S_{k} >t-\widehat{T}_k \bigg\vert \bigcap_{0 \leq i\leq k}\left\{\widehat{\mathcal{V}}_i=v_i\right\} \bigg)\mathbbm{1}_{\{\widehat{T}_k \leq t\}} =\\ \frac{1}{\psi\left(Y_t,\Zhat_t,t\right)}\exp\left(\int_{0}^{t}\frac{\mathcal{G}\psi\left(Y_s,\Zhat_s,s\right)}{\psi\left(Y_s,\Zhat_s,s\right)}\textrm{d}s\right)\mathbb{E}\left[\mathbbm{1}_{\{\widehat{T}_{k} \leq t < \widehat{T}_{k+1}\}} \vert \widehat{\mathcal{F}}_k\right].
\end{multline}
On the event $\{\widehat{T}_k \leq t < \widehat{T}_{k+1}\}$, $E_k=E(t)$ almost surely, then using \eqref{eq break} and \eqref{eq up to t} gives 
\begin{multline*}
\mathbb{E}_{\bar{z}}\left[\mathbbm{1}_{\{T_{\textrm{Exp}} > t, \mathbb{G}(t) \neq \emptyset\}}H\left(U_t,(\Zbar_s,s\leq t)\right) \right] = \left\langle z, \psi\left(\cdot,z\right) \right\rangle \sum_{k \geq 0}\sum_{a \in \mathfrak{U}_k(\bar{z})}\sum_{ e \in \mathbb{G}_k(a)} \mathbb{E}_{\bar{z}}\Big[\mathbbm{1}_{\{E(t) = e\}} \\
\times \frac{\mathbbm{1}_{\{\widehat{T}_k \leq t < \widehat{T}_{k+1}\}}}{\psi\left(Y_t,\Zhat_t,t\right)}\exp\left(\int_{0}^{t}\frac{\mathcal{G}\psi\left(Y_s,\Zhat_s,s\right)}{\psi\left(Y_s,\Zhat_s,s\right)}\textrm{d}s\right)F_{k,a}^{t,e}\left(\widehat{\mathcal{V}}_i, 0 \leq i \leq k\right)\prod_{i=0}^{k}\mathbbm{1}_{\{\widehat{A}_i = a_i\}} \Big].
\end{multline*}
Reconstructing the right-hand side and using the fact that the spine process does not extinct, we have
\begin{multline*}
\mathbb{E}_{\bar{z}}\left[\mathbbm{1}_{\{T_{\textrm{Exp}} > t,\mathbb{G}(t) \neq \emptyset\}}H\left(U_t,(\Zbar_s,s\leq t)\right) \right] = \\
\langle z,\psi(\cdot,z,0) \rangle \mathbb{E}_{\bar{z}}\left[\mathbbm{1}_{\{\widehat{T}_{\textrm{Exp}}>t\}} \frac{p_{E_t}\left(\Zbarhat_t\right)}{\psi\left(Y_t,\Zhat_t,t\right)}\exp\left(\int_0^t \frac{\mathcal{G}\psi\left(Y_s,\Zhat_s,s\right)}{\psi\left(Y_s,\Zhat_s,s\right)}\textrm{d}s\right) H\left(E_t,(\Zbarhat_s, s\leq t)\right)\right],
\end{multline*}
that concludes the proof.
\end{proof}

\subsection{Proofs of Section \ref{section: llogl}}\label{section: proof llogl}
In this section we derive the results on the limiting martingale and a $L\log L$ criterion.
\begin{proof}[Proof of Proposition \ref{prop:martingale}]
Let $\psi \in \mathcal{D}$ \correct{satisfying} Assumption \ref{assumption B}, we first show for all $t\geq0$, the integrability of the random variable $W_t(\psi)$ introduced in \eqref{def:martingale}. As Assumption \ref{assumption A} is satisfied, we can apply Corollary \ref{cor: pdmc} for any $t\geq 0$ to the positive function $f$ on $\mathbb{D}([0,t],\mathcal{X}\times\Zbarset)$, such that for all $u \in \mathbb{G}(t)$,
$$
f\left((X^u_s,\Zbar_s), 0 \leq s \leq t \right):= \exp\left(-\int_0^t\frac{\mathcal{G} \psi\left(X^u_s,\Z_s,s\right)}{\psi\left(X^u_s,\Z_s,s\right)}\textrm{d}s\right).
$$
Corollary \ref{cor: pdmc} applied to the function $f$, ensures that for any initial condition $\bar{z}$,
\begin{equation}\label{eq: integrability}
\mathbb{E}_{\bar{z}}\left[\sum_{u \in \mathbb{G}(t)}\exp\left(-\int_0^t\frac{\mathcal{G} \psi\left(X^u_s,\Z_s,s\right)}{\psi\left(X^u_s,\Z_s,s\right)}\textrm{d}s\right)\psi\left(X^u_t,\Z_t,t\right)\right] = \langle z,\psi(\cdot,z,0) \rangle.
\end{equation}
This last identity guarantees the integrability of $W_t(\psi)$. 

Now, for all $r,t \in \mathbb{R}_+^*$, such that $t \geq r$, we decompose the individuals alive in the population at time $t$ according to their ancestors at time $r$. see Figure \ref{schema CT} bellow. 
\begin{figure}[h!]
\begin{center}
\captionsetup{justification=centering}
\newcommand{\boundellipse}[3]
{(#1) ellipse (#2 and #3)
}
\begin{tikzpicture}[thick,scale=0.6, every node/.style={scale=0.9}]
\draw[->][very thick] (0,0) -- (15,0);
\draw node  at (0.5,-0.5) {$0$};
\draw[thick] (0.5,0.2) -- (0.5,-0.2);
\draw node at (6.5,-0.5) {$r$};
\draw[thick] (6.5,0.2) -- (6.5,-0.2);
\draw node at (12.5,-0.5) {$t$};
\draw[thick] (12.5,0.2) -- (12.5,-0.2);
\draw[fill=black] (0.5,2) circle (0.1);
\draw[dotted] (0.5,2) to[bend left=4.25] (12.5, 3.7);
\draw[dotted] (0.5,2) to[bend right=4.25] (12.5, 0.3);
\draw node at (0.5,2.5) {$\mathbb{G}(0)$};
\draw node at (6.47,3.7) {$v \in \mathbb{G}(r)$};
\fill[pattern=north west lines,opacity=.6,draw] \boundellipse{6.47,2}{-0.3}{1.1};
\draw node at (12.5,4.2) {$ \left\{u \in \mathbb{G}(t),\ s.t \ v \preceq u\right\}$};
\fill[pattern=north west lines,opacity=.6,draw] \boundellipse{12.5,2}{-0.5}{1.7};
\end{tikzpicture}
\end{center}
\caption{Decomposition of the population according to the ancestors at time $r$}\label{schema CT}
\end{figure}
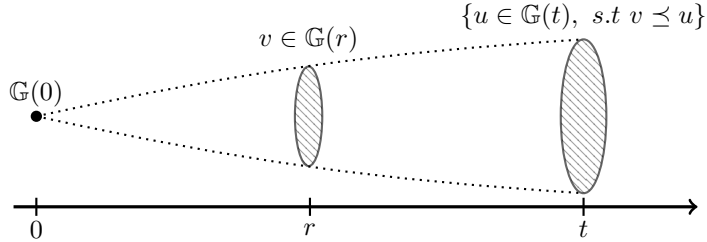

We then have
\begin{multline*}
\mathbb{E}_{\bar{z}}\left[W_t(\psi)\big\vert \mathcal{F}_r\right] = \sum_{v \in \mathbb{G}(r)}\exp\left(-\int_0^r\frac{\mathcal{G} \psi\left(X^u_s,\Z_s,s\right)}{\psi\left(X^u_s,\Z_s,s\right)}\textrm{d}s\right)\\
\times\mathbb{E}_{\bar{z}}\left[\sum_{u \in \mathbb{G}(t), \ s.t \ v \preceq u}\exp\left(-\int_r^t\frac{\mathcal{G} \psi\left(X^u_s,\Z_s,s\right)}{\psi\left(X^u_s,\Z_s,s\right)}\textrm{d}s\right)\psi\left(X^u_t,\Z_t,t\right)\Bigg\vert \mathcal{F}_r\right].
\end{multline*}
Now, in order to establish that the conditional expectation is equal to $\psi\left(X^v_r,\Z_r\right)$  for all $v \in \mathbb{G}(r)$, we apply Corollary \ref{cor: pdmc} to the positive functions $f_v$ on $\mathcal{U}\times\mathbb{D}([0,t],\Zbarset)$, such that
$$
f_v\left(u,(\Zbar_s, 0\leq s \leq t) \right):= \mathbbm{1}_{v \preceq u}\exp\left(-\int_r^t\frac{\mathcal{G}\psi\left(X^u_s,\Z_s,s\right)}{\psi\left(X^u_s,\Z_s,s\right)}\textrm{d}s\right)
$$
and use the Markov property. We get that for all $v \in \mathbb{G}(r)$ 
\begin{multline*}
\mathbb{E}\left[\sum_{\substack{u \in \mathbb{G}(t), \\ \textit{s.t.} \ v \preceq u}}\exp\left(-\int_r^t\frac{\mathcal{G} \psi\left(X^u_s,\Z_s,s\right)}{\psi\left(X^u_s,\Z_s,s\right)}\textrm{d}s\right)\psi\left(X^u_t,\Z_t,t\right)\Bigg\vert \mathcal{F}_r\right] =\\ \langle \Z_r,\psi(\cdot,\Z_r,r)\rangle\mathbb{E}_{\Z_r}\left[\mathbbm{1}_{\{v \preceq E_t\}}\right].
\end{multline*}
Finally using the new spinal individual distribution \eqref{spine proba}, we get
$$
\mathbb{E}\left[W_t(\psi)\big\vert \mathcal{F}_r\right] = \sum_{v \in \mathbb{G}(r)}\exp\left(-\int_0^r\frac{\mathcal{G}\psi\left(X^u_s,\Z_s,s\right)}{\psi\left(X^u_s,\Z_s,s\right)}\textrm{d}s\right) \psi\left(X^v_r,\Z_r,r\right).
$$
That concludes the proof.
\end{proof}

To establish Theorem \ref{thm LlogL}, we need the following lemma from measure theory. 

\begin{lemma}\label{lemma:durrett}
Let $\left(\Omega,\mathcal{F}, \mu \right)$ be a probability space and let $\widehat{\mu}$ be a finite non negative measure on $\Omega$. Let $\left(\mathcal{F}_t, 0 \leq t \right)$ be increasing $\sigma$-fields such that $\sigma\left(\underset{0 \leq t}{\bigcup}\mathcal{F}_t\right) = \mathcal{F}$, and $\widehat{\mu}_t$, $\mu_t$ be the restrictions of  $\widehat{\mu}$ and $\mu$ to $\mathcal{F}_t$. Suppose that there exists a non-negative $\mathcal{F}_t$-martingale $\left(W_t, 0 \leq t\right)$ such that for all $t \geq 0$
\begin{equation*}
\frac{\textrm{d}\widehat{\mu}_t}{\textrm{d}\mu_t} = W_t.
\end{equation*}
Then, denoting $W := \limsup_{t \to +\infty}W_t$, we have the following dichotomy:
\begin{enumerate}
\item $\int W \textrm{d}\mu = \int W_0 \textrm{d}\mu \textrm{ if and only if } \  W < +\infty \ \ \widehat{\mu}\textrm{-a.s.}$
\item $W=0 \ \ \mu\textrm{-a.s. if and only if } \ W = +\infty \ \ \widehat{\mu}\textrm{-a.s.}$
\end{enumerate}
\end{lemma}
\begin{proof}
We refer to \correct{\cite{Athreya00,LPP95}} for the proof of this result in discrete time. The extension to continuous time changes of measure uses Kolmogorov's extension theorem, see  Durrett Appendix A \cite{Durrett} for further details.
\end{proof}
We now state the following lemma, that is the dual proposition of Theorem \ref{thm LlogL}.
\begin{lemma}\label{lemma: convergence W hat} Let $\psi \in\mathcal{D}$ satisfying Assumption \ref{assumption H}, let $(\widehat{W}_t(\psi), t\leq \widehat{T}_{\textrm{Exp}})$ be the $\widehat{\mathcal{F}}_t$-adapted process such that, for all $t\leq \widehat{T}_{\textrm{Exp}}$\correct{,}
\begin{equation}\label{def: martingale spine}
\widehat{W}_t(\psi) := \sum_{u \in \widehat{\mathbb{G}}(t)}\psi\left(X^u_t,\Zhat_t,t\right)\exp\left(-\int_0^t\frac{\mathcal{G}\psi\left(X^u_s,\Zhat_s,s\right)}{\psi\left(X^u_s,\Zhat_s,s\right)}\textrm{d}s\right),
\end{equation}
where $\widehat{\mathbb{G}}(t)$ is the set of labels of individuals living in the spinal population at time $t$.

Under Assumption \ref{assumption H_technical}, \eqref{crit sup finite} implies that, $\widehat{T}_{\textrm{Exp}} = \infty$  and
$\limsup_{t \to +\infty}\widehat{W}_t(\psi) < +\infty$ almost surely.

Under Assumption \ref{assumption H_inf_technical}, \eqref{crit inf infinite} implies that $\limsup_{t \to +\infty}\widehat{W}_t(\psi) = +\infty$ almost surely.
\end{lemma}
\correct{This lemma is the innovative part of this proof. It involves the decomposition of the spinal process as a process with immigration where the spinal individual provides new individuals at a $\psi$-biased rate, and the new spine is the new source of immigration. However as the function $\psi$ is not constant, the spinal construction also changes the dynamics of individuals outside the spine, that do not behave as those in the original process $(\Z_t, t\geq 0)$. Assumption \ref{assumption H_technical} is used to control the behavior of the individuals outside the spine to prove the non degeneracy of the limiting martingale. For the second part of this theorem, we use Assumption \ref{assumption H_inf_technical} to establish that the contribution of the offspring without the new spine is enough to ensure the degeneracy.}

\begin{proof}[Proof of Theorem \ref{thm LlogL}]
Theorem \ref{thm:pdmc} applied with any function $\psi \in\mathcal{D}$ satisfying Assumption \ref{assumption H} exhibits the martingale $(W_t(\psi), t\geq 0)$ as a Radon-Nikodym derivative. We can thus apply Lemma \ref{lemma:durrett} to the spinal change of measure. 
The martingale $W_t(\psi)$ and its limit $W(\psi)$, introduced in \eqref{def:martingale}, verify the following dichotomy:
\begin{enumerate}
\item $\mathbb{E}_{\bar{z}}\left[W(\psi)\right] = \mathbb{E}_{\bar{z}}\left[W_0(\psi)\right]$ 
if and only if $
\limsup_{t \to +\infty}\widehat{W}_t(\psi) < +\infty \ \textrm{ a.s.}$
\item $W(\psi) = 0  \ \textrm{ a.s.}$ if and only if $\limsup_{t \to +\infty}\widehat{W}_t(\psi) = +\infty \ \textrm{ a.s.}$
\end{enumerate}
In this case, we use \eqref{eq: integrability} to obtain that
$$
\mathbb{E}_{\bar{z}}\left[W_0(\psi)\right] =\langle z,\psi(\cdot,z,0) \rangle,
$$
and a direct application of Lemma \ref{lemma: convergence W hat} concludes the proof.
\end{proof}
\begin{proof}[Proof of Lemma \ref{lemma: convergence W hat}]
Let $\psi\in\mathcal{D}$ satisfying Assumption \ref{assumption H}, and introduce  for all $k \in \mathbb{N}$, 
$$
\overline{\widehat{p}}^*_n:=\sup_{(x,\Z,t) \in \Zhatset\times\mathbb{R}_+}\widehat{p}^*_n(x,\Z,t), \quad \ \underline{\widehat{p}}^*_n:=\inf_{(x,\Z,t) \in \Zhatset\times\mathbb{R}_+}\widehat{p}^*_n(x,\Z,t)
$$
$$
\overline{\widehat{K}}^*_n(\cdot):=\sup_{(x,\Z,t) \in \Zhatset\times\mathbb{R}_+}\widehat{K}^*_n(x,\Z,t,\cdot), \quad \text{and} \quad \underline{\widehat{K}}^*_n(\cdot):=\inf_{(x,\Z,t) \in \Zhatset\times\mathbb{R}_+}\widehat{K}^*_n((x,\Z,t,\cdot),
$$
where $(\widehat{p}^*_n, n\geq 0)$ and $\widehat{K}^*_n$ are respectively the law of the number of children and the measure giving the offspring traits for the $\psi$-spine, defined respectively in \eqref{def: p star hat} and \eqref{def: K star hat}.

We first state two technical lemmas, proven in Appendix \ref{appendix immigration}.
\begin{lemma}\label{lemma: H finite}
For all $\psi\in \mathcal{D}$ satisfying Assumption \ref{assumption H}, the criteria
\begin{equation*}
\sum_{n\geq 1} \overline{\widehat{p}}^*_n \int_{\mathcal{X}^n}\sup_{(x,\Z,t) \in \Zhatset\times\mathbb{R}_+}\left[\log\left(\sum_{i=1}^n\psi(y^i,\Z^+(x,\boldsymbol{y}),t) \right)\right] \overline{\widehat{K}}^*_n(\boldsymbol{y})\mathcal{M}_n(\text{d}\boldsymbol{y})<+\infty
\end{equation*}
implies that
\begin{equation*}
\limsup\limits_{t \to +\infty}\psi\left(Y_t,\Zhat_t,t \right)\exp\left(-\int_0^t\frac{\mathcal{G}\psi\left(Y_s,\Zhat_s,s\right)}{\psi\left(Y_s,\Zhat_s,s\right)}\text{d}s\right) = 0 \quad \text{a.s}.
\end{equation*}
\end{lemma}
We recall the notation $\Z^+(x,\boldsymbol{y}) := \Z + \sum_{i=1}^n\delta_{y^i} - \delta_{x}$, introduced in \eqref{nu+}.
\begin{lemma}\label{lemma: H infty}
For all $\psi\in \mathcal{D}$ satisfying Assumptions \ref{assumption H_inf_technical} and \ref{assumption H}, the criteria 
\begin{equation*}
\sum_{n\geq 1} \underline{\widehat{p}}^*_n \int_{\mathcal{X}^n}\inf_{(x,\Z,t) \in \Zhatset\times\mathbb{R}_+}\left[\log\left(\sum_{i=1}^n\psi(y^i,\Z^+(x,\boldsymbol{y}),t) \right)\right]\underline{\widehat{K}}^*_n(\boldsymbol{y})\mathcal{M}_n(\text{d}\boldsymbol{y}) = +\infty
\end{equation*}
implies that for any sequence of integers $(j_n, n\geq 2)$ such that $1\leq j_n\leq n$,  
\begin{equation*}
\sum_{n\geq 2} \underline{\widehat{p}}^*_n \int_{\mathcal{X}^n}\inf_{(x,\Z,t) \in \Zhatset\times\mathbb{R}_+}\left[\log\left(\sum_{\substack{i=1\\ i\neq j_n}}^n\psi(y^i,\Z^+(x,\boldsymbol{y}),t) \right)\right]\underline{\widehat{K}}^*_n(\boldsymbol{y})\mathcal{M}_n(\text{d}\boldsymbol{y}) = +\infty.
\end{equation*}
\end{lemma}
To establish Lemma \ref{lemma: convergence W hat}, we follow the conceptual decomposition of the spine process, first introduced by Lyons, Pemantle and Peres \cite{LPP95}, \correct{that is, viewing the spine as an immigration source into a process without a spine}. We consider the $\psi$-spinal process $(\mathring{\Zbarhat}_t, t\geq 0)$ describing all the individuals outside the spine and its associated marginal process $\mathring{\Zhat}_t$, defined for all $t\geq0$ by:
\begin{equation*}
\mathring{\Zbarhat}_t := \Zbarhat_t - \delta_{(E_t,Y_t)}, \quad \textrm{and} \quad \mathring{\Zhat}_t := \Zhat_t - \delta_{Y_t}.
\end{equation*}
We denote by $\mathring{\mathbb{G}}(t)$ the random set of labels of non-spinal individuals living at time $t$
\begin{equation*}
\mathring{\mathbb{G}}(t) = \left\{u \in\mathcal{U}\backslash \{E_t \} : \ \int_{\mathcal{U}\times \mathcal{X}}\mathbbm{1}_{\{v=u\}}\Zbarhat_t(\textrm{d}v,\textrm{d}x) >0 \right\} .
\end{equation*}
We also introduce for all $t\geq 0$, 
\begin{equation*}
\mathring{W}_t(\psi) := \sum_{u \in \mathring{\mathbb{G}}(t)}\psi\left(X^u_t,\Zhat_t,t\right)\exp\left(-\int_0^t\Lambda\left(X^u_s,\Zhat_s,s\right)\textrm{d}s\right),
\end{equation*}
where for all $x,\Z,t \in \Zhatset\times\mathbb{R}_+$, $$\Lambda(x,\Z,t):=\frac{\mathcal{G}\psi(x,\Z,t)}{\psi(x,\Z,t)}.$$
\newline
$\bullet$ We handle first the \textbf{degenerated case}, and suppose that Assumption \ref{assumption H_inf_technical} holds and that
\begin{equation}\label{eq H+}
\sum_{n\geq 1} \underline{\widehat{p}}^*_n \int_{\mathcal{X}^n}\inf_{(x,\Z,t) \in \Zhatset\times\mathbb{R}_+}\left[\log\left(\sum_{i=1}^n\psi(y^i,\Z^+(x,\boldsymbol{y}),t) \right)\right]\underline{\widehat{K}}^*_n(\boldsymbol{y})\mathcal{M}_n(\text{d}\boldsymbol{y}) = +\infty.
\end{equation}
We notice that, almost surely for all $t \geq 0$, the martingale $\widehat{W}_t(\psi)$ introduced in \eqref{def: martingale spine} verifies
\begin{equation}\label{eq W inf}
\mathring{W}_t(\psi) \leq \widehat{W}_t(\psi).
\end{equation}

We denote $(\widehat{T}^*_k, k\geq0)$ the sequence of random jumps times of the spine, and $(\widehat{N}^*_k, k\geq 0)$ the sequence of random variables giving the number of children at each branching event of the spine. For every $k \geq 0,\  (X^i_{k}, 1 \leq i \leq \widehat{N}^*_k)$ is the random vector giving the types of the children of the spine, among them the trait of the new spine is denoted $Y_{\widehat{T}^*_k}$. 
We also introduce the sequence of filtrations $\left(\mathcal{F}_k,k\geq 1\right)$ such that for all $k\geq 0$:
\begin{equation}\label{lemma: filtration}
\mathcal{F}_{k} := \widehat{\mathcal{F}}_{\widehat{T}^*_{k+1}\minus}.
\end{equation} 
where $\widehat{\mathcal{F}}_t$ is the canonical filtration of the spinal process up to time $t$.
Thus, $\mathcal{F}_k$ corresponds to the information on the process until the time of the $(k+1)$-th jump.
We notice that
\begin{align*}
\mathring{\mathcal{W}}_{\widehat{T}^*_k} = \mathring{\mathcal{W}}_{\widehat{T}^*_k-} &+ \sum_{i=1}^{\widehat{N}^*_k}\psi\left(\widehat{X}^{i}_k,\Zhat_{\widehat{T}^*_k},{\widehat{T}^*_k}\right)\exp\left(-\int_0^{\widehat{T}^*_k}\Lambda\left(X^i_s,\Zhat_s,s\right)\textrm{d}s \right)\\
& - \psi\left(Y_{\widehat{T}^*_k},\Zhat_{\widehat{T}^*_k},{\widehat{T}^*_k}\right)\exp\left(-\int_0^{\widehat{T}^*_k}\Lambda\left(Y_s,\Zhat_s,s\right)\textrm{d}s \right).
\end{align*}
Using that for all $t\geq 0$, $\mathring{\mathcal{W}}_{t}$ is almost surely non-negative, and using the upper bound of $\Lambda$ in Assumption \ref{assumption H}, we obtain
\begin{align}\label{eq min sup}
\mathring{\mathcal{W}}_{\widehat{T}^*_k} &\geq  \left(\sum_{i=1}^{\widehat{N}^*_k}\psi\left(\widehat{X}^{i}_k,\Zhat_{\widehat{T}^*_k},{\widehat{T}^*_k}\right) - \psi\left(Y_{\widehat{T}^*_k},\Zhat_{\widehat{T}^*_k},{\widehat{T}^*_k}\right)\right)\exp\left(-C\widehat{T}^*_k\right) \nonumber \\
&\geq  \left(\sum_{i=1}^{\widehat{N}^*_k}\psi\left(\widehat{X}^{i}_k,\Zhat_{\widehat{T}^*_k},{\widehat{T}^*_k}\right) - \max_{1\leq j \leq n}\psi\left(X^j_{k},\Zhat_{\widehat{T}^*_k},{\widehat{T}^*_k}\right)\right)\exp\left(-C\widehat{T}^*_k\right).
\end{align}

Let us introduce for all $K>0$ and all $k\geq 1$, the event
$$
B^K_k := \left\{\log\left(\sum_{i=1}^{\widehat{N}^*_k}\psi\left(\widehat{X}^{i}_k,\Zhat_{\widehat{T}^*_k},{\widehat{T}^*_k}\right) - \max_{1\leq j \leq k}\psi\left(X^j_{k},\Zhat_{\widehat{T}^*_k},{\widehat{T}^*_k}\right)\right)\geq kK \right\} \in \mathcal{F}_{n}.
$$
We will now establish that 
\begin{equation*}
\sum_{k \geq 1}\mathbb{P}\left(B^K_k \Big\vert \mathcal{F}_{k-1}\right) = +\infty.
\end{equation*}
First notice that 
\begin{multline*}
\sum_{k \geq 1}\mathbb{P}\left(B^K_k \Big\vert \mathcal{F}_{k-1}\right) = 
\sum_{k \geq 1}\sum_{n\geq 1} \widehat{p}^*_n\left(Y_{\widehat{T}^*_k-},\Zhat_{\widehat{T}^*_k-},{\widehat{T}^*_k}\right)\int_{\mathcal{X}^n}\widehat{K}^*_n\left(Y_{\widehat{T}^*_k-},\Zhat_{\widehat{T}^*_k-},\widehat{T}^*_k,\boldsymbol{y}\right)\\
\times\mathbbm{1}_{\left\{\log\left(\sum_{i=1}^{n}\psi\left(y^i,\Zhat_{\widehat{T}^*_k},{\widehat{T}^*_k}\right) - \max_{1\leq j \leq n}\psi\left(y^j,\Zhat_{\widehat{T}^*_k},{\widehat{T}^*_k}\right)\right) \geq kK\right\}}\mathcal{M}_n(\text{d}\boldsymbol{y}).
\end{multline*}
Notice that the indicator function is always zero for $n=1$. Taking the infimum over $\Zhatset\times\mathbb{R}_+$, we get that
\begin{multline*}
\sum_{n \geq 1}\mathbb{P}\left(B^K_k \Big\vert \mathcal{F}_{k-1}\right) 
\geq \sum_{n\geq 2} \underline{\widehat{p}}^*_n\int_{\mathcal{X}^n}\underline{\widehat{K}}^*_n(\boldsymbol{y})\\ \times\sum_{k\geq 1}\mathbbm{1}_{\left\{\underset{(x,\Z,t) \in \Zhatset\times\mathbb{R}_+}{\inf}\left(\log\left(\sum_{i=1}^{n}\psi\left(y^i,\Z^+(x,\boldsymbol{y}),t\right) - \max_{1\leq j \leq n}\psi\left(y^j,\Z^+(x,\boldsymbol{y}),t\right)\right)\right) \geq kK\right\}}\mathcal{M}_n(\text{d}\boldsymbol{y}).
\end{multline*}
Finally, using that $\sum_{n\geq 2} \underline{\widehat{p}}^*_n \int_{\mathcal{X}^n}\underline{\widehat{K}}^*_n(\boldsymbol{y})\mathcal{M}_n(\text{d}\boldsymbol{y})\leq 1$ and that, for all $A\in\mathbb{R}$, $$\sum_{k\geq0}\mathbbm{1}_{\{A\geq kK\}} \geq -1+\frac{A}{K},$$ we get
\begin{multline*}
\sum_{k \geq 1}\mathbb{P}\left(B^K_k \Big\vert \mathcal{F}_{k-1}\right) 
\geq -1 + \frac{1}{K}\sum_{n\geq 2} \underline{\widehat{p}}^*_n\int_{\mathcal{X}^n}\underline{\widehat{K}}^*_n(\boldsymbol{y})\\
\times\underset{(x,\Z,t) \in \Zhatset\times\mathbb{R}_+}{\inf}\left(\log\left(\sum_{i=1}^{n}\psi\left(y^i,\Z^+(x,\boldsymbol{y}),t\right) - \max_{1\leq j \leq n}\psi\left(y^j,\Z^+(x,\boldsymbol{y}),t\right)\right)\right)\mathcal{M}_n(\text{d}\boldsymbol{y}).
\end{multline*}
Lemma \ref{lemma: H infty} ensures that the lower bound is infinite and therefore 
\begin{equation*}
\sum_{k \geq 1}\mathbb{P}\left(B^K_k \Big\vert \mathcal{F}_{k-1}\right) = +\infty.
\end{equation*}
We can thus apply the conditional second Borel-Cantelli lemma, see Theorem 4.3.4 in \cite{Durrett}. For all $K > 0$, 
\begin{equation*}
\limsup_{k \rightarrow \infty}\left(\sum_{i=1}^{\widehat{N}^*_k}\psi\left(\widehat{X}^{i}_k,\Zhat_{\widehat{T}^*_k},{\widehat{T}^*_k}\right) - \psi\left(Y_{\widehat{T}^*_k},\Zhat_{\widehat{T}^*_k},{\widehat{T}^*_k}\right)\right)e^{-Kk} = +\infty \quad \textrm{a.s.}
\end{equation*}
Furthermore, the bounds on the branching rate in Assumption \ref{assumption H} ensure that $\widehat{T}^*_k$ grows linearly almost surely to infinity as $k$ tends to infinity. Thus, using \eqref{eq min sup} we get that $\limsup_{t\rightarrow\infty}\mathring{W}_t(\psi) = +\infty$ and relation \eqref{eq W inf} concludes the proof.
\newline
$\bullet$ Now we treat the \textbf{non-degenerated} case, and suppose that Assumption \ref{assumption H_technical} holds and that
\begin{equation}\label{crit technic}
\sum_{n\geq 1}\overline{\widehat{p}}^*_n<+\infty, \qquad \sup_{n\geq 1}\int_{\mathcal{X}^n}\overline{\widehat{K}}^*_n(\boldsymbol{y})\mathcal{M}_n(\text{d}\boldsymbol{y})<+\infty
\end{equation}
and
\begin{equation}\label{eq H-}
\sum_{n\geq 1} \overline{\widehat{p}}^*_n \int_{\mathcal{X}^n}\sup_{(x,\Z,t) \in \Zhatset\times\mathbb{R}_+}\left[\log\left(\sum_{i=1}^n\psi(y^i,\Z^+(x,\boldsymbol{y}),t) \right)\right] \overline{\widehat{K}}^*_n(\boldsymbol{y})\mathcal{M}_n(\text{d}\boldsymbol{y})<+\infty.
\end{equation}
We notice that
\begin{equation}\label{eq inf}
\mathring{W}_t(\psi) = \widehat{W}_t(\psi) - \psi\left(Y_t,\Zhat_t,t \right)e^{-\int_0^t\Lambda\left(Y_s,\Zhat_s,s\right)\text{d}s}.
\end{equation}
In the following we prove that $\limsup\limits_{t \to +\infty} \mathring{W}_t(\psi) <+\infty$ almost surely, and use Lemma \ref{lemma: H finite} to conclude the proof.

Let $\widehat{E}= \mathcal{U} \times \mathbb{R}_+ \times \mathbb{N} \times \mathcal{X}^{\mathbb{N}}$ and $\widehat{E}^*= \mathbb{R}_+ \times \mathbb{N} \times \mathcal{X}^{\mathbb{N}}$. We introduce $\widehat{Q}\left(\textrm{d}s,\textrm{d}u,\textrm{d}r,\textrm{d}n,\textrm{d}\boldsymbol{y} \right)$ and $\widehat{Q}^*\left(\textrm{d}s,\textrm{d}r,\textrm{d}n,\textrm{d}\boldsymbol{y} \right)$ two independent Poisson point measures on $\mathbb{R}_+ \times\widehat{E}$ and $\mathbb{R}_+ \times\widehat{E}^*$ with respective intensity $\textrm{d}s(\sum_{v \in\mathcal{U}}\delta_v(\textrm{d}u))\textrm{d}r\sum_{i\geq1}\delta_{i}(\text{d}n)\mathcal{M}_i(\text{d}\boldsymbol{y})$ outside the spine and $\textrm{d}s\textrm{d}r\sum_{i\geq1}\delta_{i}(\text{d}n)\mathcal{M}_i(\text{d}\boldsymbol{y})$ for the spine. We denote by $\big(\widehat{\mathcal{F}}^*_t, t\geq 0\big)$ the canonical filtration associated with $\widehat{Q}^*$. Recall that the set of labels of the living individuals in the population outside the spine is denoted by $\mathring{\mathbb{G}}(t)$ and the random index of the new spine, introduced in \eqref{spine proba}, is denoted by $J(x,\Z,t,\boldsymbol{y})$.

We localize the process $(\mathring{W}_t(\psi), t\geq 0)$ to avoid explosion and use Itô's formula to explicit $\mathring{W}_t(\psi)$. Let $\widehat{T}^m$ be the $m$-th branching event of the spinal process for $m\geq 1$.
 
For all $T\leq\widehat{T}^m$, 
\begin{align}\label{eq: ito w rond}
\mathring{W}_T(\psi)&= \mathring{\mathcal{W}}_0 - \int_0^TI_t\text{d}t + \int_0^T\int_{\widehat{E}}\widehat{I}_t\widehat{Q}\left(\text{d}t,\text{d}u,\text{d}r,\text{d}n,\text{d}\boldsymbol{y}\right) \nonumber\\
&+\int_0^T\int_{\widehat{E}^*}\left\{\sum_{i=1}^n\mathbbm{1}_{\left\{i \neq J\left(Y_{t\minus},\chi_{t\minus},t,\boldsymbol{y} \right) \right\}}\psi\left(y^i,\Zhat_{t}^+\left(Y_t,\boldsymbol{y}\right),t\right)\exp\left(-\int_0^t\Lambda\left(Y_s,\chi_s,s\right)\text{d}s\right) \right.\nonumber\\
& \hspace{1.1cm}  + \sum_{v\in\mathring{\mathbb{G}}_{t\minus}} \left[\psi\left(X^v_t,\Zhat_{t}^+\left(Y_t, \boldsymbol{y}\right),t\right)-\psi\left(X^v_t,\chi_{t\minus},t\right)\right]\exp\left(-\int_0^t\Lambda\left(X^v_s,\chi_s,s\right)\text{d}s\right) \Bigg\}\nonumber\\
&  \hspace{3.8cm} \times \mathbbm{1}_{\left\{r\leq \widehat{B}^*_n\left(Y_{t\minus},\chi_{t\minus},t\right)\widehat{K}^*_n\left(Y_{t\minus},\chi_{t\minus},t,\boldsymbol{y} \right)\right\}}  \widehat{Q}^*\left(\text{d}t,\text{d}r,\text{d}n,\text{d}\boldsymbol{y}  \right)
\end{align}
where 
$$
I_t := \sum_{u\in\mathring{\mathbb{G}}_t}\left[\mathcal{G}\psi-G\psi\right]\left(X^u_t,\Zhat_t,t\right)e^{-\int_0^t\Lambda\left(X^u_s,\Zhat_s,s\right)\text{d}s},
$$ 
and  
\begin{multline}\label{def I_t hat}
\widehat{I}_t := \Bigg\{\left[\sum_{i=1}^n\psi\left(y^i,\Zhat_{t}^+\left(X^u_t, \boldsymbol{y}\right),t\right)-\psi\left(X^u_t,\Zhat_{t}^+\left(X^u_t, \boldsymbol{y}\right),t\right)\right]\exp\left(-\int_0^t\Lambda\left(X^u_s,\chi_s,s\right)\text{d}s\right) \\
+ \sum_{v\in\mathring{\mathbb{G}}_{t\minus}} \left[\psi\left(X^v_t,\Zhat_{t}^+\left(X^u_t, \boldsymbol{y}\right),t\right)-\psi\left(X^v_t,\Zhat_{t\minus},t\right)\right]\exp\left(-\int_0^t\Lambda\left(X^v_s,\chi_s,s\right)\text{d}s\right) \Bigg\}\\
\times \mathbbm{1}_{\left\{u\in\mathring{\mathbb{G}}_t\right\}}\mathbbm{1}_{\left\{r\leq \widehat{B}_n\left(Y_t,X^u_t,\Zhat_{t\minus},t\right)\widehat{K}_n\left(Y_t,X^u_t,\Zhat_{t\minus},t,\boldsymbol{y}\right)\right\}}.
\end{multline}

$I_t$ describes the deterministic evolution of types between the time of jump, $\widehat{I}_t$ corresponds to the jumps outside the spine, and the last integral in \eqref{eq: ito w rond} is the contribution from the spinal individual.

Using the expression of the operator $\mathcal{G}$ and Assumption \ref{assumption H_technical} we have for all $t>0$, all $u\in\mathring{\mathbb{G}}_t$
\begin{multline*}
I_t = \sum_{u\in\mathring{\mathbb{G}}_t}\exp\left(-\int_0^t\Lambda\left(X^u_s,\chi_s,s\right)\text{d}s\right)\sum_{n\geq 0} B_n\left(X^u_t,\Zhat_t,t\right)\\
\times\int_{\mathcal{X}^n}\left[\sum_{i=1}^n\psi\left(y^i,\Zhat_{t}^+\left(X^u_t, \boldsymbol{y}\right),t\right)-\psi\left(X^u_t,\Zhat_{t}^+\left(X^u_t, \boldsymbol{y}\right),t\right)\right]K_n\left(X^u_t,\Zhat_t,t,\boldsymbol{y}\right) \mathcal{M}_n(\text{d}\boldsymbol{y}).
\end{multline*}
Notice that Assumption \ref{assumption H_technical} implies that, for all $(x,(x_e,\Z),t) \in \mathcal{X}\times\Zhatset\times\mathbb{R}_+$, 
$$
\widehat{B}_n(x_e,x,\Z,t)\widehat{K}_n\left(x_e,x,\Z,t,\boldsymbol{y}\right) = B_n(x,\Z,t)K_n\left(x,\Z,t,\boldsymbol{y}\right).
$$
Thus using Assumption \ref{assumption H_technical} in \eqref{def I_t hat} we get 
$$
\int_0^T\int_{\widehat{E}}\widehat{I}_t\ \textrm{d}t\textrm{d}r\sum_{i\geq1}\delta_{i}(\text{d}n)\mathcal{M}_i(\text{d}\boldsymbol{y})=\int_0^TI_t\text{d}t. 
$$ 
Thus, using Assumption \ref{assumption H_technical} we get, for all $T\leq \widehat{T}^m$
\begin{align*}
\mathbb{E}\left.\left[\mathring{W}_T(\psi) \right\vert\widehat{\mathcal{F}}^*_{T}\right]&= \mathbb{E}\left[\mathring{W}_0(\psi)\right] \\
&\hspace{-0.8cm}+\int_0^T\int_{\widehat{E}^*}\sum_{i=1}^n\mathbbm{1}_{\left\{i \neq J\left(Y_{t\minus},\chi_{t\minus},t,\boldsymbol{y} \right) \right\}}\psi\left(y^i,\Zhat_{t}^+\left(Y_t,\boldsymbol{y}\right),t\right)\exp\left(-\int_0^t\Lambda\left(Y_s,\chi_s,s\right)\text{d}s\right)\\
&  \hspace{3.2cm} \times \mathbbm{1}_{\left\{r\leq \widehat{B}^*_n\left(Y_{t\minus},\chi_{t\minus},t\right)\widehat{K}^*_n\left(Y_{t\minus},\chi_{t\minus},t,\boldsymbol{y} \right)\right\}}  \widehat{Q}^*\left(\text{d}t,\text{d}r,\text{d}n,\text{d}\boldsymbol{y}  \right).
\end{align*} 
Using the positivity of $\psi$ we get a majorant by adding the new spine in the sum.
We then take the supremum over $\Zhatset\times\mathbb{R}_+$ in the indicator function and notice that for all $n \geq 0$, all $(x,\Z,t) \in \Zhatset\times\mathbb{R}_+$,
$$\widehat{B}^*_n\left(x,\Z,t\right) = \widehat{p}^*_n\left(x,\Z,t\right)\sum_{k\geq 0}\widehat{B}^*_k(x,\Z,t) \leq \overline{\widehat{p}}^*_n \overline{\tau}$$ 
where we used Assumption \ref{assumption H} for the last inequality. Putting these steps together and using the lower bound for $\mathcal{G}\psi/\psi$ in Assumption \ref{assumption H} we get:
\begin{multline}\label{eq: maj w ring}
\mathbb{E}\left.\left[\mathring{W}_T(\psi) \right\vert\widehat{\mathcal{F}}^*_{T}\right]\leq \mathbb{E}\left[\mathring{W}_0(\psi)\right] +  \int_0^T\int_{\widehat{E}^*}\mathbbm{1}_{\left\{r\leq \overline{\tau}\overline{\widehat{p}}^*_n\overline{\widehat{K}}^*_n(\boldsymbol{y})\right\}}\\
\times \exp\left\{\underset{(x,\Z,t) \in \Zhatset\times\mathbb{R}_+}{\sup}\left(\log\left(\sum_{i=1}^k\psi\left(y^i,\nu_{+}\left(x, \boldsymbol{y}\right),s\right)\right)\right)-ct \right\}  \widehat{Q}^*\left(\textrm{d}t,\textrm{d}r,\textrm{d}n,\textrm{d}\boldsymbol{y}\right).
\end{multline}
Finally we use \eqref{crit technic} to get an upper bound of the term in the indicator 
$$
\overline{\tau}\overline{\widehat{p}}^*_n\overline{\widehat{K}}^*_n(\boldsymbol{y}) \leq \overline{\tau}\left(\sum_{i\geq 1}\overline{\widehat{p}}^*_i\right)\left(\sup_{i\geq 1}\int_{\mathcal{X}^{i}}\overline{\widehat{K}}^*_i(\boldsymbol{y})\mathcal{M}_k(\text{d}\boldsymbol{y})\right)\left(\frac{\overline{\widehat{p}}^*_n}{\sum_{i\geq 1}\overline{\widehat{p}}^*_i}\right) \left(\frac{\overline{\widehat{K}}^*_n(\boldsymbol{y})}{\int_{\mathcal{X}^{n}}\overline{\widehat{K}}^*_n(\boldsymbol{w})\mathcal{M}_n(\text{d}\boldsymbol{w})} \right)
$$
and introduce the sequences of independent random variables $(S_k,k\geq 1)$ of exponential law of parameter $\overline{\tau}\left(\sum_{i\geq 1}\overline{\widehat{p}}^*_i\right)\left(\sup_{i\geq 1}\int_{\mathcal{X}^{i}}\overline{\widehat{K}}^*_i(\boldsymbol{y})\mathcal{M}_i(\text{d}\boldsymbol{y})\right)$. We also introduce the independent sequence of couples of random variables $((N_k,\boldsymbol{Y}_k), k\geq 1)$ such that for all $k$, $N_k$ follows the law $(\overline{\widehat{p}}^*_{\cdot}/\sum_{i\geq 1}\overline{\widehat{p}}^*_i)$ and $\boldsymbol{Y}_k\vert N_k$ follows the law $\overline{\widehat{K}}^*_{N_k}(\boldsymbol{y})/\int_{\mathcal{X}^{\mathbb{N}}}\overline{\widehat{K}}^*_{N_k}(\boldsymbol{w})\mathcal{M}_{N_k}(\text{d}\boldsymbol{w})$.
Thus we can write \eqref{eq: maj w ring} with a compound Poisson process
\begin{multline*}
\mathbb{E}\left.\left[\mathring{W}_T(\psi) \right\vert\widehat{\mathcal{F}}^*_{T}\right]\leq \mathbb{E}\left[\mathring{W}_0(\psi)\right] \\+  \sum_{k\geq 1}\exp\left\{\underset{(x,\Z,t) \in \Zhatset\times\mathbb{R}_+}{\sup}\left(\log\left(\sum_{j=1}^{N_k}\psi\left(Y^j_k,\nu_{+}\left(x, \boldsymbol{Y}_k\right),s\right)\right)\right) - c\sum_{j=1}^kS_j\right\}.
\end{multline*} 
In order to show that the series is almost surely finite, we introduce the following sequence of events: for every $K>0$ and every $k \geq 0$
$$
B_k^K := \left\{\underset{(x,\Z,t) \in \Zhatset\times\mathbb{R}_+}{\sup}\left(\log\left(\sum_{j=1}^{N_k}\psi\left(Y^j_k,\nu_{+}\left(x, \boldsymbol{Y}_k\right),s\right)\right)\right)\geq kK \right\}.
$$
Using the law of the couple of random variables $((N_k,\boldsymbol{Y}_k), k\geq 1)$ we get that
\begin{multline*}
\sum_{k\geq 1}\mathbb{P}\left(B_k^K\right) = \sum_{k\geq 1} \sum_{n\geq 1}   \frac{\overline{\widehat{p}}^*_n}{\sum_{i\geq 1}\overline{\widehat{p}}^*_i} \\
\times \int_{\mathcal{X}^{n}}\mathbbm{1}_{\left\{\underset{(x,\Z,t) \in \Zhatset\times\mathbb{R}_+}{\sup}\left(\log\left(\sum_{j=1}^{n}\psi\left(y^j,\nu_{+}\left(x, \boldsymbol{y}\right),s\right)\right)\right)\geq kK \right\}}\frac{\overline{\widehat{K}}^*_n(\boldsymbol{y})\mathcal{M}_n(\text{d}\boldsymbol{y})}{\int_{\mathcal{X}^{n}}\overline{\widehat{K}}^*_n(\boldsymbol{w})\mathcal{M}_n(\text{d}\boldsymbol{w})}.
\end{multline*}
Using that $\int_{\mathcal{X}^{n}}\overline{\widehat{K}}^*_n(\boldsymbol{w})\mathcal{M}_n(\text{d}\boldsymbol{w}) \geq 1$, $\sum_{i\geq 1}\overline{\widehat{p}}^*_i\geq 1$ and $\sum_{k\geq 1}\mathbbm{1}_{A\geq kK} \leq 0 \vee (A/K)$, we get
\begin{multline*}
\sum_{k\geq 1}\mathbb{P}\left(B_k^K\right) \leq \\
0 \vee\left( \frac{1}{K}\sum_{k\geq 1} \overline{\widehat{p}}^*_k \int_{\mathcal{X}^{\mathbb{N}}}\underset{(x,\Z,t) \in \Zhatset\times\mathbb{R}_+}{\sup}\left(\log\left(\sum_{j=1}^{k}\psi\left(y^j,\nu_{+}\left(x, \boldsymbol{y}\right),s\right)\right)\right)\overline{\widehat{K}}^*_k(\boldsymbol{y})\mathcal{M}_n(\text{d}\boldsymbol{y})\right).
\end{multline*}
We now use the Borel Cantelli lemma -see Theorem 2.3.1 in \cite{Durrett}- with \eqref{eq H-} to ensure that 
$$\limsup_{n \rightarrow\infty}\frac{\underset{(x,\Z,t) \in \Zhatset\times\mathbb{R}_+}{\sup}\left(\log\left(\sum_{j=1}^{N_n}\psi\left(Y^j_n,\nu_{+}\left(x, \boldsymbol{Y}_n\right),s\right)\right)\right)}{n} = 0 \quad \text{a.s.}
$$
Notice that and \eqref{crit technic} ensures that, asymptotically, $\sum_{j=1}^kS_j$ grows linearly with $k$. Thus
$$\sum_{k\geq 1}\exp\left\{\underset{(x,\Z,t) \in \Zhatset\times\mathbb{R}_+}{\sup}\left(\log\left(\sum_{j=1}^{N_k}\psi\left(Y^j_k,\nu_{+}\left(x, \boldsymbol{Y}_k\right),s\right)\right)\right) - c\sum_{j=1}^kS_j\right\}< \infty \quad  \text{a.s.}
$$
The upper bound for the branching rate in Assumption \ref{assumption H} ensures that $\widehat{T}^m \rightarrow \infty$ almost surely as $m$ tends to infinity, and Fatou's lemma gives that $\sup_{t\geq0}\mathbb{E}_{\bar{z}}\left[\mathring{W}_t(\psi)\big\vert\widehat{\mathcal{F}}^*_{\infty}\right] < \infty$. Thus, the quenched submartingale $(\mathring{W}_{t \wedge \widehat{T}^m}(\psi), t\geq 0)$ converges almost surely to a finite random variable and
$$\limsup_{t \to +\infty}\mathring{W}_t(\psi) < +\infty \ \textrm{ a.s.}$$ 
We conclude the proof using \eqref{eq inf} and Lemma \ref{lemma: H finite}.
\end{proof}

\appendix
\section{Proof of Proposition \ref{prop: generator spine} }\label{appendix a}
The spine process $\left(\left(E_t,\Zbarhat_t\right), t\geq 0\right)$ defined by Dynamics \ref{spinal outside rates} and \ref{spine rates} can be rigorously expressed as the solution of a SDE driven by a multivariate point measure. 

We recall that the offspring traits at birth $\boldsymbol{y}=(y^1,\cdots,y^n)$ of an individual of trait $x$ in a population $\Z$ at time $t$ is given by the law $K_n(x,\Z,t,\boldsymbol{y})\mathcal{M}_n(\text{d}\boldsymbol{y})$. 

Let $\widehat{E}= \mathbb{R}_+ \times \mathcal{U} \times \mathbb{R}_+ \times \mathbb{N} \times \mathcal{X}^{\mathbb{N}}$ and $\widehat{E}^*= \mathbb{R}_+ \times \mathbb{R}_+ \times \mathbb{N} \times \mathcal{X}^{\mathbb{N}}$ and let $\widehat{Q}\left(\textrm{d}s,\textrm{d}u,\textrm{d}r,\textrm{d}n,\textrm{d}\boldsymbol{y} \right)$ and $\widehat{Q}^*\left(\textrm{d}s,\textrm{d}r,\textrm{d}n,\textrm{d}\boldsymbol{y}\right)$
be two independent Poisson point measures on $\widehat{E}$ and $\widehat{E}^*$ with respective intensity measures $\textrm{d}s(\sum_{v \in\mathcal{U}}\delta_v(\textrm{d}u))\textrm{d}r\sum_{i\geq1}\delta_{i}(\text{d}n)\mathcal{M}_i(\text{d}\boldsymbol{y})$ outside the spine and $\textrm{d}s\textrm{d}r\sum_{i\geq1}\delta_{i}(\text{d}n)\mathcal{M}_i(\text{d}\boldsymbol{y})$ for the spine. We denote by $\big(\widehat{\mathcal{F}}_t, t\geq 0\big)$ the canonical filtration associated with these Poisson point measures. The $\widehat{\mathcal{F}}_t$-adapted set of labels of the living individuals in the population outside the spine is denoted $\mathring{\mathbb{G}}(t)$. Finally we recall the notation $J(x,\Z,t,\boldsymbol{y})$, introduced in \eqref{spine proba}, for the random index of the new spine after the branching of a spine of trait $x$ in a population $\Z$ at time $t$ to an offspring $\boldsymbol{y}$.

Let $\left(e,\bar{z}\right) \in \Zbarhatset$. 
Under Assumptions \ref{assumption A} and \ref{assumption C}, the process $\left(\left(E_t,\Zbarhat_t\right), t\geq 0\right)$ is the unique $\widehat{\mathcal{F}}_t$-adapted solution, for every function $g \in \mathcal{C}^1\left(\mathcal{U}\times\mathcal{U}\times\mathcal{X},\mathbb{R} \right)$ and $t\geq0$, of the following equation
\begin{align}\label{PPM spinal}
\left\langle\Zbarhat_t,g\left(E_t,\cdot\right)\right\rangle :&= \langle\bar{z},g(e,\cdot)\rangle + \int_0^t \int_{\mathcal{U}\times\mathcal{X}}\frac{\partial g}{\partial x}(E_s,u,x)\cdot \mu\left(x,\Zhat_s,s\right)\Zbarhat_s(\textrm{d}u,\textrm{d}x)\textrm{d}s \nonumber \\
&\quad+\int_{\widehat{E}}
 \left[\sum_{i=1}^ng\left(E_{s\minus},ui,\boldsymbol{y}\right) - g\left(E_{s\minus},u,X^u_{s\minus}\right)\right]  \nonumber\\
& \qquad \quad \times \mathbbm{1}_{\left\{u\in \mathring{\mathbb{G}}_{s\minus}\right\}}  \mathbbm{1}_{\left\{r \leq \widehat{B}_n(Y_{s\minus},X^u_{s\minus},\Zhat_{s\minus},s)\widehat{K}_n(Y_{s\minus},X^u_{s\minus},\Zhat_{s\minus},s,\boldsymbol{y})\right\}}\widehat{Q}\left(\textrm{d}s,\textrm{d}u,\textrm{d}r,\textrm{d}n,\textrm{d}\boldsymbol{y} \right) \nonumber\\
&\quad+\int_{\widehat{E}^*}
 \left[\sum_{i=1}^ng\left(E_{s\minus}J(Y_{s\minus},\Zhat_{s\minus},s,\boldsymbol{y}),E_{s\minus}J(Y_{s\minus},\Zhat_{s\minus},s,\boldsymbol{y}),\boldsymbol{y}\right) - g\left(E_{s\minus},E_{s\minus},Y_{s\minus}\right)\right]  \nonumber\\
& \qquad \qquad \times \mathbbm{1}_{\left\{\ r \leq \widehat{B}^*_k(Y_{s\minus},\Zhat_{s\minus},s)\widehat{K}^*_k(Y_{s\minus},\Zhat_{s\minus},s,\boldsymbol{y})\right\}}\widehat{Q}^*\left(\textrm{d}s,\textrm{d}r,\textrm{d}n,\textrm{d}\boldsymbol{y}\right).
\end{align}
This assertion is shown following the same computations than for proof of Proposition \ref{prop: unique original} and \cite{Marguet19}, using the positivity of the $\psi$ function. Assumption \ref{assumption C} leads to the same bound as  Assumption \ref{assumption A}.3 in the case of a spine process.

Now we establish the expression of the generator of the marginal spine process. We recall the notation $\Z^+(x,\boldsymbol{y}) := \Z + \sum_{i=1}^n\delta_{y^i} - \delta_{x}$, introduced in \eqref{nu+}. 
Taking expectations of \eqref{PPM spinal} for the marginal process on $\Zhatset$, we derive the non-homogeneous infinitesimal operator of $((Y_t,\Zhat_t), t\geq 0)$, following steps in \cite{FM04}. It is given by the operator $\widehat{L}_{\psi}$, defined for every $F \in \mathcal{D}$, introduced in \eqref{def: D space}, and $(x_e,\Z,t) \in \Zhatset\times\mathbb{R}_+$, by
\begin{align*}
\widehat{L}_{\psi}F(x_e,\Z,t) &:= GF\left(x_e,\Z,t\right) \\
&\hspace{-1.9cm}+ \sum_{n\geq 0} \Bigg\{\widehat{B}^*_n(x_e,\Z,t)\int_{\mathcal{X}^n} \left[  \frac{\sum_{i=1}^nF\psi\left(y^i, \Z^+(x_e,\boldsymbol{y}),t \right)}{\sum_{j=1}^n\psi\left(y^j, \Z^+(x_e,\boldsymbol{y}),t \right)} \right. - F(x_e,\Z,t)\Bigg] \widehat{K}^*_n\left(x_e,\Z,t,\boldsymbol{y}\right)\mathcal{M}_n(\text{d}\boldsymbol{y}) \nonumber \\
& \hspace{-0.8cm}+ \int_{\mathcal{X}} \widehat{B}_n(x,\Z,t)\int_{\mathcal{X}^n} \Bigg[ F\left(x_e,\Z^+(x,\boldsymbol{y}),t \right) - F(x_e,\Z,t) \Bigg] \widehat{K}_n\left(x,\Z,t,\boldsymbol{y}\right)\mathcal{M}_n(\text{d}\boldsymbol{y})\Z(\textrm{d}x)\nonumber\\
& \hspace{0.4cm}- \widehat{B}_n(x_e,\Z,t)\int_{\mathcal{X}^n} \Bigg[ F\left(x_e, \Z^+(x_e,\boldsymbol{y})\right) - F(x_e,\Z) \Bigg] \widehat{K}_n\left(x_e,\Z,t,\boldsymbol{y}\right)\mathcal{M}_n(\text{d}\boldsymbol{y})\Bigg\}.
\end{align*}
The first line gives the dynamical evolution between branching events. The second is related to spinal branching events and the choice of a new spinal individual among the offspring population. The last two lines describe the branching events outside the spine, for all individuals but the spinal one.
Using that $G\left[\psi F\right] = FG\psi  + \psi GF$, we get
$$
\frac{\mathcal{G}\left[\psi F\right]}{\psi}\left(\cdot\right) - \frac{\mathcal{G}\psi}{\psi}F\left(\cdot\right)  = GF\left(\cdot\right)  + \sum_{n\geq 0} \frac{B\left(\cdot\right)}{\psi\left(\cdot\right)}\widehat{\mathcal{T}}_n\left(\cdot\right),
$$
where the jump part $\widehat{\mathcal{T}}_n$ is defined for all $\left(x_e,\Z,t\right) \in \Zhatset\times\mathbb{R}_+$ by
\begin{align}\label{eq: jump part}
\widehat{\mathcal{T}}_n&(\rho_e) := p_n(\rho_e)\int_{\mathcal{X}^n} \sum_{i=1}^n \psi F\left(y_i, \Z^+(x_e,\boldsymbol{y}),t \right)  - \psi F(x_e,\Z^+(x_e,\boldsymbol{y}),t) K_n\left(\rho_e,\boldsymbol{y}\right)\mathcal{M}_n(\text{d}\boldsymbol{y}) \nonumber \\
& + \int_{\mathcal{X}}p_n(x,\Z,t)\int_{\mathcal{X}^n} [\psi F\left(x_e,\Z^+(x,\boldsymbol{y}),t \right)- \psi F(\rho_e) ]K_n\left(x,\Z,t,\boldsymbol{y}\right)\mathcal{M}_n(\text{d}\boldsymbol{y})\Z(\textrm{d}x)\nonumber\\
&- p_n(\rho_e)\int_{\mathcal{X}^n} \left[\sum_{i=1}^n \psi\left(y_i, \Z^+(x_e,\boldsymbol{y}),t \right)  - \psi(x_e,\Z^+(x_e,\boldsymbol{y}),t)\right]F(\rho_e) K_n\left(\rho_e,\boldsymbol{y}\right)\mathcal{M}_n(\text{d}\boldsymbol{y}) \nonumber \\
& - \int_{\mathcal{X}}p_n(x,\Z,t)\int_{\mathcal{X}^n} [\psi\left(x_e,\Z^+(x,\boldsymbol{y}),t \right)- \psi(\rho_e) ]F(\rho_e)K_n\left(x,\Z,t,\boldsymbol{y}\right)\mathcal{M}_n(\text{d}\boldsymbol{y})\Z(\textrm{d}x)
\end{align}
where we used the notation $\rho_e := (x_e,\Z,t)$.
Rearranging the terms in \eqref{eq: jump part} we get
\begin{align*}
\widehat{\mathcal{T}}_n(\rho_e)& = p_n(\rho_e)\int_{\mathcal{X}^n} \sum_{i=1}^n \psi\left(y_i, \Z^+(x_e,\boldsymbol{y}) ,t\right) \left[F \left(y_i, \Z^+(x_e,\boldsymbol{y}),t \right) - F(\rho_e)\right] K_n\left(\rho_e,\boldsymbol{y}\right)\mathcal{M}_n(\text{d}\boldsymbol{y}) \nonumber \\
& + \int_{\mathcal{X}}p_n(x,\Z,t)\\
&\hspace{0.65cm}\times\int_{\mathcal{X}^n} [F\left(x_e,\Z^+(x,\boldsymbol{y}),t \right)- F(\rho_e) ]\psi\left(x_e, \Z^+(x,\boldsymbol{y}),t\right)K_n\left(x,\Z,t,\boldsymbol{y}\right)\mathcal{M}_n(\text{d}\boldsymbol{y})\Z(\textrm{d}x)\nonumber\\
&- p_n(\rho_e)\int_{\mathcal{X}^n} [ F\left(x_e, \Z^+(x_e,\boldsymbol{y}),t\right) - F(\rho_e) ]\psi\left(x_e, \Z^+(x_e,\boldsymbol{y}),t\right)K_n\left(\rho_e,\boldsymbol{y}\right)\mathcal{M}_n(\text{d}\boldsymbol{y}).
\end{align*}
We conclude the proof using the branching rates introduced in Dynamics \ref{spinal outside rates} and \ref{spine rates}.

\section{Proof of Lemma \ref{chain rule}}\label{appendix chain rule}
The recursive equality for $\xi_k$ derives from the fundamental theorem of calculus. 
Let $k \geq 1$,
\begin{equation*}
\frac{\xi_k\left(E_k,\left(\widehat{\mathcal{V}}_i, 0 \leq i \leq k\right) \right)}{\xi_{k-1}\left(E_{k-1},\left(\widehat{\mathcal{V}}_i, 0 \leq i \leq k-1\right) \right)} =\frac{\psi\left(Y_{\widehat{T}_{k-1}},\Zhat_{\widehat{T}_{k-1}}\right)}{\psi\left(Y_{\widehat{T}_k},\Zhat_{\widehat{T}_k}\right)} \exp\left(\int_{\widehat{T}_{k-1}}^{\widehat{T}_{k}}\frac{\mathcal{G}\psi\left(Y_s,\Zhat_s\right)}{\psi\left(Y_s,\Zhat_s\right)}\textrm{d}s\right).
\end{equation*}
We notice that $\mathcal{G}$ introduced in \eqref{def: mathcal G}, verify for all $(x_e,\Z,t) \in \Zhatset\times\mathbb{R}_+$ 
\begin{equation*}
\frac{\left[\mathcal{G}-G\right]\psi(x_e,\Z,t)}{\psi(x_e,\Z,t)}= 
\widehat{\tau}_{\textrm{tot}}(x_e,\Z,t) -  \int_{\mathcal{X}}B(x,\Z,t)\Z(\textrm{d}x),
\end{equation*}
where $\widehat{\tau}_{\textrm{tot}}$ is defined in \eqref{tau hat tot}. Then, using derivative chain rule we get that $$G\left(\ln\circ \psi\right)(x_e,\Z,t) = \frac{G\psi(x_e,\Z,t)}{\psi(x_e,\Z,t)}. $$ 
Between successive branching events, the evolution of the process is purely deterministic and we can apply the fundamental theorem of calculus:
$$
\int_{\widehat{T}_{k-1}}^{\widehat{T}_{k}}G\left(\ln\circ \psi\right)\left(Y_s,\Zhat_s,s\right)\textrm{d}s = \ln\left(\psi\left(Y_{\widehat{T}_{k}\minus},\Zhat_{\widehat{T}_{k}\minus},\widehat{T}_{k}\right)\right) -\ln\left(\psi\left(Y_{\widehat{T}_{k-1}},\Zhat_{\widehat{T}_{k-1}},\widehat{T}_{k-1}\right)\right).
$$
We thus have
$$
\exp\left(\int_{\widehat{T}_{k-1}}^{\widehat{T}_{k}}\frac{\mathcal{G}\psi\left(Y_s,\Zhat_s,s\right)}{\psi\left(Y_s,\Zhat_s,s\right)}\textrm{d}s\right) = \frac{\psi\left(Y_{\widehat{T}_{k}\minus},\Zhat_{\widehat{T}_{k}\minus},\widehat{T}_{k}\right)}{\psi\left(Y_{\widehat{T}_{k-1}},\Zhat_{\widehat{T}_{k-1}},\widehat{T}_{k-1}\right)}\exp\left(\int_{\widehat{T}_{k-1}}^{\widehat{T}_{k}}\lambda\left(Y_s,\Zhat_s,s\right)\textrm{d}s\right),
$$ 
that concludes the proof.

\section{Proof of Lemmas \ref{lemma: H finite} and \ref{lemma: H infty} }\label{appendix immigration}

\begin{proof}[Proof of Lemma \ref{lemma: H finite}]
We recall the notations $((\widehat{T}^*_k,\widehat{N}^*_k,(\widehat{X}^{i}_k, 1 \leq i \leq \widehat{N}^*_k)), k\geq0)$ for the sequence of jumps times, number of children and type corresponding to branching events of the spine. The trait of the new spine is denoted $Y_{\widehat{T}^*_k}$. We also recall the sequence of filtrations $\left(\mathcal{F}_k,k\geq 1\right)$, introduced in \eqref{lemma: filtration}.
Using these notations, with the lower bound for $\mathcal{G}\psi/\psi$ in Assumption \ref{assumption H} and the positivity of the function $\psi$, we have for all $k\geq 0$
\begin{equation}\label{eq: log - ct}
\log\left(\psi\left(Y_{\widehat{T}^*_k},\Zhat_{\widehat{T}^*_k},\widehat{T}^*_k \right)\right)-\int_0^{\widehat{T}^*_k}\Lambda\left(Y_s,\Zhat_s\right)\text{d}s \leq \log\left(\sum_{i=1}^{\widehat{N}^*_k}\psi\left(\widehat{X}^{i}_k,\Zhat_{\widehat{T}^*_k},\widehat{T}^*_k \right)\right) -c\widehat{T}^*_k
\end{equation}
as $Y_{T_k} \in \left\{X^i_k, 1\leq i\leq N_k\right\}$.
We introduce for all $K>0$ and all $k\geq 1$, the event
$$
B^K_k := \left\{\log\left(\sum_{i=1}^{\widehat{N}^*_k}\psi\left(\widehat{X}^{i}_k,\Zhat_{\widehat{T}^*_k},\widehat{T}^*_k \right)\right)\geq kK \right\} \in \mathcal{F}_{k}.
$$
First notice that 
\begin{multline*}
\sum_{k \geq 1}\mathbb{P}\left(B^K_k \Big\vert \mathcal{F}_{k-1}\right) =  \sum_{k \geq 1}\sum_{n\geq 1} \widehat{p}^*_n\left(Y_{\widehat{T}^*_k\minus},\Zhat_{\widehat{T}^*_k\minus},{\widehat{T}^*_k\minus}\right)\\
\times \int_{\mathcal{X}^n}\mathbbm{1}_{\left\{\log\left(\sum_{i=1}^{n}\psi\left(y^i,\Zhat_{\widehat{T}^*_k},{\widehat{T}^*_k}\right) \right) \geq kK\right\}}\widehat{K}^*_n\left(Y_{\widehat{T}^*_k\minus},\Zhat_{\widehat{T}^*_k\minus},{\widehat{T}^*_k\minus},\mathbf{y}\right)\mathcal{M}_n\left(\text{d}\mathbf{y}\right).
\end{multline*}
Taking the supremum over $\Zhatset\times\mathbb{R}_+$, we get that
\begin{multline*}
\sum_{k \geq 1}\mathbb{P}\left(B^K_k \Big\vert \mathcal{F}_{k-1}\right) \leq \\ \sum_{n\geq 1} \overline{\widehat{p}}^*_n\int_{\mathcal{X}^n}\sum_{k\geq 1}\mathbbm{1}_{\left\{\underset{(x,\Z,t) \in \Zhatset\times\mathbb{R}_+}{\sup}\left(\log\left(\sum_{i=1}^{n}\psi\left(y^i,\Z^+(x,\mathbf{y}),t\right) \right)\right)\geq kK\right\}}\overline{\widehat{K}}^*_n\left(\mathbf{y}\right)\mathcal{M}_n\left(\text{d}\mathbf{y}\right).
\end{multline*}
We recall the notation $\Z^+(x,\boldsymbol{y}) := \Z + \sum_{i=1}^n\delta_{y^i} - \delta_{x}$, introduced in \eqref{nu+}.
Finally, using that for all $A\in\mathbb{R}$, $\sum_{k\geq1}\mathbbm{1}_{A\geq kK} \leq 0 \vee (A/K)$ we get
\begin{multline*}
\sum_{k \geq 1}\mathbb{P}\left(B^K_k \Big\vert \mathcal{F}_{k-1}\right) \leq \\ 0 \vee \left(\frac{1}{K}\sum_{n\geq 1}\overline{\widehat{p}}^*_n \int_{\mathcal{X}^n} \sup_{(x,\Z,t) \in \Zhatset\times\mathbb{R}_+}\left[\log\left(\sum_{i=1}^n\psi(y^i,\Z^+(x,\mathbf{y}),t) \right)\right]\overline{\widehat{K}}^*_n\left(\mathbf{y}\right)\mathcal{M}_n\left(\text{d}\mathbf{y}\right)\right).
\end{multline*}
Thus \eqref{eq H-} ensures that 
\begin{equation*}
\sum_{k \geq 1}\mathbb{P}\left(B^K_k \Big\vert \mathcal{F}_{k-1}\right) < +\infty
\end{equation*}
and the second Borel-Cantelli Lemma- see Theorem 4.3.4 in \cite{Durrett}- gives that for all $K>0$, eventually
$$
\log\left(\psi\left(Y_{\widehat{T}^*_k},\Zhat_{\widehat{T}^*_k},{\widehat{T}^*_k}\right) \right) \leq kK \quad \text{a.s.}
$$
Furthermore, the bounds on the jump rate in Assumption \ref{assumption H} ensure that $\widehat{T}^*_k$ grows linearly almost surely to infinity as $k$ tends to infinity. Thus, using \eqref{eq: log - ct} we get that almost surely 
$$\limsup_{k\rightarrow\infty}\left[\log\left(\psi\left(Y_{\widehat{T}^*_k},\Zhat_{\widehat{T}^*_k},\widehat{T}^*_k \right)\right)-\int_0^{\widehat{T}^*_k}\Lambda\left(Y_s,\Zhat_s\right)\text{d}s\right] = -\infty.$$ 
The fact that $\lim\limits_{k\rightarrow\infty}\widehat{T}^*_k=+\infty$ almost surely concludes the proof.
\end{proof}

\begin{proof}[Proof of Lemma \ref{lemma: H infty}]
First, notice that, using a log-sum inequality along with the positivity of the function $\psi$ we get, for all $n\geq 2$, $(x,\nu,t) \in \Zhatset\times\mathbb{R}_+$, $\mathbf{y}\in\mathcal{X}^n$ and $j_n\in \{1,\cdots,n\}$
\begin{align*}
\log\left(\sum_{i=1}^n\psi(y^i,\Z^+(x,\mathbf{y}),t)\right) &= \log\left(\sum_{i=1}^n\mathbbm{1}_{i\neq j_n}\psi(y^i,\Z^+(x,\mathbf{y}),t) + \psi(y^{j_n},\Z^+(x,\mathbf{y}),t)\right)\\
&\leq \log(2) + \frac{\psi(y^{j_n},\Z^+(x,\mathbf{y}),t)}{\sum_{i=1}^n\psi(y^i,\Z^+(x,\mathbf{y}),t)}\log\left(\psi(y^{j_n},\Z^+(x,\mathbf{y}),t)\right)\\
+& \frac{\sum_{i=1}^n\mathbbm{1}_{i\neq j_n}\psi(y^i,\Z^+(x,\mathbf{y}),t)}{\sum_{i}\psi(y^i,\Z^+(x,\mathbf{y}),t)}\log\left(\sum_{i=1}^n\mathbbm{1}_{i\neq j_n}\psi(y^i,\Z^+(x,\mathbf{y}),t)\right).
\end{align*}
Thus, taking the infimum in the previous inequality and using \eqref{crit inf infinite}, we have for any arbitrary state $(x_0,\Z_0,t_0)\in \Zhatset\times\mathbb{R}_+$
\begin{multline}\label{eq: inf inter}
\sum_{n\geq 2} \widehat{p}^*_n(\rho_0) \int_{\mathcal{X}^n}\frac{\psi(y^{j_n},\Z_0^{+}(x_0,\mathbf{y}),t_0)}{\sum_{i=1}^n\psi(y^i,\Z_0^{+}(x_0,\mathbf{y}),t_0)}\log\left(\psi(y^{j_n},\Z_0^{+}(x_0,\mathbf{y}),t_0)\right)\widehat{K}^*_n(\rho_0,\mathbf{y})\mathcal{M}_n(\text{d}\mathbf{y}) \\ +\sum_{n\geq 2}\underline{\widehat{p}}^*_n \int_{\mathcal{X}^n}\inf_{(x,\Z,t) \in \Zhatset\times\mathbb{R}_+}\left[\log\left(\sum_{i\neq j_n}\psi(y^i,\Z^+(x,\mathbf{y}),t) \right)\right]\underline{\widehat{K}}^*_n(\mathbf{y})\mathcal{M}_n(\text{d}\mathbf{y}) = +\infty
\end{multline}
where we used the notation $\rho_0 := x_0,\Z_0,t_0$.
We conclude the proof using Assumption \ref{assumption H_inf_technical} to ensure that the first term in \eqref{eq: inf inter} is finite.
\end{proof}

\section{Algorithmic construction}\label{appendix algorithm}
We propose an algorithm that generates trajectories of the $\psi$-spine process introduced in \eqref{function psi}. We denote by $F^{-1}_{\textrm{div}},F^{-1}_{\textrm{loss}}$ and $F^{*-1}_{\textrm{loss}}$ the generalized inverse of the cumulative distribution functions of the random variables $\widehat{\Lambda}, \widehat{\Theta}$ and $\Theta$, defined in \eqref{q hat} and \eqref{p hat}. 
Using the deterministic evolution of the traits and knowing the branching events, it is easy to recover the traits at all time of the individuals. 

We first generate a realization of an homogeneous Poisson point process of intensity $1$ on the interval $[t_1,t_2]$, starting from $n$ roots following classical algorithms \cite{miles1970homogeneous}. It returns a list $T_{\textrm{div}}=[t_1,T_1,T_2, \cdots]$ of increasing times, a list $I_{\textrm{div}} = [i_1, i_2, \cdots]$ containing the numbering of the individuals that branched at these times, a list $L_{\textrm{div}} = [\lambda_1, \lambda_2, \cdots]$ of fractions of mass at birth, and the numbering $E$ of the spinal individual. We call numbering a labeling method of the individuals in the population that does not encodes the whole lineage of the individual.
The numbering choice is arbitrary and we chose to add every new individuals at the end of the hidden list, thus at each division event a parent is chosen uniformly in the population and its child is added at the end of the list. This computation is handled by the function Tree explained in Figure \ref{algo:tree}, where $i_0$ is the initial numbering of the spinal individual.
\begin{figure}[htbp]
\centering
\begin{algorithmic}
\Function{$(T_{\textrm{div}},I_{\textrm{div}},E,L_{\textrm{div}})$ = Tree}{$t_1,t_2,n,i_0,F^{-1}_{\textrm{div}}$}
\State $(T_{\textrm{div}},N) = ([t_1],n)$
\State {\color{gray}\% Step 1; generating the division times}
\While{$T_{\textrm{div}}[\textrm{end}] < t_2$}
	\State Generate $u \sim \textrm{uniform}(0,1)$
	\State append $T_{\textrm{div}}[\textrm{end}]-\ln(u)/N$ to $T_{\textrm{div}}$ \Comment{Branching time}
	\State $N \gets N+1$ \Comment{New population size}
\EndWhile
\State \textbf{if} {$T_{\textrm{div}}[\textrm{end}] > t_2$} \textbf{then} pop $T_{\textrm{div}}$
\State {\color{gray}\% Step 2; generating the fractions at birth and choosing the spinal individual}
\State $(I_{\textrm{div}},E) = ([0,\cdots,0],i_0)$ \Comment{$I_{\textrm{div}}$ of size $\textrm{length}(T_{\textrm{div}})$}
\For{$i\in\{1, \cdots, \textrm{length}(T_{\textrm{div}})-1\}$}
		\State Generate $I \sim \textrm{uniform}(\{1,\cdots,n+i-1\})$ \Comment{Branching individual}
		\State $I_{\textrm{div}}[i] \gets I$
		\State Generate $v,q \sim \textrm{uniform}(0,1)$ 
	\State append $F^{-1}_{\textrm{div}}(u)$ to $L_{\textrm{div}}$ \Comment{Fraction $\lambda$ at birth}
	\State \textbf{if} {$E = i$ and $p > F^{-1}_{\textrm{div}}(u)$} \textbf{then} $E \gets n+i$ \Comment{New spinal individual position}
	\EndFor
\EndFunction
\end{algorithmic}
\caption{Simulation algorithm of the jumps times until time $T$}
\label{algo:tree}
\end{figure}

From the list $I_{\textrm{div}}$, it is possible to retrieve the lineage of all individuals using any chosen labeling method. In our case we used the Ulam-Harris-Neveu notations and the algorithm generating the list $U$ of labels from the list $I$ previously obtained and the initial size of the population $n$ is detailled Figure \ref{algo:labeling}.
\begin{figure}[htbp]
\centering
\begin{algorithmic}
\Function{$U$ = Labels}{$I_{\textrm{div}},n$}
\State $N_{\textrm{div}} = \textrm{length}(I_{\textrm{div}})$
\State $U = [1,\cdots,n,0,\cdots,0]$ \Comment{$U$ is of size $n+N_{\textrm{div}}$}
\For{$i \in  \{1,\cdots,N_{\textrm{div}}\}$}
	\State $(U[I_{\textrm{div}}[i]],U[n+i]]) \gets (U[I_{\textrm{div}}[i]]1,U[I_{\textrm{div}}[i]]2)$ \Comment{$U[I_{\textrm{div}}[i]] \in \mathcal{U}$}
\EndFor
\State $U\gets\textrm{QuickSort}(U)$ \Comment{sorting algorithm in $O((N_{\textrm{div}}+n)\log (N_{\textrm{div}}+n))$}
\EndFunction
\end{algorithmic} 
\caption{Labeling function}
\label{algo:labeling}
\end{figure}
The last line of this algorithm lists the labels of individuals in increasing order, grouping siblings together. The label of the spinal individual is thus $U[E]$.

To compute a trajectory of the spinal process on the interval $[0,T]$, we need to establish the time of jumps and their outcomes. Using the deterministic evolution of the traits and knowing the branching events, it is easy to recover the traits at all time for the individuals. Furthermore, depending on the statistics that one want to evaluate on the system, it might not be necessary to compute the traits of all individuals. For that reason we propose an algorithmic construction of a trajectory of the spinal process, returning the list of jumps times, events and labels of the individuals living at time $T$. This algorithm, presented in Figure \ref{algo:main} also distinguishes the spinal individual in the population by returning the label of the spinal individual in the living ones at time $T$. To simplify we will denote $\textrm{PPP}_{[t_1,t_2]}(c(\cdot))$ the list of times given by a Poisson point process of intensity $c(\cdot)$ on the interval $[t_1,t_2]$, computed using Lewis' thinning algorithm \cite{Lewis79}.
\begin{figure}
\begin{algorithmic}
\Require model parameters: $d,F^{-1}_{\textrm{div}},F^{-1}_{\textrm{loss}},F^{*-1}_{\textrm{loss}},T$,  initial condition: $z=[x^1,\cdots,x^n]$.
\Ensure $(T_{\textrm{div}},I_{\textrm{div}},E,L_{\textrm{div}},T_{\textrm{loss}},T_{\textrm{loss}}^*, I_{\textrm{loss}}, L_{\textrm{loss}},L_{\textrm{loss}}^*)$
\State {\color{gray}\% Initializing the spine}
\State Generate $u \sim $uniform$(0,1)$
\State $i_0 = 1$
\While{$\sum_{i=1}^{i_0} x^i/(\sum_{i=1}^n x^i)<u$}	
\State $i_0 \gets i_0+1$
\EndWhile
\State {\color{gray}\% Generating binary spinal tree on division events}
\State $(T_{\textrm{div}},I_{\textrm{div}},E,L_{\textrm{div}})$ = Tree($0,T,n,i_0,F^{-1}_{\textrm{div}}$) 
\State {\color{gray}\% Generating loss events for the spine and individuals outside the spine}
\State $\left(T_{\textrm{loss}},T_{\textrm{loss}}^*, I_{\textrm{loss}}\right) = \left([\cdot],[\cdot],[\cdot]\right) $
\For{$i \in \{1,\cdots,\textrm{length}(T_{\textrm{div}})\}$}
\State {\color{gray}\% For the individuals outside the spine}
\State Generate $P \sim \textrm{PPP}_{\big[T_{\textrm{div}}[i-1],T_{\textrm{div}}[i]\big]}(K_{\textrm{loss}}(n+i-1)^2d(\cdot))$\Comment{Times of loss events}
	\State append $P$ to $T_{\textrm{loss}}$
	\State Generate $I_u \sim \left(\textrm{uniform}\left(\{1, \cdots, n+i-1\}\right)^{\textrm{length}(T_{\textrm{loss}}[i])}\right)$ \Comment{Distributing loss events}
	\State append $I_u$ to $I_{\textrm{loss}}$
	\State {\color{gray}\% For the spine}	
	\State Generate $P^* \sim \textrm{PPP}_{\big[T_{\textrm{div}}[i-1],T_{\textrm{div}}[i]\big]}((n+i-1)d(\cdot))$ \Comment{Times of loss events}
	\State append $P^*$ to $T^*_{\textrm{loss}}$ 
\EndFor
\State Generate $L_u \sim \left(\textrm{uniform}(0,1)^{\textrm{length}(T_{\textrm{loss}})}\right)$ \Comment{Generating fractions lost}
	\State $L_{\textrm{loss}} = F^{-1}_{\textrm{loss}}(L_u)$
\State Generate $L_u^* \sim \left(\textrm{uniform}(0,1)^{\textrm{length}(T_{\textrm{loss}}^*)}\right)$ \Comment{Generating fractions lost}
\State $L_{\textrm{loss}}^* = F^{*-1}_{\textrm{loss}}(L_u^*)$
\end{algorithmic}
\caption{Simulation algorithm of the jumps times until time $T$}\label{algo:main}
\end{figure}

The output tuple $(T_{\textrm{div}},I_{\textrm{div}},E,L_{\textrm{div}},T_{\textrm{loss}},T_{\textrm{loss}}^*, I_{\textrm{loss}}, L_{\textrm{loss}},L_{\textrm{loss}}^*)$ is the minimal information needed to construct the spinal process introduced in Section \ref{section: Yule}. The list $U$ of labels of the individuals can be computed with the function Labels($I_{\textrm{div}},n$). Notice that a lot of operations can be parallelized in this algorithm, unlike in the classical Lewis' algorithm. Moreover, depending on the statistic that one wants to compute on this process, the algorithm can be further simplified. For example, if one takes interest in the total biomass of the population, the indexes of the individuals that branched are not necessary. 

We recall that our statistical value of interest here is the mean size of an individual picked uniformly at random ion the population. To right change of measure from the spinal process is a function of the number of individuals in the population at every time $s \in [0, T]$ that can be computed from the time of branching events and the initial condition only, as established in Proposition \ref{prop size}. To compute this change of measure we also need the total biomass in the population at every time $s \in [0, T]$, that is performed using the Proposition ref{prop biomass}
\begin{proposition}\label{prop size}
Let $T >0$ and $T_{\text{div}} = \left(T^i_{\text{div}}, 0\leq i \leq N_{\text{div}}+1\right)$ be a sequence of increasing branching times with $T^0_{\text{div}}:=0$ and $T^{N_{\text{div}}+1}_{\text{div}} := T$. The size $\widehat{N}$  of the spinal population starting from $N_0$ individuals is such that 
$$
\int_0^T\widehat{N}_s\text{d}s = (N_0+N_{\text{div}})T-\sum_{i=0}^{N_{\text{div}}}T^i_{\text{div}}
$$
and
$$
\int_0^T\widehat{N}^2_s\text{d}s = \left(N_0+N_{\text{div}}\right)^2T-\left(2N_0-1\right)\sum_{i=1}^{N_{\text{div}}}T^i_{\text{div}}-2\sum_{i=1}^{N_{\text{div}}}iT^i_{\text{div}}.
$$
\end{proposition}
\begin{proof}
The first integral term is directly computed using that 
\begin{equation}\label{sequence 1}
\sum_{i=0}^{N_{\text{div}}}i(T_{\text{div}}^{i+1}-T_{\text{div}}^i) =N_{\text{div}}T- \sum_{i=0}^{N_{\text{div}}} T_{\text{div}}^i.
\end{equation}
The second integral term is computed by noticing that
$$
\int_0^T\widehat{N}^2_s\text{d}s = N_0^2T + 2N_0\sum_{i=0}^{N_{\text{div}}}i\left(T^{i+1}_{\text{div}}-T^i_{\text{div}}\right) + \sum_{i=0}^{N_{\text{div}}}i^2\left(T^{i+1}_{\text{div}}-T^i_{\text{div}}\right)
$$
and that 
\begin{equation}\label{sequence 2}
\sum_{i=0}^{N_{\text{div}}}i^2(T_{\text{div}}^{i+1}-T_{\text{div}}^i) =N_{\text{div}}^2T- \sum_{i=1}^{N{\text{div}}} (2i-1)T_{\text{div}}^i.
\end{equation}
Using \eqref{sequence 1} and \eqref{sequence 2} and rearranging the terms concludes the proof. 
\end{proof}

\begin{proposition}\label{prop biomass} For all $0\leq i\leq N_{\text{div}}$, let $\left(T_{\text{loss}}^{i,j}, 1\leq j\leq N^i_{\text{loss}}\right)$ be the sequence of loss events times in $[T_{\text{div}}^i, T_{\text{div}}^{i+1})$. If there is no loss event in this period of time, $N_{\text{loss}}^i = 0$ and we use the formalism $T_{\text{loss}}^{i,N_{\text{loss}}^i} = T_{\text{loss}}^{i,1} = T_{\text{div}}^{i-1}$. The fraction of lost masses are denoted by $\theta_{i,j}$ and the label of the individual that suffered the loss is denoted $(i,j)$. Then
\begin{multline*}
\int_0^TB_s\text{d}s = \left.\sum_{i=0}^{N_{\text{div}}}\right[\frac{B_{T_{\text{div}}^{i}}}{\mu}\left(e^{\mu \left(T_{\text{div}}^{i+1}-T_{\text{div}}^{i}\right)}-1\right) \\
-\left.\sum_{k=1}^{N^i_{\text{loss}}}\frac{1-\theta_{i,k}}{\mu}X^{(i,k)}_{T_{\text{div}}^{i}} \left(e^{\mu \left(T_{\text{div}}^{i+1}-T_{\text{div}}^{i}\right)}-e^{\mu \left(T_{\text{loss}}^{i,k}-T_{\text{div}}^{i}\right)}\right)\right].
\end{multline*}
\end{proposition}

\begin{proof}
We recall that $T^{N_{\text{div}}+1}_{\text{div}} = T$, and then
$$
\int_0^TB_s\text{d}s = \sum_{i=0}^{N_{\text{div}}}\left[\int_{T_{\text{div}}^{i}}^{T_{\text{loss}}^{i,1}}B_s\text{d}s + \sum_{j=2}^{N_{\text{loss}}^i}\int_{T_{\text{loss}}^{i,j-1}}^{T_{\text{loss}}^{i,j}}B_s\text{d}s + \int_{T_{\text{loss}}^{i,N^i_{\text{loss}}}}^{T_{\text{div}}^{i+1}}B_s\text{d}s\right].
$$
Furthermore, for all $0\leq i\leq N_{\text{div}}$, all $1\leq j\leq N_{\text{loss}}^i-1$ and all time $t\in \left(T_{\text{loss}}^{i,j},T_{\text{loss}}^{i,j+1}\right)$, the total biomass verifies
$$
B_{t}=\left(B_{T_{\text{div}}^{i}}-\sum_{k=1}^j(1-\theta_{i,k})X^{(i,k)}_{T_{\text{div}}^{i}}\right)\exp\left(\mu(t-T_{\text{div}}^{i}) \right).
$$
Thus we integrate on those intervals to get
\begin{align*}
\int_0^TB_s\text{d}s &= \left.\sum_{i=0}^{N_{\text{div}}}\right[\frac{B_{T_{\text{div}}^{i}}}{\mu}e^{-\mu  T_{\text{div}}^{i}} \left(e^{\mu T_{\text{loss}}^{i,1}}-e^{-\mu T_{\text{div}}^{i}}\right) \\
&  \qquad \qquad +\sum_{j=2}^{N_{\text{loss}}^i}
\left(\frac{B_{T_{\text{div}}^{i}}}{\mu} - \sum_{k=1}^{j-1}\frac{1-\theta_{i,k}}{\mu}X^{(i,k)}_{T_{\text{div}}^{i}} \right)e^{-\mu  T_{\text{div}}^{i}} \left(e^{\mu T_{\text{loss}}^{i,j}}-e^{\mu T_{\text{loss}}^{i,j-1}}\right)
\\
&\qquad \qquad +\left.\left(\frac{B_{T_{\text{div}}^{i}}}{\mu} - \sum_{k=1}^{N^i_{\text{loss}}}\frac{1-\theta_{i,k}}{\mu}X^{(i,k)}_{T_{\text{div}}^{i}} \right)e^{-\mu  T_{\text{div}}^{i}} \left(e^{\mu T_{\text{div}}^{i+1}}-e^{\mu T_{\text{loss}}^{i,N^i_{\text{loss}}}}\right)\right].
\end{align*}
Notice that for all $0\leq i\leq N_{\text{div}}$ the sum over $j$ is telescoping for the term in $B_{T_{\text{div}}^{i}}$ and thus
\begin{align*}
\int_0^TB_s\text{d}s &= \left.\sum_{i=0}^{N_{\text{div}}}\right[\frac{B_{T_{\text{div}}^{i}}}{\mu}e^{-\mu  T_{\text{div}}^{i}} \left(e^{\mu T_{\text{div}}^{i+1}}-e^{\mu T_{\text{div}}^{i}}\right) \\
&  \qquad \qquad -\sum_{j=2}^{N_{\text{loss}}^i}
\sum_{k=1}^{j-1}\frac{1-\theta_{i,k}}{\mu}X^{(i,k)}_{T_{\text{div}}^{i}} e^{-\mu  T_{\text{div}}^{i}} \left(e^{\mu T_{\text{loss}}^{i,j}}-e^{\mu T_{\text{loss}}^{i,j-1}}\right)
\\
&\qquad \qquad -\left.\sum_{k=1}^{N^i_{\text{loss}}}\frac{1-\theta_{i,k}}{\mu}X^{(i,k)}_{T_{\text{div}}^{i}}e^{-\mu  T_{\text{div}}^{i}} \left(e^{\mu T_{\text{div}}^{i+1}}-e^{-\mu T_{\text{loss}}^{i,N^i_{\text{loss}}}}\right)\right].
\end{align*}
Inverting the sums over $j$ and $k$ in the second line and using the fact that the sum over $j$ becomes telescoping we get
\begin{align*}
\int_0^TB_s\text{d}s &= \left.\sum_{i=0}^{N_{\text{div}}}\right[\frac{B_{T_{\text{div}}^{i}}}{\mu}e^{-\mu  T_{\text{div}}^{i}} \left(e^{\mu T_{\text{div}}^{i+1}}-e^{\mu T_{\text{div}}^{i}}\right) \\
&  \qquad \quad -\sum_{k=1}^{N_{\text{loss}}^i-1}\frac{1-\theta_{i,k}}{\mu}X^{(i,k)}_{T_{\text{div}}^{i}} e^{-\mu  T_{\text{div}}^{i}} 
\left(e^{\mu T_{\text{loss}}^{i,N^i_{\text{loss}}}}-e^{\mu T_{\text{loss}}^{i,k}}\right)
\\
&\qquad \quad -\left.\sum_{k=1}^{N^i_{\text{loss}}}\frac{1-\theta_{i,k}}{\mu}X^{(i,k)}_{T_{\text{div}}^{i}}e^{-\mu  T_{\text{div}}^{i}} \left(e^{\mu T_{\text{div}}^{i+1}}-e^{\mu T_{\text{loss}}^{i,N^i_{\text{loss}}}}\right)\right].
\end{align*}
We conclude the proof using the distributive property of the sum and rearranging the terms.
\end{proof}
Notice that in order to compute $B_{T_{\text{div}}^{i}}$, we need to know the past loss events of each individuals. Thus two approaches are possible: forward or backward. The backward method avoid the storage of the traits of each individuals but is computationally heavy as many calculations are performed multiple times. In fact at every branching event, the past events of the common ancestor are integrated twice. The forward method, on the other hand, avoid these multiplicities but need to store at every branching event the traits of every individual. According to our available RAM, we chose the forward method. This method is also more convenient as the computation of $\Pi_T$ is then straightforward.





\textbf{Acknowledgments.}\\
I would like to thank L. Coquille, A. Marguet and C. Smadi for their useful advises and comments during this work. I thank V. Bansaye for stimulating discussions on the subject of this paper and S. Billiard for fruitful discussions on the Yule model. \\
This work is supported by the French National Research Agency in the framework of the "France 2030" program (ANR-15-IDEX-0002) and by the LabEx PERSYVAL-Lab (ANR-11-LABX-0025-01). I acknowledge partial support by the Chair "Modélisation Mathématique et Biodiversité" of VEOLIA-Ecole Polytechnique-MNHN-F.X.

\end{document}